\documentclass[a4paper, 11pt]{amsart}

\usepackage[a4paper, scale=0.8,centering]{geometry}
\usepackage{amsmath,amssymb,amsthm,graphicx,bbm,pdfpages}
\usepackage{color,mathrsfs,tikz,hyperref}
\usepackage[font={footnotesize}]{caption}

\usepackage[T1]{fontenc}

\usepackage{enumitem}
\setitemize{itemsep=1pt,topsep=2pt,parsep=1pt,partopsep=1pt}

\newcommand{\bbE}{{\ensuremath{\mathbbm E}} }

\newcommand{\bbN}{{\ensuremath{\mathbbm N}} }

\newcommand{\bbP}{{\ensuremath{\mathbbm P}} }

\newcommand{\bbR}{{\ensuremath{\mathbbm R}} }

\newcommand{\bbV}{{\ensuremath{\mathbbm V}} }

\newcommand{\bbZ}{{\ensuremath{\mathbbm Z}} }

\newcommand{\cA}{{\ensuremath{\mathcal A}} }

\newcommand{\cD}{{\ensuremath{\mathcal D}} }
\newcommand{\cE}{{\ensuremath{\mathcal E}} }

\newcommand{\cG}{{\ensuremath{\mathcal G}} }
\newcommand{\cH}{{\ensuremath{\mathcal H}} }

\newcommand{\cJ}{{\ensuremath{\mathcal J}} }

\newcommand{\cN}{{\ensuremath{\mathcal N}} }

\newcommand{\cP}{{\ensuremath{\mathcal P}} }

\newcommand{\cR}{{\ensuremath{\mathcal R}} }

\newcommand{\cT}{{\ensuremath{\mathcal T}} }

\newcommand{\cW}{{\ensuremath{\mathcal W}} }
\newcommand{\cX}{{\ensuremath{\mathcal X}} }

\newcommand{\cZ}{{\ensuremath{\mathcal Z}} }

\newcommand{\sL}{\mathscr{L}}

\newcommand{\sP}{\mathscr{P}}

\newcommand{\bE}{{\ensuremath{\mathbf E}} }

\newcommand{\bM}{{\ensuremath{\mathbf M}} }

\newcommand{\bP}{{\ensuremath{\mathbf P}} }

\newcommand{\bY}{{\ensuremath{\mathbf Y}} }

\newcommand{\ga}{\alpha}
\newcommand{\gb}{\beta}
\newcommand{\gd}{\delta}
\newcommand{\gep}{\varepsilon}       

\newcommand{\go}{\omega}

\newcommand{\gO}{\Omega}

\newcommand{\gU}{\Upsilon}

\renewcommand{\epsilon}{\varepsilon}
\newcommand{\widebar}{\overline}

\newcommand{\brP}{\widebar{\bP}}

\newcommand{\ent}{\mathrm{Ent}}
\newcommand{\hatent}{\hat{\mathrm{E}}\mathrm{nt}}

\newcommand{\R}{\mathbb{R}}
\newcommand{\Z}{\mathbb{Z}}
\newcommand{\N}{\mathbb{N}}

\newcommand{\sumtwo}[2]{\sum_{\substack{#1 \\ #2}}}

\newcommand{\inftwo}[2]{\inf_{\substack{#1 \\ #2}}}
\renewcommand{\tilde}{\widetilde}

\newcommand{\ind}{\mathbbm{1}}
\newcommand{\dd}{{\ensuremath{\mathrm d}} }

\definecolor{rosso}{RGB}{206,43,55}
\definecolor{blu}{RGB}{140,204,171}
\definecolor{verde}{RGB}{0,146,70}
\definecolor{arancione}{RGB}{255,102,51}
\definecolor{viola}{RGB}{255,0,255}

\renewcommand{\hat}{\widehat}

\numberwithin{equation}{section}

\newtheorem{theorem}{Theorem}[section]
\newtheorem{definition}[theorem]{Definition}
\newtheorem{lemma}[theorem]{Lemma}

\newtheorem{corollary}[theorem]{Corollary}
\newtheorem{proposition}[theorem]{Proposition}
\newtheorem{remark}{Remark}[section]

\title[Non-directed polymers in dimension $d\geq 2$]{Non-directed polymers in heavy-tail random environment\\
in dimension $d\geq 2$}

\author[Q. Berger]{Quentin Berger}
\address{Sorbonne Universit\'e, LPSM,
Campus Pierre et Marie Curie, case 158,
4 place Jussieu, 75252 Paris Cedex 5, France}
\email{quentin.berger@sorbonne-universite.fr} 

\author[N. Torri]{Niccol\`o Torri}
\address{Universit\'e Paris-Nanterre, 200 Av. de la R\'epublique, 92001, Nanterre, France and FP2M, CNRS FR 2036}
\email{niccolo.torri@parisnanterre.fr}

\author[R. Wei]{Ran Wei}
\address{Department of Mathematics, Nanjing University, 22 Hankou Road, 210093 Nanjing, China}
\email{weiran@nju.edu.cn}

\thanks{Q. Berger, N. Torri and R. Wei are supported by a public grant overseen by the French National Research Agency, ANR SWiWS (ANR-17-CE40-0032-02).
N. Torri was also supported by the Labex MME-DII}
\date{}

\begin{document}

\begin{abstract}
In this article we study a \emph{non-directed} polymer model in dimension $d\ge 2$:
we consider a simple symmetric random walk on $\mathbb{Z}^d$ which interacts with a random environment,
represented by i.i.d.\ random variables $(\omega_x)_{x\in \mathbb{Z}^d}$.
The model consists in modifying
the law of the random walk up to time (or length) $N$ by the exponential of
 $\sum_{x\in \mathcal{R}_N}\beta (\omega_x-h)$ 
where $\mathcal{R}_N$ is the range of the walk, \textit{i.e.}\ the set of  visited sites up to time $N$, and $\beta\geq 0,\, h\in \mathbb{R}$ are two parameters.
We study the behavior of the model in a weak-coupling regime, that is taking $\beta:=\beta_N$ vanishing as the length $N$ goes to infinity,
and in the case where the random variables $\omega$ have a heavy tail with exponent $\alpha  \in (0,d)$.
We are able to obtain precisely the behavior of polymer trajectories under all possible weak-coupling regimes $\beta_N  = \hat \beta N^{-\gamma}$ with $\gamma \geq 0$:
we find the correct transversal fluctuation exponent
$\xi$ for the polymer (it depends on $\alpha$ and $\gamma$)
and we give the limiting distribution of the rescaled log-partition function. This extends existing works
to the non-directed case and to higher dimensions.
\\[0.1cm]
\textit{2010 Mathematics Subject Classification}:
		82D60, 60K37, 60G70
	\\[0.05cm]
\textit{Keywords}:
	Random polymer, Random walk, Range, Heavy-tail distributions, Weak-coupling limit, Super-diffusivity, Sub-diffusivity\end{abstract}

\maketitle

\section{Introduction}

We consider in this paper a \emph{non-directed polymer model}  in which a simple symmetric random walk on $\mathbb{Z}^d$ (with $d\geq2$) interacts with a random environment. 
This model is closely related to the celebrated \textit{directed polymer model}, for which we refer to \cite{C17} for a complete overview.
The main difference here is that
the random walk is \emph{not} directed: it can come back to a site that has
already been visited, but the interaction is however counted only once (in the spirit of the excited random walk, see~\cite{BW03}).
The model can also be viewed as a randomly perturbed version of random walks penalized or rewarded by their ranges, depending on the positivity or negativity of the external field.
The non-directed polymer model was first introduced by \cite{H19} (with exponential moment conditions on the environment)
and its study was followed in~\cite{BHWT20} with a detailed analysis in dimension $d=1$.
Here, we investigate the case of a heavy-tail random environment,
in analogy with \cite{BT19a,DZ16,HM07} for the directed polymer model --- our results can be viewed as a extension of
\cite{BT19a,DZ16,HM07} to a non-directed setting 
(and to higher dimensions).
%

\subsection{The non-directed polymer model}

Let $S=(S_{n})_{n\geq0}$ be a simple symmetric random walk on~$\bbZ^{d}$ with $d\geq 1$. We denote the probability and expectation with respect to $S$ by $\bP$ and $\bE$ respectively. 
Let $\omega=(\omega_{x})_{x\in\bbZ^{d}}$ be a field of i.i.d.\ random variables (the environment), independent of $S$. The probability and expectation with respect to $\omega$ are denoted by $\bbP$ and $\bbE$ respectively. 

The \textit{non-directed polymer model} up to time $N$ is defined via the following Gibbs' transform: for a given realization $\go$
of the random environment and for $\beta\geq 0$ (the inverse temperature) and $h\in \bbR$ (an external field), let
\begin{equation}\label{themodel}
\frac{\dd \bP_{N,\beta}^{\omega,h}}{\dd \bP}(S):=\frac{1}{Z_{N,\beta}^{\omega,h}}\exp\bigg(\sum\limits_{x\in\cR_{N}}\beta( \omega_{x}-h)\bigg),
\end{equation}
where $\cR_{N}:=\{S_{0},\dots,S_{N}\}$ is the range of the random walk up to time $N$.
The partition function $Z_{N,\gb}^{\go,h}$ is the constant which makes $\bP_{N,\beta}^{\omega,h}$ a probability measure and is equal to
\begin{equation}\label{def:partition}
Z_{N,\beta}^{\omega,h}:=\bE\Big[\exp\Big(\sum\limits_{x\in\cR_{N}}\beta(\omega_{x}-h)\Big)\Big]\,.
\end{equation}

Note that if $\omega\equiv0$ (\textit{i.e.}\ when there is no disorder) and $\beta h\equiv c\in(0,+\infty)$, then \eqref{themodel} becomes the model of a random walk penalized by its range, defined by 
$\frac{\dd\bP_{N,c}}{\dd \bP}(S):=\frac{1}{Z_{N,c}} e^{-c|\cR_N|}$.
That model has been well-studied (starting with the seminal work~\cite{DV79}) and is now well-understood: with high $\bP_{N,c}$-probability, $\cR_N$ is close to a $d$-dimensional ball with an explicit radius $\rho_{d,c}N^{1/(d+2)}$ without holes (see \cite{Bolt94, BC18,DFSX18}).
	
When $\omega$ is non-trivial, the model \eqref{themodel} describes a self-attracting (if $h>0$) or self-repulsing (if $h<0$) polymer interacting with a random environment. At each site, the polymer chain interacts with the disorder exactly once, which may model 
a \emph{screened} interaction, one monomer ``absorbing'' all the interaction at a specific site.
We have chosen to stick to this setting in order to pursue the study initiated in~\cite{BHWT20,H19}, which considers
this model as a disordered version of a random walk penalized by its range. 

\begin{remark}
One could study a polymer  interacting repeatedly with the disorder, that is
\[
\frac{\dd\tilde \bP_{N,\gb}^{\go, h}}{\dd \bP}(S) :=\frac{1}{\tilde Z_{N,\gb}^{\go, h}} \exp\Big( \sum_{i=1}^N \gb (\go_{S_i} - h) \Big) 
\]
(note that the term $-\gb h N$ in the Hamiltonian is a constant and thus does not change the measure).
In that case, large values of~$\go_x$ will play an overwhelming role in the measure, since their effect can be accumulated by returning repeatedly to a given site. One can therefore expect a very strong localization phenomenon, the polymer staying on one site with the largest possible $\go_x$, with high probability.
In the literature,
this model is referred to as the parabolic Anderson model
and the one-site localization for heavy-tailed environment has been proven in~\cite{CCP12},
see~\cite{KLMS09} for a continuous-time version of the result.
\end{remark}



The main goal is then to understand the shape of typical polymer trajectories
$(S_i)_{0 \leq i \leq N}$ under the measure $\bP_{N,\beta}^{\go,h}$, as $N\to\infty$.
One of the main question is to know if and how the presence of the environment perturbs the structure of the typical trajectories of the walk. 
In particular, one is interested in describing
the  end-to-end (or wandering) exponent $\xi$
and the  fluctuation (or volume) exponent $\chi$, that are loosely speaking defined as:
\[
\bE_{N,\beta}^{\go,h} ( \| S_N \|^2)\approx N^{2\xi}\,,  \qquad
\mathbb{V}\mathrm{ar} ( \log Z_{N,\gb}^{\omega,h} )\approx N^{2\chi} \, .
\]

\smallskip
For comparison, let us stress that the random walk $(S_n)_{n \geq 0}$ in dimension $d$, is replaced in the directed polymer model by a directed random walk $(n,S_n)_{n\geq 0}$ in dimension $1+d$.
In particular, the parameter $h$ does not have any influence on the polymer measure, since the number of visited sites is deterministic (and no site is visited twice).
An important feature of the directed polymer model is that a phase transition is known to occur:  when $\go_x$ have an exponential moment, if $\beta$ is (striclty) smaller than some critical value $\beta_c$, then trajectories are diffusive
($\xi=1/2$, $\chi=0$), whereas if $\beta$ is (stricly) larger than 
$\beta_c$ trajectories exhibit some localization properties and are conjectured to have a super-diffusive behavior (at least in low dimensions) ---~for instance, in dimension $1+1$ it is conjectured that $\xi=2/3$ and $\chi =1/3$.
The critical value $\beta_c$ is known to be $\beta_c =0$ in dimension $1+1$ and $1+2$ (hence there is no phase transition)
and $\beta_c >0$ in dimension $1+d$ with $d\geq 3$.
We refer to \cite{C17} and references therein for more details.

\smallskip

 In the non-directed case, however, the parameter $h$ plays an important role,
and random walk trajectories with a large range are rewarded or penalized (depending on whether $h$ is positive or negative).
In analogy with the directed polymer model, the presence
of a random environment should still have a stretching
effect.
The competition between the folding effect of range penalties
and the stretching effect of a random environment has been  recently
 investigated in detail in the case of dimension $d=1$, \cite{BHWT20}
 (in particular, it has been found that $\xi=2/3$ when $h=0$).
But such a study appears difficult in higher dimension,
because the range of the simple random walk then
has a complex geometry.
However, in the case of a heavy-tail random environment,
the localization features of the model become more salient,
since a few sites in the environment will have a much
higher value than the others: we
are indeed able to describe quite precisely the behavior
of the non-directed polymers in that case.
This generalizes the study in \cite{BT19a,DZ16} to the case of non-directed polymers and to higher dimensions.
Let us stress that as a by-product of our results we have that in the heavy-tail setting there is no phase transition: $\beta_c=0$ in any dimension.

\subsubsection*{Weak-coupling regime.}

A recent approach for studying disordered system, initiated in~\cite{AKQ14}, and
which has been developed extensively over the past few years,
is to consider weak-coupling regimes, see e.g.\ \cite{BL20,CSZ15,CSZ13,DZ16,H19}.
The idea is to take the inverse temperature $\beta_N$ vanishing as $N$ goes to infinity.
In the papers cited above, the goal is to find the appropriate
scaling for $\beta_N$ in order for the partition function
to converge to a non-trivial random variable.
This regime is called of \emph{intermediate disorder}: loosely speaking, 
it corresponds to a regime in which disorder just \textit{kicks in},
in the sense that it is still felt in the limit, but is not strong enough
to make the (averaged) disordered measure singular with respect to the reference measure.

We will assume that there is some $\gamma\geq 0$ and some 
$\gb \in (0,\infty)$ such that 
\begin{equation}
\label{def:betaN}
 \lim_{N\to\infty} N^{\gamma} \gb_N = \gb \, , \qquad \text{ as } N\to +\infty \, ,
\end{equation}
and we will write $\gb_N \sim \gb N^{-\gamma}$; we may also consider the case $\gb=0$ or $\gb=+\infty$.
We think of $\gamma$ as a parameter one can play with, which tunes the speed at which $\gb_N$ goes to zero.
We could actually work with general sequences $(\gb_N)_{N\geq 1}$
(in particular, the important relation is~\eqref{eq:xi} below),
but we stick to the pure power case \eqref{def:betaN} in order to clarify the exposition ---~it will capture all
 the essential features of the model, see Remark \ref{rem:generalseq}.
 
Let us also stress that we will consider $h$ fixed, seen as a centering term for the disorder~$\omega_x$.


%

\subsubsection*{Heavy-tail environment.}

Our main assumption in this paper is that $\omega=(\go_x)_{x\in \bbZ^d}$ are i.i.d.\ random variables,
with a pure power tail behavior.
We assume that there exists some $\alpha>0$ such that
\begin{equation}\label{eq:tailOmega}
\bbP(\omega_{0}>t)\sim t^{-\alpha} \, , \qquad \text{ as } t\to+\infty \,.
\end{equation}
Also here,  we could consider a more general asymptotic behavior in~\eqref{eq:tailOmega},
for instance replacing the pure power $t^{-\alpha}$ by  $L(t)t^{-\alpha}$ with $L(\cdot)$ a slowly varying function.
However, we stick to the pure power case as in \eqref{eq:tailOmega}
for the sake of clarity ---~again, it will capture all the essential features of the model.

For simplicity, we also assume that $\omega\geq 0$, but this does not hide anything deep.

\subsection{Heuristics for the phase diagram}
\label{prediction}

Let us present a Flory argument to guess the wandering and fluctuation exponents.
The idea is to find the correct transversal fluctuations $r_N$ (with $N^{1/2}\leq r_N \leq N$) such that 
the entropic cost for the random walk to stretch to a distance $r_N$ 
is balanced by the possible energetic gain from 
high weights $\go_x$ contained in a ball of radius $r_N$.
At the exponential level:
\begin{itemize}
\item the entropic cost is of order $r_N^2/N$ ;
\item the energetic gain is of order $ \beta_N \,  (r_N^d)^{1/\alpha}$, where $(r_N^d)^{1/\alpha}$ is the order of the
maximal weight~$\go_x$ in a ball of radius $r_N$.
\end{itemize}
Hence,  the energy-entropy balance leads to the relation
\begin{equation}\label{eq:xi}
\beta_{N} (r_N)^{d/\ga} \sim \frac{r_N^{2}}{N} \, \qquad ( N^{1/2} \leq r_N \leq N).
\end{equation}
In the case we are interested in, that is if $\gb_N \sim N^{-\gamma}$ with $\gamma\geq 0$, this gives $r_N \sim N^{\xi}$ with $\xi$ verifying
\begin{equation}\label{eq:xigamma}
\frac{\xi d}{\alpha} -\gamma = 2\xi-1 \qquad \Longleftrightarrow
 \qquad \xi=\frac{\alpha(1-\gamma)}{2\alpha-d} \,,
\end{equation}
provided that $\xi \in [\frac12,1]$, that is $ \frac{d-\ga}{\ga}  \leq \gamma \leq \frac{d}{2\ga}$ (note that this range of parameter $\gamma$ is non-empty only  if $\alpha > d/2$).
However, this picture should fail when $\alpha$ is large (in particular when $\ga>d$), because then the strategy of visiting mostly high-energy sites
should be outperformed by a \emph{collective} optimization
(that is still poorly understood, especially in dimension $d\ge 3$). We refer to \cite{BBP07} for a discussion on that matter.
We therefore focus our attention on the case $\alpha<d$. 

When $ 0\leq \gamma < \frac{d-\ga}{\ga}$, then since the transversal fluctuations cannot exceed $N$, the energetic gain should always overcome the entropic cost: we should have $\xi=1$.
On the other hand, when $\gamma > \frac{d}{2\ga}$ for $\alpha>d/2$ or $\gamma >\frac{d-\ga}{\ga}$ for $\alpha<d/2$, then the energetic gain can never
overcome the entropic cost of reaching distance much larger than $N^{1/2}$: we should have $\xi=1/2$.

\smallskip
To summarize, one should have the following three regions when $\ga<d$. 
\begin{enumerate}
\item[\textbf{A}.] If $\alpha\in (0,d)$ and $\gamma < \frac{d-\ga}{\ga}$, then  $\xi=1$.
\item[\textbf{B}.] If $\alpha \in (\frac{d}{2}, d)$ and
 $\gamma\in(\frac{d-\alpha}{\alpha},\frac{d}{2\alpha})$, then 
 $\xi  = \frac{\alpha(1-\gamma)}{2\alpha-d}  \in (\frac12,1)$. 
\item[\textbf{C}.] If $\alpha\in (0,d)$ and $\gamma> \max\{ \frac{d}{2\ga} , \frac{d-\ga}{\ga} \}$, then $\xi=\frac12$.
\end{enumerate}
We refer to Figure~\ref{fig1} below for a graphical representation of these regions, and to \cite[Fig.~1]{BT19a} for a comparison with the corresponding picture in the directed polymer in dimension $1+1$ (where conjectures for larger values of $\alpha$ have the merit to exist).
Our main results 
consist in establishing this phase diagram.

\begin{figure}[ht]
\vspace{-0.5\baselineskip}
\captionsetup{width=0.93\textwidth}
\centering
\includegraphics[scale=0.37]{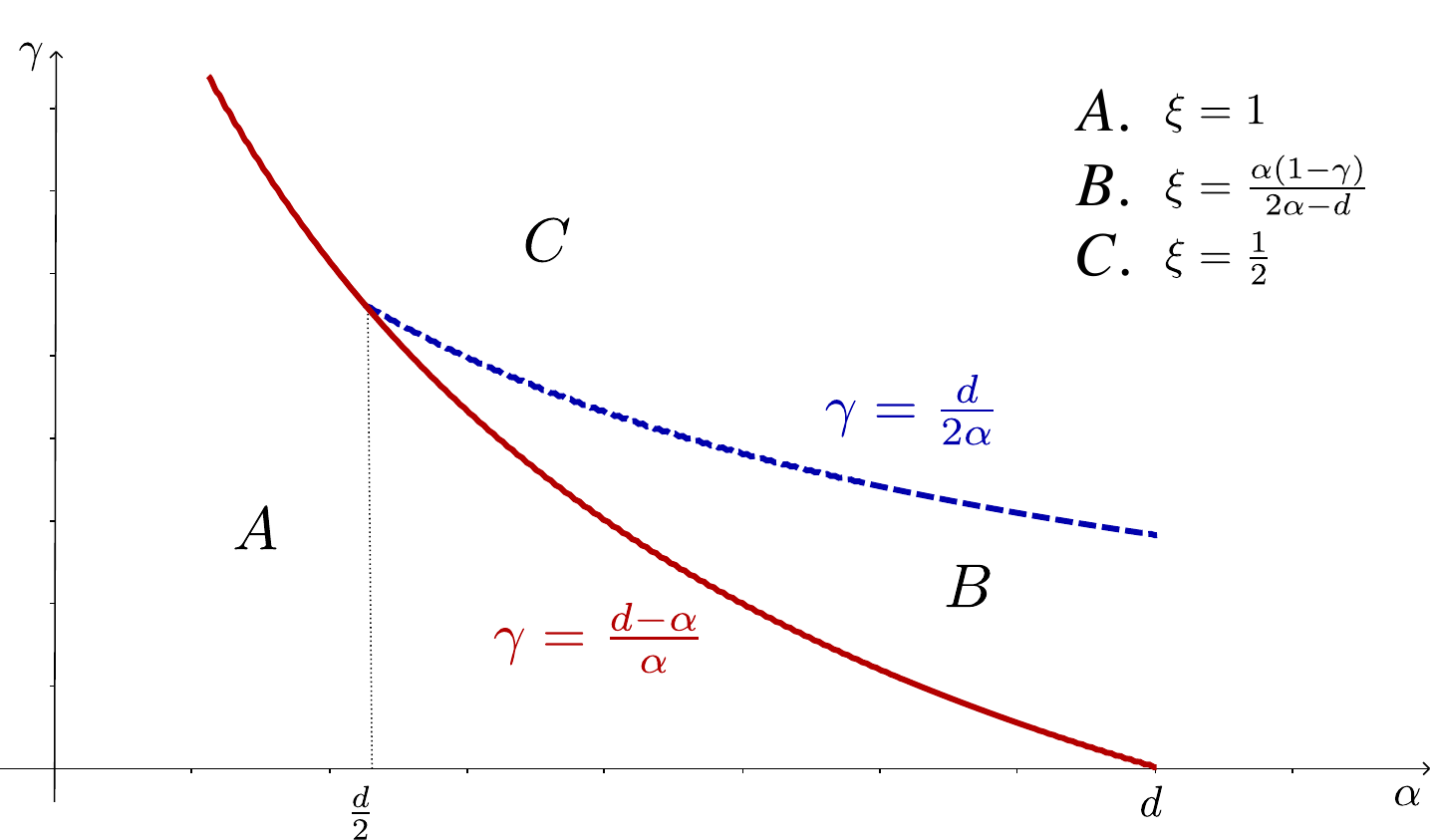}
\caption{Phase diagram for the wandering exponent $\xi$ depending on the parameters $\alpha, \gamma$, with $\alpha\in (0,d)$.
We have identified three regions. In region \textbf{A}, we have $\xi=1$. In region~\textbf{C}, we have $\xi =\frac12$.
In region~\textbf{B}, we have $\xi= \frac{\alpha(1-\gamma)}{2\alpha-d} \in (\frac12,1)$.
Note that when $\alpha \in (\frac{d}{2},d)$, the wandering
exponent $\xi$ in region~\textbf{B} interpolates between regions~\textbf{A}
and~\textbf{C}: all values  $\xi\in (\frac12,1)$ can be attained.
On the other hand,  region~\textbf{B} does not exist
when $\alpha\in(0,\frac{d}{2})$:
the wandering exponent $\xi$
drops abruptly from $1$ to $\frac12$ and no intermediate transversal fluctuations are possible. 
}\label{fig1}
\vspace{-0.5\baselineskip}
\end{figure}


\subsection{Definition of quantities arising in the scaling limit}
Our results additionally provide the scaling limit of the 
log-partition function in all weak-coupling regimes
when $\alpha\in (0,d)$.
In order to describe the limit, let us introduce the continuum counterparts of the random environment and of random walk trajectories, and define the corresponding notion of energy and entropy.

\smallskip
The continuum environment arises as the extremal field of $(\go_x)_{x\in \bbZ^d}$:
we let $\cP:=\{(x_{i},w_i)\}$ be a Poisson point process on $\bbR^{d} \times (0,+\infty)$ of intensity $\eta(\dd x, \dd w)=\alpha w^{-(1+\alpha)}\ind_{\{w>0\}} \dd x \,\dd w$. 
With a slight abuse of notation (it will not draw any confusion), we denote its law by $\bbP$.
We also let
\begin{equation}
\cD:=\big\{s:[0,1]\to\bbR^{d}: s(0)=0, s~\mbox{continuous and a.e. differentiable} \big\} \, ,
\end{equation}
which represents the set of allowed paths (corresponding to scaling limits of random walk trajectories).
Then, for a path $s\in \cD$,  we define its energy
\begin{equation}
\label{def:energy}
\pi(s) = \pi^{\cP}(s) := \sum_{(x,w) \in \cP} w \ind_{\{ x \in s([0,1])  \}} \, .
\end{equation}

\smallskip
To define the entropic cost, 
let us observe that we have the following large deviation principles for the simple random walk: for $u\in (0,1]$ and $x\in \bbR^d$ (we omit integer parts to lighten notation), 
cf.\ Lemma~\ref{lem:rate} in Appendix
\[
\lim_{N\to\infty} -\frac{1}{N^{2\xi-1}} \log\bP( S_{uN} = x N^{\xi} )
 = \begin{cases}
 u \, \mathrm{J}_d( \frac{x}{u} ) & \quad  \text{ if } \xi=1 \, ,\\
 \frac d2 \frac{\|x\|^2}{u} &\quad  \text{ if } \xi \in (\frac12,1) \,,
 \end{cases}
\]
where $\mathrm{J}_d(\cdot)$ is a given rate function; we stress that $\mathrm{J}(x) < +\infty$ if $\|x\|_1 \leq 1$ and $\mathrm{J}(x)=+\infty$ if $\|x\|_1 >1$, with $\|x\|_1$ the $\ell^1$ norm of $x$. 

We then define two different entropic cost functions for a path
$s\in \cD$ ---~or rather for the image $s([0,1])$--- depending on whether the path comes from the scaling of a random walk trajectory with wandering exponent $\xi=1$ or $\xi \in (\frac12,1)$.
In the case $\xi=1$, we define
\begin{equation}
\label{def:entropyCont2}
\hatent (s) := \inf_{\varphi\in\Phi} \int_0^1 \mathrm{J}_d \big(  (s\circ \varphi )' (t) \big) \dd t\, , 
\end{equation}
where the infimum is over
\begin{equation}\label{phi}
\Phi:=\{\varphi:[0,1]\overset{\text{onto}}{\longrightarrow}[0,1]: \varphi~\text{is non-decreasing and a.e. differentiable}\}.
\end{equation}
In the case $\xi \in (\frac12,1)$, we define analogously
\begin{equation}\label{def:entropyCont}
\ent(s)= \inf_{\varphi\in\Phi} \int_0^1  \frac{d}{2}\|  (s\circ \varphi )' (t) \|^2 \dd t
 =\frac{d}{2} \Big( \int_0^1 \| s'(t)\| \dd t\Big)^2 \, .
\end{equation}
The second identity comes from the fact that: (i) the right-hand side is a lower bound  for the left-hand side (by Cauchy--Schwarz inequality); (ii) the lower bound is attained for the parametrization of $s$ by its length, that is choosing $\varphi$ such that $\int_0^{\varphi(u)} \|s'(t)\| \dd t = u \int_0^{1} \|s'(t)\| \dd t$ for $u\in [0,1]$.

Let us also stress that
$\mathrm{J}_d (x)\geq \frac{1}{2} \|x\|^2$ for all $x\in \bbR^d$,
so that $\hatent(s) \geq \frac1d\ent(s)$ for all $s\in \cD$.

\smallskip
The continuous energy-entropy variational problem that we expect to arise 
as the scaling limit of the log-partition function are $\sup_{ s\in\cD}  \big\{ \beta \pi(s) - \ent(s) \big\}$, the entropy term $\ent(s)$ 
being one of~\eqref{def:entropyCont2}
or~\eqref{def:entropyCont} depending on the scaling considered.
As an important part of the proofs, we will show that these continuous variational problems are well-defined.

\section{Main Results}

Before we state our results, let us define more precisely 
what we intend when we say that 
the polymer has transversal fluctuations~$N^{\xi}$.

\begin{definition}\label{def:transfluct}
We say that $(S_n)_{0\leq n \leq N}$ has  transversal fluctuations of order $r_N$ under $\bP_{N,\gb_N}^{\go,h}$ if for any $\gep>0$ there exists some $\eta\in(0,1)$ such that for all~$N$ large enough
\[
\bP_{N,\gb_N}^{\go,h} \Big(  \max_{1\leq n \leq N} \|S_n\|  \in [\eta,\tfrac1\eta]\,  r_N  \Big) \geq 1- \gep \quad \text{ with $\bbP$-probability larger than $1-\gep$}\, .
\]
\end{definition}
We separate our results according to the different regions we consider.
In all the rest of the paper, we work in dimension $d\geq 2$.

\smallskip
\noindent
{\it Notational disclaimer.} For two sequences $(a_n)_{n\ge 1}, (b_n)_{n\ge 1}$, we write $a_n \sim b_n$ if $\lim_{n\to\infty} a_n/b_n =1$, $a_n \ll b_n$ if $\lim_{n\to\infty} a_n/b_n =0$, and $a_n \asymp b_n$ if $0<\liminf a_n/b_n \le \limsup a_n/b_n <\infty$.

\subsection{Region A: $\ga\in (0,d)$, $\gamma \leq \frac{d-\ga}{\ga}$}

Our first result gives the scaling limit of the model in region~\textbf{A} of Figure~\ref{fig1}.  
This is the analogue of what
Auffinger and Louidor \cite{AL11} proved in the context of the directed polymer model in dimension $1+1$:
the following result therefore generalizes~\cite{AL11}  to the non-directed framework and to higher dimensions.

\begin{theorem}
\label{thm:A}
Let $\ga\in (0,d)$ and let $(\gb_N)_{N\geq 1}$ be such that $\lim_{N\to \infty}  N^{\frac{d-\ga}{\ga}} \gb_N = \gb  \in (0,+\infty]$.
Then for any fixed $h\in \bbR$, we have the following convergence in distribution, as $N\to\infty$,
\begin{equation}
\label{def:varprob1}
\frac{1}{ \gb_N N^{d/\ga}} \log Z_{N,\gb_N}^{\go,h}  \stackrel{(d)}{ \longrightarrow} \hat\cT_{\gb} := \sup_{s\in \cD, \hatent(s) <+\infty} 
\big\{ \pi(s) - \tfrac{1}{ \gb} \hatent(s) \big\} \,, 
\end{equation}
where $\pi(s)$ and $\hatent(s)$ are defined in \eqref{def:energy} and \eqref{def:entropyCont2} respectively.
\end{theorem}

\begin{remark}\label{rem:Skorokhod}
Let us stress that by an extended version of
Skorokhod representation theorem, see \cite[Cor.~5.12]{K97},
one can upgrade the convergence~\eqref{def:varprob1}
to an almost sure one. More precisely, if we look at $Z_{N,\gb_N}^{\go,h}$ as a function of $\mathcal P^{(N)}:=(x, \omega_x)_{x\in [-N,N]^d\cap \mathbb Z^d}$ and the continuous limit $ \hat\cT_{\gb}= \hat\cT_{\gb}(\mathcal P)$, then by \cite[Cor.~5.12]{K97}, the convergence in distribution in Theorem \ref{thm:A} implies that there exist some random elements $\tilde {\mathcal P}^{(N)}$ and $\tilde {\mathcal P}$ defined on the same probability space such that $\tilde {\mathcal P}^{(N)}\overset{(d)}{=}\mathcal P^{(N)}$, $\tilde {\mathcal P}\overset{(d)}{=}\mathcal P$ and such that $\lim_{N\to\infty}\frac{1}{ \gb_N N^{d/\ga}} \log Z_{N,\gb_N}^{\go,h}(\tilde {\mathcal P}^{(N)}) = \hat\cT_{\gb}(\tilde {\mathcal P })$ almost surely.
\end{remark}

It therefore makes sense to work conditionally
on $\hat \cT_{\gb} >0$, even at the discrete level.
We have the following corollary, which says that
conditionally
on $\hat \cT_{\gb} >0$,
 transversal fluctuations are of order $N$.
\begin{proposition}
\label{prop:fluctuationsN}
Let $\ga\in (0,d)$ and let $(\gb_N)_{N\geq 1}$ be such that $\lim_{N\to \infty}  N^{\frac{d-\ga}{\ga}} \gb_N = \gb  \in (0,+\infty]$.
Then $(S_n)_{0\leq n \leq N}$
has transversal fluctuations of order $N$ under $\bP_{N,\gb_N}^{\go,h} \big( \, \cdot \, | \,\hat \cT_{\gb} >0 \big)$.
\end{proposition}

Let us now give some property on the variational
problem $\hat \cT_{\gb}$, and in particular state that it is well-defined.

\begin{proposition}
\label{prop:transition}
When $\ga\in (0,d)$,
the variational problem $\hat\cT_{ \gb}$
defined in \eqref{def:varprob1}
is a.s.\ finite  for all $\gb \in (0,+\infty]$.
Moreover, letting $\gb_c = \gb_c(\cP):= \inf\{ \gb >0 \colon \hat \cT_{ \gb} >0 \}$, we have that 
$\bbP(  \gb_c  \in (0,\infty) ) =1$  if  $\alpha \in (0,\frac{d}{2})$ and
$\bbP(  \gb_c = 0 ) =1$  if  $\alpha \in (\frac{d}{2},d)$.
\end{proposition}

\begin{remark}\rm 
\label{rem:1/beta}
We could have formulated the convergence in distribution in Theorem~\ref{thm:A} as follows:
\[
\frac{1}{ N} \log Z_{N,\gb_N}^{\go,h}  \stackrel{(d)}{ \longrightarrow} \sup_{s\in \cD, \hatent(s) <+\infty} 
\big\{  \gb \pi(s) - \hatent(s) \big\}  =  \gb\,  \hat \cT_{\gb} \,, \quad \text{ as } N\to\infty\, .
\]
However, in the case $ \gb = +\infty$,
this would only give that $\frac{1}{ N} \log Z_{N,\gb_N}^{\go,h}$
goes to $+\infty$.
The formulation of Theorem~\ref{thm:A} allows us to treat
the case $\gb = +\infty$,
in which $\gb_N N^{d/\ga}$ is much larger than $N$ ---~in particular, it
includes the case $\gamma< \frac{d-\ga}{\ga}$.
In that case, the entropy term disappears in $\hat \cT_{\infty}$,
even though the constraint that paths have a finite
entropy $\hatent(s)$ has to remain (for instance to avoid having paths of length larger than $1$).
\end{remark}

\subsection{Region B: $\ga \in (\frac{d}{2}, d)$, $ \frac{d-\ga}{\ga} <\gamma  < \frac{d}{2\ga}$}

In region \textbf{B}, our result is the analogous 
to Theorem~2.5 in \cite{BT19a} in the context of the directed polymer model in dimension $1+1$:
this generalizes~\cite{BT19a} to the non-directed framework and to higher dimensions.
Here, we will fix the parameter $h$ to be equal to $\mu:=\bbE[\go_0] \in (0,\infty)$, which is well defined since $\ga>\frac{d}{2} \geq 1$.

\begin{theorem}
\label{thm:B}
Let $\ga\in (\frac{d}{2},d)$ and
$\gamma  \in ( \frac{d-\ga}{\ga} , \frac{d}{2\ga})$.
Let $(\gb_N)_{N\geq 1}$ verify $\lim_{N\to \infty}  N^{\gamma} \gb_N = \gb  \in (0,+\infty)$.
Then $(S_n)_{0\leq n \leq N}$ has transversal fluctuations of order $N^{\xi}$ under $\bP_{N,\gb_N}^{\go,h=\mu}$, where $\xi := \frac{\alpha (1-\gamma)}{2\ga-d} \in (\frac12,1)$.
Additionally,
we have the following convergence in distribution, as $N\to\infty$,
\begin{equation}
\label{eq:VarPr}
\frac{1}{N^{2\xi -1}}\log Z_{N,\gb_N}^{\go,h=\mu} \overset{(d)}{\longrightarrow}\cT_{ \beta} :=\sup_{s\in \cD, \ent(s) <+\infty} 
\big\{  \gb \pi(s) -  \ent(s) \big\} \,,
\end{equation}
where $\pi(s)$ and $\ent(s)$ are defined in \eqref{def:energy} and \eqref{def:entropyCont} respectively.
\end{theorem}
Let us stress that the choice of $h=\mu$ is crucial because the contribution of the small values of the environment (that is $\beta_N\omega_x\le 1$)  to the partition function is negligible only if we center the random variables $\omega_x$. This centering term is also needed in the directed case, cf.\ Section 4.2 Eq. (4.27) of~\cite{BT19a}.

The fact that the variational problem $\cT_{ \gb}$
is well-defined in the case $\alpha\in (\frac{d}{2},d)$ is non-trivial,
and relies on (non-directed) entropic last-passage percolation (E-LPP) estimates, introduced in~\cite{BT19c}.
Properties of the variational problem $\cT_{ \gb}$
are summarized in the following Proposition, which is the analogous of \cite[Thm.~2.4]{BT19b} (and generalizes it to the non-directed case and to higher dimensions).

\begin{proposition}
\label{prop:polym}
If $\alpha\in (\frac{d}{2},d)$, 
the variational problem $\cT_{ \gb}$ defined
in \eqref{eq:VarPr} is a.s.\ positive and finite for all $ \gb \in (0,+\infty)$.
Moreover,  $\bbE[ ( \cT_{\gb})^{\kappa}] <\infty$ for any $\kappa <\ga-d/2$, and we have the scaling relation
\begin{equation}\label{eq:scalinrT}
 \cT_{\gb} \stackrel{(d)}{=} \beta^{\frac{2\alpha}{2\alpha-d}}\,  \cT_{1} \, .
\end{equation}
On the other hand, if $\ga \in (0, \frac{d}{2}]$, we have $\cT_\beta=+\infty$ a.s., for all $\gb >0$.
\end{proposition}

\begin{remark}\label{rem:generalseq}\rm
If we consider more general sequences $\gb_N$,
then the corresponding transversal fluctuations $r_N$
are given by the relation~\eqref{eq:xi}, that is $r_N \sim (N \gb_N)^{\frac{\ga}{2\ga-d}}$.
In analogy with \cite[Thms.~2.5--2.7]{BT19a},
we should find three different scaling limits in region \textbf{B},
according to whether $\sqrt{N \log N} \ll r_N$ (corresponding to
the bulk of region \textbf{B} Theorem~\ref{thm:B} above),
$r_N \asymp \sqrt{N \log N}$ or $\sqrt{N} \ll r_N \ll \sqrt{N\log N}$
(corresponding to boundary regions between region \textbf{B} and region~\textbf{C}).
We have chosen here to consider only pure powers
for $\gb_N$ (and for $r_N$) in order to keep the exposition
clearer ---~the arguments from~\cite{BT19a} could be adapted here
but it would significantly lengthen the paper.
\end{remark}

\subsection{Region C: $\ga \in (0, d)$, $\gamma \geq \max\{ \frac{d-\ga}{\ga} , \frac{d}{2\ga} \}$}
\label{results:C}

In region \textbf{C}, we will show that transversal fluctuations of the polymer are of order $N^{1/2}$.
As a preliminary remark, let us stress that  for any $x\in \bbR^d \setminus \{0\}$ (omitting integer parts 
to lighten notation), 
we have the following asymptotics for $\bP(x\sqrt{N} \in \cR_N)$. Setting $v_n := \log N$ if $d=2$
and $v_n = N^{\frac{d}{2}-1}$ if $d\geq 3$, we have
\begin{equation}
\label{def:f}
\lim_{N\to\infty} v_N \bP\big( x \sqrt{N} \in \cR_N \big) = f(x) :=\begin{cases}
\int_{\|x\|^2/2}^{\infty} u^{-1} e^{-u} \dd u &\quad  \text{ if } d=2 ,\\
2\lambda_d \int_{0}^1 \rho_d(u,x) \dd u & \quad \text{ if } d\geq 3 \, ,
\end{cases}
\end{equation}
where $\lambda_d := \bP(S_n \neq 0 \,, \forall n\geq 1 )$ is the  escape probability and $\rho_d(t,x) := (2\pi t/d)^{-d/2} e^{ - \frac{d}{2t} \|x\|^2}$ is the heat kernel of a $d$-dimensional Brownian motion with covariance matrix $\frac1d \mathrm{Id}$.
The asymptotics~\eqref{def:f} can
be derived from Uchiyama's results~\cite{Uch11}:
see~\cite[Thm.~1.6]{Uch11} for the case of dimension $d=2$
and~\cite[Thm.~1.7]{Uch11} for dimensions $d\geq 3$.

\smallskip

In region \textbf{C}, the cases $\ga\in (0,\frac{d}{2})$, $\gamma \geq \frac{d-\ga}{\ga}$ and $\ga \in (\frac{d}{2},d)$, $\gamma\geq \frac{d}{2\ga}$ need to be treated separately,
similarly to the case of the directed polymer model~\cite{BT19a}.

\subsubsection*{Case $\ga \in (\frac{d}{2},d)$}
We start with the case $\ga \in (\frac{d}{2},d)$, $\gamma\geq \frac{d}{2\ga}$, which  is easier to state. 
Let us introduce some quantities that arise in the limit: we stress that different objects arise depending on the dimension.

\begin{proposition}
\label{prop:finiteSum}
In dimension $d\geq 5$,
define
\begin{equation}
\label{def:Chi}
\cX := \sum_{x\in \bbZ^d} (\go_x -\mu) \bP(x\in \cR_{\infty})  = \bE\Big[   \sum_{x\in \cR_{\infty}} (\go_x -\mu) \Big]\, ,
\end{equation}
where we recall that $\mu=\bbE[\go_0]$.
Then the random variable $\cX$ is well-defined if $\ga> \frac{d}{d-2}$.
\end{proposition}

\begin{proposition}\label{prop:finiteW}
In dimensions $d=2,3$ and $\ga\in (\frac d2 ,2)$, define
\begin{equation}\label{eq:defWbeta}
\cW_{\beta}=\int_{\R^d\times\R_+} \frac{1}{\gb} \big(e^{\beta w}-1-\beta w \big) f(x) \cP(\dd x,\dd w)+\int_{\R^d\times\R_+}w f(x)(\cP-\eta)(\dd x,\dd w),
\end{equation}
with $f(x)$ defined in \eqref{def:f} and $\eta(\dd x,\dd w)=\alpha w^{-\alpha-1}\dd x\, \dd w$ the intensity measure for $\cP$.
If $\alpha\in(\frac{d}{2},2)$, then for any $\beta\in (0,\infty)$, the random variable $\cW_\beta$ in \eqref{eq:defWbeta}
is almost surely finite.
\end{proposition}

We are now ready to state our results in the case $\ga \in (\frac{d}{2},d)$, $\gamma\geq \frac{d}{2\ga}$.
It is the analogous of Theorem~1.4 in~\cite{DZ16} for the directed polymer model in dimension $1+1$:
it generalizes~\cite{DZ16} to the non-directed case and to higher dimensions. Here again, we need to fix $h=\mu:=\bbE[\mu_0]$.

\begin{theorem}
\label{thm:C1}
Let $\ga\in (\frac{d}{2},d)$ and let $(\gb_N)_{N\geq 1}$
be such that $\lim_{N\to\infty}  N^{\frac{d}{2\ga}} \gb_N = \gb \in [0,+\infty)$.
Then
 $(S_n)_{0\leq n \leq N}$ has transversal fluctuations of order $N^{1/2}$ under $\bP_{N,\gb_N}^{\go,h=\mu}$.
We also have the following convergences in distribution, as $N\to\infty$:
 
\textbullet\ If $\ga \in (2\vee \frac{d}{2} ,d)$ (in particular $d\geq 3$), then setting $a_N= N^{1/4}$ for $d=3$, $a_N= (\log N)^{1/2} $ for $d= 4$, and $a_N=1$ for $d\geq 5$ we have:
\begin{equation}
\label{conv:C11} 
 \frac{1}{a_N \gb_N } \Big(\log Z_{N,\gb_N}^{\go,h=\mu} -  \frac12 \bbV\mathrm{ar}(\go) \gb_N^2 \bE[|\cR_N|] \Big)
\overset{(d)}{\longrightarrow}
\begin{cases}
\cN(0,\sigma_d^ 2 \bbV\mathrm{ar}(\go) )   & \ \ \text{ if } d=3,4,\\
\cX & \ \ \text{ if } d\geq 5\,,
\end{cases}
\end{equation}
for some explicit constant $\sigma_d$; recall $\cX$ has been defined in~\eqref{def:Chi}.

\textbullet\ If $\ga \in (\frac d2,2)$ (in particular $d=2$ or $d=3$), then setting $v_n := \log N$ for $d=2$
and $v_n = N^{1/2}$ for $d= 3$, we have
\begin{equation}
\label{conv:C12}
\frac{v_N}{\gb_N N^{\frac{d}{2\ga}}} \log Z_{N,\gb_N}^{\go,h=\mu}
\overset{(d)}{\longrightarrow}
\cW_{\gb} \,,
\end{equation}
with $\cW_{\gb}$ defined in~\eqref{eq:defWbeta} ---~if $\gb =0$  the first term in~\eqref{eq:defWbeta} is set to zero, see~Proposition~\ref{prop:finiteW0} below.
\end{theorem}

\begin{remark}\label{rem:betaNERN}
Let us stress that when $\ga>d/2$, in dimension $d\geq 3$
we have $\gb_N^2 \bE[|\cR_N|] \sim c_{\gb} N^{1-d/\ga}$ (assume $\gb>0$ to simplify), so in particular it goes to $0$ as $N\to\infty$.
One can easily check that we always have $(a_N \gb_N)^{-1} \times \gb_N^2 \bE[|\cR_N|] \to +\infty$ when $\ga > 2\vee \frac d2$ (and thus, $d\geq 3$), so the centering in~\eqref{conv:C11}  is non-trivial.
\end{remark}



\subsubsection*{Case $\ga\in (0,\frac{d}{2})$}
In this second case, recall Theorem~\ref{thm:A}
and Proposition~\ref{prop:transition}.
If $\ga\in (0,\frac{d}{2})$ and $\gb_N \sim \gb N^{-(d-\ga)/\ga}$ as $N\to\infty$,
then the scaling limit of the log-partition function
has been identified as $\hat \cT_{\gb}$.
However,
when $\ga\in (0,\frac{d}{2})$ we have $\hat \cT_{\gb} =0$ when $\gb \leq \gb_c$, with $\gb_c=\gb_c(\cP) >0$ a.s.: in that case, the scaling limit is thus trivial.
Our next result shows that when $\gb \leq \gb_c$, then transversal fluctuations
are necessarily of order $\sqrt{N}$;
this has to be compared with Proposition \ref{prop:fluctuationsN} which asserts that transversal fluctuations
are of order $N$ when $\gb >\gb_c$.
This shows that when $\alpha\in (0, \frac d2)$ we cannot  have intermediate fluctuations between $N^{1/2}$ and $N$: the system exhibits a sharp phase transition at the critical point $\beta=\beta_c$.
Our result is the analogous of Theorem~2.12 in~\cite{BT19a}, and generalizes it to the non-directed case and to higher dimensions.

Let us define another quantity, analogous to $\cW_{\gb}$ defined in~\eqref{eq:defWbeta}.
\begin{proposition}\label{prop:finiteW0}
For $\ga\in(0,1) \cup(1,2)$, define
\begin{equation}
\label{def:W0}
\cW_0=
\begin{cases}
\int_{\bbR^d\times\bbR_+} w f(x)(\cP-\eta)(\dd x, \dd w),
&\quad \text{if } \alpha\in (1,2), \\
\int_{\bbR^d\times\bbR_+} w f(x)\cP(\dd x, \dd w),
&\quad \text{if } \alpha\in (0,1) \,,
\end{cases}
\end{equation}
with $f(x)$ defined in \eqref{def:f} and $\eta(\dd x,\dd w)=\alpha w^{-\alpha-1}\dd x\, \dd w$ the intensity measure for $\cP$.
If $\ga\in(0,1)$, then  $\cW_{0}$ is a.s.\ finite in any dimension $d$.
If $\ga\in(1,2)$,
then $\cW_0$ is a.s.\ finite if and only if $\ga < \frac{d}{d-2}$.
\end{proposition}

Recall that by an extended version of
Skorokhod representation theorem \cite[Cor.~5.12]{K97} (cf. Remark~\ref{rem:Skorokhod}),
it makes sense to work conditionally on $\hat\cT_{\gb}>0$ or $\hat \cT_{\gb} =0$, even at the discrete level.
If $\ga>1$, we need to take $h=\mu:=\bE[\go_0]$ and if $\ga<1$, then $h$ may be any real number.

\begin{theorem}
\label{thm:C2}
Let $\ga\in (0,\frac{d}{2})$, and let $(\gb_N)_{N\geq 1}$
be such that $\lim_{N\to\infty}  N^{\frac{d-\ga}{\ga}} \gb_N = \gb \in [0,+\infty)$. 
Then conditionally on the event $\{\hat \cT_{\gb} =0 \}$ (\textit{i.e.}\ $\gb\leq \gb_c$) the polymer $(S_n)_{0\leq n \leq N}$ has transversal fluctuations of order $N^{1/2}$ under~$\bP_{N,\gb_N}^{\go,h}$.
Moreover, conditionally on  $\{\hat \cT_{\gb} =0 \}$,
we have the following convergences in distribution, as $N\to\infty$.

\textbullet\ If $\ga \in (\frac{d}{d-2},\frac{d}{2})$ (in particular $d\geq 5$), then we have
\begin{equation}
\label{convC2-1}
\frac{1}{\gb_N} \log Z_{N,\gb_N}^{\go,h=\mu}  \overset{(d)}{\longrightarrow} \cX \,, 
\end{equation}
 with $\cX$ defined in \eqref{def:Chi}.

\textbullet\ If $\ga < \min(\frac{d}{2} ,\frac{d}{d-2} )$ (in particular $\ga<2$) and  $\ga\neq 1$, then 
\begin{equation}
\label{convC2-2}
\frac{v_N}{\gb_N N^{\frac{d}{2\ga}}} \log Z_{N,\gb_N}^{\go,h}
\overset{(d)}{\longrightarrow}
\cW_{0},
\end{equation}
where $v_n := \log N$ if $d=2$
and $v_n = N^{\frac{d}{2}-1}$ if $d\geq 3$, with $\cW_0$ defined in~\eqref{def:W0}.
%
\end{theorem}



\begin{remark}\label{alpha=1}\rm
The case $\alpha=1$  in \eqref{convC2-2} could be treated similarly to what is done in \cite[Thm~1.4]{DZ16}. In the case $\ga=1$, a centering term for $\log Z_{N,\beta_N}^{\omega,h}$ is needed in~\eqref{convC2-2} and the scaling limit should be $\cW_0=\int_{\bbR^d\times(0,1]} w f(x)(\cP-\eta)(\dd x,\dd w)+\int_{\bbR^d\times(1,\infty)}\omega f(x)\cP(\dd x,\dd w)$. 
Since this is fairly technical, we prefer to omit the details for simplicity.
\end{remark}



\subsection{Comparison with directed polymers}
As we stressed in the introduction of this paper, the model~\eqref{themodel} is closely related to the directed polymer 
model --- our results can be seen as an extension of existing results to a non-directed setting and to higher dimension.
Let us now briefly discuss this relation by comparing the techniques exploited in the present article with the one developed for the directed polymer.

\smallskip
\textbullet\ In Region \textbf{A} our results extend \cite{AL11, HM07}, where the authors considered the directed polymer in dimension $1+1$.
In this region, only a few points give an energy contribution to the variational problem: the random walk linearly interpolates between these points to get the maximal energy reward to compensate the entropy cost. The fact that we can approximate the problem by considering only a finite number of points allows us to extend and exploit the techniques introduced for the directed case: the scheme of the proof
is very close to the one in~\cite{AL11}.

\smallskip
\textbullet\ In Region \textbf{B}, our results extend the analysis performed in \cite{BT19a} to the non-directed framework. In particular our techniques are based on the non-directed Entropy-controlled last passage percolation (E-LPP) introduced in \cite{BT19c} (as an extension of the directed E-LPP of~\cite{BT19b}), a crucial step
consisting in showing that we can restrict the partition function to trajectories staying at scale $N^\xi$. The extension to a non-directed setting and to higher dimension is the main novelty here, as discussed in Appendix~\ref{sec:ELPP}.
Overall, the proof follows the same scheme as in~\cite{BT19a},
but several adaptations are needed, with some tedious technicalities.

\smallskip
\textbullet\ In region \textbf{C}, we perform a polynomial expansion analysis, in the same spirit as in \cite{DZ16} for the directed polymer model with heavy-tail disorder. The crucial difference here is that in our non-directed case the geometry of the range plays a central role in the behavior of the model (and so in the polynomial expansion analysis). 
We therefore have to use here a local central limit theorem for the range, considering the probability that a point at scale $\sqrt N$ belongs to the range $\cR_N$, \textit{i.e.}\ $\bP(x\sqrt N\in \cR_N)$. 
Even if the proofs in region \textbf{C} start from the same idea as in~\cite{DZ16}, the technical treatment of the polynomial expansion
requires new ideas and becomes more technical
and dimension dependent.
Let us stress that the case of dimension $d=2$ requires particular care, because then $\bP(x\sqrt N\in \cR_N)$ scales as $\log N$,
compared to the polynomial behavior $N^{\frac{d}{2}-1}$ in dimension $d\ge 3$.

To conclude, let us comment on the
limiting random variables that appear in Theorems~\ref{thm:C1}-\ref{thm:C2}; note that they indeed depend on the dimension.
The random variable $\cW_{\gb}$ (see Propositions~\ref{prop:finiteW}-\ref{prop:finiteW0}) appears when $\ga<2$, in dimensions $d=2,3$ or when $\gb=0$: it is the analogous of~$\cW_{\gb}^{(\alpha)}$ defined in~\cite[p.~4011]{DZ16} (see also \cite[Thm.~1.4]{DZ16}).
When $\ga>2$, then in dimensions $d=3,4$, analogously to~\cite[Thm.~1.2]{DZ16}
a normal random variable appears,
whereas in dimension $d\geq 5$ a new random variable~$\mathcal{X}$ pops up (see Proposition~\ref{prop:finiteSum} for the definition of $\mathcal{X}$),
which to our knowledge has no analogous in the literature.

\subsection{Further comments and conjectures}

Our article solves completely the case $\alpha \in (0,d)$, but there are several aspects that remain to be tackled:
\begin{itemize}
\item In Region~\textbf{B} and in Region~\textbf{C} with $\alpha > \min\{2,\frac{d}{2},\frac{d}{d-2}\}$, the parameter $h$ is needed to be set equal to $\mu$ to center the environment: the case of a general $h$ should be investigated;
\item More precise statements on the convergence of paths  could be extracted from our results, in particular one could try to prove in some cases a localization of the trajectories near an optimal path;
\item The case $\alpha>d$ is still a challenge: the wandering exponent $\xi$ is not known --- except in the intermediate disorder regime (Region \textbf{C}), where $\xi=1/2$ (but one still has to understand the corrrect scaling for $\gb_N$).
\end{itemize}

In this section, we develop further on these aspects, we comment their relation with the literature, and we present some open problems and conjectures.

\subsubsection{About the external field $h$} 

It is not hard to see that in Region \textbf{A}, the parameter~$h$ is unimportant. On the other hand, for Region \textbf{B} and for Region \textbf{C} with $\ga>\min\{2,\frac{d}{2},\frac{d}{d-2}\}$, we can see from the proofs (cf.\ Sections~\ref{sec:convpartfunc}-\ref{sec:C2}) that a centering is needed, so $h$ needs to be fixed equal to $\mu=\bbE[\omega_0]$.
In order to make the proofs more transparent (which are already quite technical) and to simplify some non-central arguments, we have chosen to stick to a non-negative disorder: this is used in particular in Region \textbf{A} and \textbf{B} (they are related to the \textit{discrete energy-entropy variational problems}, cf. Section~\ref{sec:varprofCont}).
This also has the advantage to illustrate 
when a centering of the environment is crucial or not.
Without the assumption of non-negativity of the disorder, the proofs may require some extra technical work, and one should still center the $\omega$ by the parameter $h$ when needed (or take $\bbE[\omega_0]=0$ and $h=0$).

Furthermore, we can also consider a more general setting for our non-directed polymer model, by considering the following:
\begin{equation*}
\frac{\dd \bP_{N,\beta_N}^{\omega,h_N}}{\dd \bP}(S):=\frac{1}{Z_{N,\beta_N}^{\omega,h_N}}\exp\bigg(\sum\limits_{x\in\cR_{N}}(\beta_N \omega_{x}-h_N)\bigg),
\end{equation*}
with $h_N=\hat{h}N^{-\zeta}$ and $\zeta\in\bbR$ (note that in this article, $h_N$ is simply $\beta_N h$). The above model can be viewed as a \textit{perturbed version of a random walk penalized by its range} and is more challenging. The case of the dimension $d=1$ has been analyzed thoroughly, see \cite{BHWT20}.
In particular, we can expect that when $h_N$ is large enough (at least compared to $\beta_N$), \textit{i.e.}\ when the penalty by the size of the range dominates, then the polymer folds into a ball (similarly to \cite{BC18,Bolt94,DFSX18}).
However, contrary to the homogeneous model where the center of the ball is random, the location of the ball may be completely determined by the environment, since the polymer tends to maximize the sum of $\omega$'s seen in the ball---see~\cite{Bouchot22} for the one-dimensional case.
On the other hand, there should be a regime where a balance is found between the energy and the entropy (the penalization by the range is negligible---this is somehow the focus of this article), and a regime where a balance is found between the penalization by the range and the energy (the entropic cost of having an unusual range being out-weighted by the range penalty). For $\ga\in(0,d)$, the phase diagram of the above model should be  similar to the one found in the one-dimensional case, cf.\ \cite{BHWT20}.

\subsubsection{About the geometry of the range}

In this paper we consider the scaling limits of the logarithmic partition function and we describe the transversal fluctuations of the walk. Such results are the starting point to push further the analysis of the geometry of the range $\cR_N$. In this setting we conjecture that the limits of the logarithmic partition function contain all the information to describe the geometry of the random walk. 
To be more precise, in Region \textbf{A} and \textbf{B} we conjecture that the supremum of the variational problem $\hat \cT_{\gb}$ and $ \cT_{\gb}$ is attained by some unique continuous path that we call $\hat s^*$ and $s^*$ respectively, by analogy with \cite{AL11,BT19b}. We then conjecture that the typical range of the random walk is concentrated around these paths. More precisely, if we look at a path as a set, \textit{i.e.}\ if we consider its support in $\mathbb R^d$, then we conjecture that under the hypothesis of Theorem \ref{thm:A}, for any $\epsilon>0$,
\begin{equation}
\lim_{N\to\infty} \bP_{N,\gb_N}^{\go,h}\Big( \frac1N\cR_N\subset \mathfrak{B}_\epsilon(\hat s^*) \Big)
\xrightarrow{\bbP}1,
\end{equation}
and under the hypothesis of Theorem \ref{thm:B}, for any $\epsilon>0$,
\begin{equation}
\lim_{N\to\infty} \bP_{N,\gb_N}^{\go,h} \Big(\frac1{N^\xi}\cR_N\subset \mathfrak{B}_\epsilon( s^*) \Big)
\xrightarrow{\bbP}1,
\end{equation}
where $\mathfrak{B}_\epsilon( \gamma)=\{x\in \mathbb R^d \colon d(x,\gamma)<\epsilon\}$ and $\bbP$ is the coupling between the discrete and continuous disorder introduced in Remark \ref{rem:Skorokhod}.

In region \textbf{C} we expect that $\frac1{\sqrt N}\cR_N$ converges to a continuous limit, which should be a random perturbation of a Brownian motion range (but a precise statement is harder to state).

\subsubsection{About the intermediate disorder regime (region \textbf{C})}

In~\cite{H19}, Huang considers the intermediate disorder regime 
of the model
(that is a regime where $\xi=1/2$ but disorder has a non-trivial effect),
in the case of a disorder with finite exponential moment, \textit{i.e.}~$\bbE[e^{\gb \go_x}]<+\infty$ (or roughly speaking, $\alpha=+\infty$).
More precisely, Huang~\cite{H19} shows that, taking 
\begin{equation}
\label{eq:scaleHuang}
\gb_N = \hat \gb N^{-1/4}  \  \text{ if } d=1\,,
\qquad 
\gb_N = \hat \gb N^{-1/2} \log N  \  \text{ if } d=2\,,
\qquad 
\gb_N = \hat \gb N^{-1/4}  \  \text{ if } d=3\,,
\end{equation}
for some $\hat \gb>0$
and taking $h_N = \gb_N^{-1} \log \bbE[e^{\gb_N \go_x}]$,
then $Z_{N,\gb_N}^{\go,h_N}$ converges in distribution
toward a random variable $\mathcal{Z}_{\hat \gb}$, given by an explicit Wiener chaos expansion. 

Our Theorem~\ref{thm:C1} is the analogue of 
this result in the case $\frac d2<\ga<d$ (and any dimension $d\geq 2$).
In particular, \eqref{conv:C11}-\eqref{conv:C12} states that the choice $\gb_N =\hat \gb N^{-d/2\ga}$ is the correct scaling in order to observe non-trivial fluctuations for $Z_{N,\gb_N}^{\go,\mu}$.
Analogously to what is done in~\cite{DZ16} for the directed polymer,
one could try to extend our Theorem~\ref{thm:C2}
to the case $\alpha>d$, \textit{i.e.}\ to study 
the intermediate disorder regime for all values of $\alpha$.
We leave this as an open problem, but let us comment on the expected results.

\smallskip
In dimension $d=2$, we expect that the correct intermediate disorder scaling is $\gb_N =\hat \gb N^{-d/2\ga} = N^{-1/\ga}$ for all $1<\alpha<2$ (the scaling limit is then given by Theorem~\ref{thm:C1}) and that one reaches the scaling $\gb_N = N^{-1/2} \log N$ given in~\eqref{eq:scaleHuang} for all $\ga>2$, with then the scaling limit as in~\cite{H19}.

\smallskip
In dimension $d=3$, we expect that the correct intermediate disorder scaling is $\gb_N =\hat \gb N^{-d/2\ga} = N^{-3/2\ga}$ for all $\frac32<\alpha<6$, with a scaling limit similar to Theorem~\ref{thm:C1}, and that one reaches the scaling $\gb_N = N^{-1/4}$ given in~\eqref{eq:scaleHuang} for all $\ga>6$,
with then the same scaling limit as in~\cite{H19}.
A way to understand this comes from~\eqref{conv:C11}: loosely speaking, it says that in dimension $d=3$, taking $\gb_N  = \hat \gb N^{-3/2\ga}$ we get that
$Z_{N,\gb_N}^{\go,\mu} \approx 1+ \gb_N N^{1/4} \mathcal{Z}$ for some random variable $\mathcal Z$.
This approximation may remain valid for some $\ga>d$ but it must fail when $\gb_N N^{1/4} =\hat \gb N^{-3/2\ga+1/4}$ becomes of order $1$ (that is when $\ga$ reaches the value~$6$),
at which point all terms in the polynomial chaos expansion of the partition function are of the same order.

\smallskip
In dimension $d \geq 4$, 
the correct intermediate disorder scaling should be $\gb_N =\hat \gb N^{-d/2\ga}$ for any $\ga>d/2$, with a scaling limit similar to Theorem~\ref{thm:C1}. 
Indeed, the convergence~\eqref{conv:C11} states that
taking $\gb_N  = \hat \gb N^{-d/2\ga}$ we have
$Z_{N,\gb_N}^{\go,\mu} \approx 1+ \gb_N a_N \mathcal{Z}$ for some random variable $\mathcal Z$:
since $a_N = (\log N)^{1/2}$ in dimension $d=4$ and $a_N=1$ in dimension $d\geq 5$, we  have that $\gb_N a_N \to 0$ for any $\ga<+\infty$.
In that case, we should therefore have that $\lim_{n\to\infty} Z_{N,\gb_N}^{\go,\mu} =1$, with fluctuations of order $\beta_N a_N$.
Additionally, for disorder with exponential moments (\textit{i.e.}\ $\alpha=+\infty$), the approximation $Z_{N,\gb_N}^{\go,\mu} \approx 1+ \gb_N a_N \mathcal{Z}$ suggests that:
in dimension $d=4$, one should take $\beta_N = \hat \beta /(\log N)^{1/2}$ to observe a non-trivial scaling limit for $Z_{N,\gb_N}^{\go,\mu}$ (disorder is marginally relevant);
in dimension $d\geq 5$, one should have that $Z_{N,\gb}^{\go,\mu}$ has a non-trivial limit for small (but fixed) $\beta$ --- in other words, there is a weak disorder phase (disorder is irrelevant), in analogy with the directed polymer model in dimension $d\geq 3$.


\subsubsection{About the transversal fluctuations in the case $\ga>d$}
Let us now consider the case of $\beta$ fixed and let us recall the Flory argument presented in Section \ref{prediction}:
for $\xi>1/2$,
the energetic gain of a polymer in a box of size~$N^\xi$ is of order $N^{d\xi/\alpha}$, while the entropic cost of going at distance $N^\xi$ is of order $N^{2\xi-1}$. The energy-entropy balance leads to the prediction that
$\xi=\frac{\alpha}{2\alpha-d}$.
This should hold as long as $\xi \in (1/2,1)$.
In particular,  if $\alpha<d$ we should have that $\xi=1$: this is what is proven (in a more general framework) in the present paper,
see Theorem~\ref{thm:A} (Region \textbf{A}).

When $\ga>d$, in comparison with the directed polymer literature (see e.g.~\cite{GLBR15}), we may conjecture that there exists some $\alpha_c$ such that for $\ga\in (d,\ga_c)$ we have $\xi=\frac{\alpha}{2\alpha-d}$
and for $\alpha>\alpha_c$  we have $\xi= \xi_c$,
where $\xi_c=\xi_c(\gb) \geq 1/2$ is  the fluctuation exponent 
obtained for a disorder with finite exponential moments.
Such $\alpha_c$ should therefore solve the equation $\frac{\alpha}{2\alpha-d}=\xi_c$, that is $\ga_c = \frac{d \xi_c}{2\xi_c-1}$; by convention $\ga_c =+\infty$ when $\xi_c=1/2$.
This  fact has indeed been proven in dimension $d=1$ in \cite{BHWT20}, in the non-directed setting of the present paper:
in that case we have $\xi =1$ if $\ga\in (0,1)$, $\xi=\frac{\ga}{2\ga-1}$ if $\ga \in (0,2)$ and $\xi=\frac23$ for all $\ga>2$.

The question in dimension $d\geq 2$ remains quite mysterious
(both for non-directed and directed polymers),
in particular because the fluctuation exponent $\xi_c$
is unknown. 
It is expected that, at least in low dimensions, 
there is some critical value $\gb_c$ below which 
$\xi_c=\frac12$ (which would lead to $\ga_c=+\infty$) and above which $\xi_c >1/2$ (which would lead to $\ga_c<+\infty$).
In view of the previous subsection, one could argue that
in dimension $d=2,3$ we have $\gb_c=0$, but for now this is purely speculative.
Let us also mention that the existence of an upper critical dimension $d_0$
above which $\xi_c(d,\gb) =\frac12$ for all $\gb$
is still controversial, see e.g.~\cite{LK97}.
In order to extend the analysis of the problem to $\alpha > d$, a first step would be to study the existence and dependence on the parameters of the conjectured exponent $\beta_c$.

\subsection{Organization of the rest of the paper}

We now present the organization of the paper and describe how the proofs are organized.
\smallskip

$\bullet$ In Section \ref{sec:propVarProbl} we prove that the variational problems $\cT_\gb$, $\hat \cT_\gb$ and the random variables $\cX$ and $\cW_\beta,\cW_0$ are well-defined (\textit{i.e.}\ Propositions~\ref{prop:transition}, \ref{prop:polym}, \ref{prop:finiteSum}, \ref{prop:finiteW} and \ref{prop:finiteW0}).
Along the way we discuss their main properties.

\smallskip
$\bullet$ In Section \ref{sec:varprofCont} we present the main auxiliary results about the discrete approximation of the energy-entropy variational problems $\cT_\gb$ and $\hat \cT_\gb$ that we use to prove Theorems~\ref{thm:A}-\ref{thm:B}. These results extend the analogue results obtained in \cite{BT19b, BT19c} to the $d$-dimensional and non-directed case; they are related to the Entropy-controlled Last-Passage Percolation (E-LPP). Some results on the \textit{non-directed} version of the E-LPP are postponed to the Appendix~\ref{sec:ELPP}.
\smallskip

$\bullet$ In Section \ref{sec:regionA} we prove Theorem \ref{thm:A}, \textit{i.e.}\ region \textbf{A}. To prove this, we first reduce to a partition function restricted to the largest $L$
weights $Z_{N,\beta_N}^{(L)}$, see~\eqref{ZL}, by showing that the contribution of all smaller weights is negligible, see Lemma~\ref{ZLellvanish+}. We then show that  $\log Z_{N,\beta_N}^{(L)}$, properly rescaled, converges to the continuous variational problem restricted to the largest $L$ weights, see Lemma \ref{logZL}.

\smallskip

$\bullet$ In Section \ref{sec:convpartfunc} we prove Theorem \ref{thm:B}.
As a crucial first step,  in Section \ref{subsec:step1B}, we show that we can restrict  the partition function to trajectories staying at scale $N^{\xi}$, see~Proposition \ref{prop:Z>Ar}.
To achieve this, we need to show that the entropic cost to reach a distance $h_N \gg N^{\xi}$ cannot be overcompensated by 
an energetic gain at scale $h_n$: this is the purpose of Lemma~\ref{lem:S<qr}. This relies once again on controlling the contribution of different ranges of weights to the partition function; the most technical part actually consists in controlling the contribution of intermediate weights, which requires a detailed analysis to match the energetic gain with the entropic cost.
Once trajectories have been reduced to being at scale $N^{\xi}$, the convergence of the partition function is performed in Section~\ref{subsec:step2B}, see Proposition \ref{prop:convZq}; the strategy of the proof is similar to what is done in Section~\ref{sec:regionA} for region \textbf{A}.

\smallskip

$\bullet$ In Section \ref{sec:regC} we prove Theorem \ref{thm:C1}. The idea of the proof is based on a truncation of the environment \eqref{tildeomega} which allows us to prove that the main contribution comes from trajectories that stay at scale $N^{1/2}$. 
The convergence of the partition function defined on the truncated environment is presented in Lemma \ref{lem:reduction0}. Its proof is based on a polynomial expansion analysis \eqref{polexpansion} in which we show that the convergence is led by the first term.
Different random limits arise depending on the tail exponent $\alpha$ and the dimension $d$.
If $\alpha >2$ the limit is Gaussian in small dimensions ($d=3,4$) and in high dimension ($d\ge 5$) the limit is described  by the averaged (w.r.t.\ the random walk) sum of the environment on $\cR_\infty$, see the definition~\eqref{def:Chi} of $\cX$. If $\alpha<2$ the limit conveys the heavy-tail structure of the environment, see the definition~\eqref{eq:defWbeta} of $\cW_{\beta}$.

\smallskip

$\bullet$ In Section \ref{sec:C2} we prove Theorem \ref{thm:C2}. 
 The most technical issue consists in showing that conditionally on $\{\hat\cT_\gb=0\}$ the main contribution to the partition function comes from trajectories that stay at scale $N^{1/2}$, see Section \ref{sec:regCprel8}; in particular, this shows that trajectories cannot  have intermediate scale.
Then, the strategy is similar to that used in Section~\ref{sec:regC}, combining a truncation of the environment and using a polynomial expansion.

\smallskip
$\bullet$
Finally, we collect in Appendix several technical estimates.
As mentioned above, Appendix~\ref{sec:ELPP} provides useful results 
for the non-directed E-LPP.
In Appendix~\ref{Appendix}, we collect some simple random walk estimates: large and moderate deviations, probabilities that a given set is visited, intersection of ranges of independent walks.
In Appendix~\ref{app:envir}, we give estimates on exponential moments of a truncated version of the environment.

\section{Properties of the limiting variational problems and random variables}
\label{sec:propVarProbl}

In this section, we prove the well-posedness of the variational problems $\hat \cT_{\gb}$ and $\cT_{\gb}$ (in Section~\ref{subsec:variat}, together with some of their properties),
of $\cX$ (in Section~\ref{sec:finiteSum}) and
of $\cW_{\gb}$ and $\cW_0$ (in Section~\ref{sec:wellposedWbeta}).

\subsection{Some notation: order statistics and truncated energy}

We introduce some notation to describe the contribution of large weights to the variational problems $\hat \cT_{\gb}$, $\cT_{\gb}$.
For $q>0$ we define
\begin{equation}\label{def:Dq}
\cD_{q}:=\Big\{s\in\cD\colon \ent(s)<\infty,\ \sup_{0\leq t\leq1}\|s(t)\|\leq q\Big\} \, ,
\end{equation}
and for a realization of $\cP$, we set $\cP_q:= \{ (x,w)\in \cP\,, \, \|x\|\leq q\}$. Then, the energy along a path $s\in\cD_{q}$ is given by
\begin{equation}\label{eq:energy1}
\pi_{q}(s):=\sum\limits_{(x,w)\in\cP_q } w\mathbbm{1}_{\{x\in s([0,1])\}}.
\end{equation}
In the ball $\mathbf\Lambda_q := \{ x \in \bbR^d \,, \|x\|\leq q\}$ we can also label the points of $\cP$ by using the \textit{order statistics}:
\begin{equation}\label{eq:energy3}
(x,w)_{(x,w)\in\cP_q }= \big( \bY_{i}^{(q)} , \bM_{i}^{(q)} \big)_{i\geq1}\,,
\end{equation}
where the distribution of $(\bY_{i}^{(q)} , \bM_{i}^{(q)})_{i\geq1}$
is given as follows:
\begin{itemize}
\item[(i)] $\bM_{i}^{(q)}=(c_d q)^{d/\ga}(\cE_1+\dots+\cE_i)^{-1/\ga}$, with $(\cE_i)_{i\in\N}$ i.i.d.\ exponential random variables of parameter~$1$; this is the \textit{sequence of weights},  $\bM_{i}^{(q)}$ being the $i$-th largest weight in~$\mathbf\Lambda_q$. 
See for instance \cite{D81} or \cite{HM07}.
\item[(ii)]  $(\bY_{i}^{(q)})_{i\ge 1}$ are i.i.d.\ uniform random variables on $\mathbf\Lambda_q$, independent of $(\cE_i)_{i\in\N}$;  these are the \textit{location of the weights};
\end{itemize}
We then define the truncated analogue of \eqref{eq:energy1}:
\begin{equation}\label{eq:energy2}
\pi_{q}^{(\ell)}(s)=\sum\limits_{i=1}^{\ell}\bM_{i}^{(q)}\mathbbm{1}_{\{\bY_{i}^{(q)}\in s([0,1])\}},
\end{equation}
and we also let $\pi_{q}^{(>\ell)} (s)=\pi_{q}(s)-\pi_{q}^{(\ell)}(s)\, .$

\begin{remark}\rm
\label{rem1}
Notice that if $s\in \cD$ verifies $\ent(s) \leq B$ with $B\geq 1$,
then the length of the path $s$ is bounded
by $\sqrt{2B/d} \leq B$, and so  we have $s\in \cD_B$.
\end{remark}

\subsection{On the variational problems $\cT_{\gb}$ and $\hat \cT_{\gb}$}
\label{subsec:variat}

Recall that $\cT_{\gb}$ and $\hat \cT_{\gb}$
are defined respectively in~\eqref{eq:VarPr} and~\eqref{def:varprob1}.
Here, we prove Propositions~\ref{prop:transition}-\ref{prop:polym}:
We mostly focus on the first one (\textit{i.e.}\ on~$\cT_{\gb}$):  we treat $\hat \cT_{\gb}$
along the way, since the results follow from a simple comparison 
with $\cT_{\gb}$ (recall that $\hatent(s) \geq \frac1d \ent(s)$ for all $s$), or with identical arguments.

Before we start, let us stress that
for almost every realization of $\cP$, the maps $\beta \mapsto \cT_\beta$ and $\beta \mapsto \hat \cT_{\gb}$ are non-decreasing and continuous.
The proof of this fact is identical to that of \cite[Thm.~2.4]{BT19b} (see Section~4.5 in \cite{BT19b}), so we omit it.

%

\subsubsection*{Proof of the scaling relation~\eqref{eq:scalinrT}}

Consider the Poisson Point Process $\cP$ on
$\R^d \times \bbR_+$ with intensity $\eta(\dd x , \dd w)= \ga  w^{-\ga -1} \dd x \dd w$. 
If we consider, for $a>0$, the scaling tranformation
$\phi_a(x, w)=(ax, a^{d/\ga} w)$, then we have that $\phi_a(\cP)\overset{(d)}{=}\cP$;
notice also that $\pi^{\phi_a \star\mathcal P}(  s )=a^{d/\ga}\pi(s/a)$ and $\ent( as)=a^2\ent(s)$. Therefore, since $\mathcal D=\frac1a \mathcal D$, by applying the scaling $\phi_a$, we get
\[
\sup_{ s\in \cD , \, \ent(s) <\infty}  \big\{  \beta  a^{-d/\ga}\pi(s) -\ent(s) \big\}\overset{(d)}{=}
\sup_{ s\in \cD, \, \ent(s) <\infty}  \big\{  \beta \pi(s) - a^2\ent(s) \big\},
\]
so that $\cT_{\beta a^{- d/\ga }}\overset{(d)}{=} a^{2}\cT_{\beta a^{-2}}$. This implies \eqref{eq:scalinrT}.
\qed

\subsubsection*{Finiteness of $\cT_{\beta}$ (and of $\hat \cT_{\gb}$)}

For any interval $[a,b)$ let us define
\begin{equation}\label{eq:Tab}
\cT_{\beta}([a,b)) := \sup\limits_{s: \ent(s)\in[a,b)}\big\{\beta\pi(s)-\ent(s)\big\}.
\end{equation}
 In such a way, we have that 
 \[
 \cT_{\beta}=\cT_{\beta} \big( [0,1) \big)\vee \sup_{k\ge 0}\cT_{\beta}\big( [2^{k},2^{k+1}) \big)\,.
 \] 
Using the scaling properties of $\cP$
and that having $\ent(s)\le 2$ implies $s\in \cD_2$ (see Remark \ref{rem1}),
 we get that
\begin{equation}\label{Tdec}
\cT_{\beta} \big( [2^{k},2^{k+1}) \big) \overset{(d)}{=} \sup\limits_{s\in \cD \,,  \ent(s)\in [1,2) } \big\{ 2^{\frac{dk}{2\alpha}}\beta \pi(s)-2^{k}\ent(s) \big\}
\le \beta 2^{\frac{dk}{2\alpha}} \sup\limits_{s\in \cD\,, \ent(s)<2 }\{\pi(s)\}-2^{k}\,.
\end{equation}
The following lemma is the key result to prove the finiteness of $\cT_{\gb}$.
Its proof is similar to that of~\cite[Lemma~4.1]{BT19b}
and is based on
entropy-controlled last-passage percolation (E-LPP)
estimates. We postpone its proof to Appendix~\ref{sec:ELPP}.
\begin{lemma}\label{lemma1}
Let $\ga\in (0,d)$.
For any $0<a<\ga$ there is a constant $c=c_{a,\ga}$ such that  for any $t>1$,
\[
\bbP \Big(\sup\limits_{s \in \cD\,, \ent(s)<2 }\{\pi(s)\} > t\Big) \le c \,  t^{-a}\,.
\]
In particular, it shows that for any $\ga<d$ we have that
$\sup_{\ent(s)<2 }\{\pi (s)\}$ is a.s. finite.
\end{lemma}
Using Lemma~\ref{lemma1} we can conclude the proof of the finiteness of $\cT_{\gb}$ 
when $\ga\in (\frac{d}{2},d)$.
For any $t >0$, using \eqref{Tdec} and Lemma \ref{lemma1} we get that for any $a<\ga$
\begin{equation}\label{eq1:lemma42}
\begin{split}
\bbP\Big(\cT_\beta \big( [2^k,2^{k+1}) \big) >t  \Big)
& \leq \bbP\Big( \sup_{s \in \cD , \ent(s) \le 2} \pi(s) > \gb^{-1} 2^{-\frac{dk}{2\ga}} (t +2^{k})\Big)\\ 
&\le  c\,    \gb^{a} \, 2^{k  \frac{da}{2\ga} }   \big( t+ 2^{k} \big)^{-a} \, .
\end{split}
\end{equation}
Then, for any $t\ge 1$, a union bound gives
\begin{equation}
\bbP\big( \cT_{\gb} >t\big) \le \sum_{k=0}^{\infty} \bbP\Big(\cT_\beta \big( [2^k,2^{k+1}) \big) >t  \Big) \le c'_a \gb^{a}\Bigg( t^{-a}  \sum_{k = 0}^{\log_2 t}  2^{ k \tfrac{da}{2\ga}}     +  \sum_{k>\log_2 t} 2^{ -a k ( 1-\frac{d}{2\ga} ) } \Bigg),
\end{equation}
where we split the series at $k=\log_2 t$ and  used that $t+2^k \ge t$ for $k\le \log_2 t$, and $t+2^k \ge 2^k$ for $k> \log_2 t$. Since $\ga >d/2$, we obtain that $\bbP( \cT_{\gb} >t) $ is smaller than
a constant (that depends on $a,\ga,\gb$) times
\[
t^{-a} t^{\frac{da}{2\ga}} + t^{-a ( 1-\frac{d}{2\ga} ) } = 2   t^{-a ( 1-\frac{d}{2\ga}) } \, .
\]
In particular, for any $a<\ga$ there is some constant $c' = c_{a,\ga,\gb}>0$ such that for any $t\ge 1$
\begin{equation}\label{bondTbeta}
\bbP\big( \cT_{\gb} >t\big) \le c' t^{- \frac{a}{\alpha}(\ga -\frac{d}{2})} \, .
\end{equation}
This concludes the proof that~$\cT_{\gb}$
is a.s.\ finite when $\ga\in (\frac{d}{2},d)$.
Also, since $a$ can be chosen arbitrarily close to $\ga$,
this proves that $\bbE[(\cT_{\gb})^{\kappa}]<\infty$ for any $\kappa <\ga -d/2$.\qed

\begin{remark}\rm
Notice that,
from the definition~\ref{def:varprob1} of $\hat \cT_{\gb}$,
we have for all $\gb \in(0,+\infty]$,
\begin{equation}\label{eq:compHatT-T_1}
\hat \cT_{\gb} \leq \sup_{s\in \cD, \hatent(s) <+\infty} \big\{  \pi(s)\big\} 
\leq \sup_{s\in \cD, \ent(s) \leq d/2 } \big\{  \pi(s)\big\} \,.
\end{equation}
For the second inequality we used that $\mathrm{J}_d(x) =+\infty$
if $\|x\|_1 > 1$, and in particular if $\|x\|^2 > 1$:
this shows that if the length of $s$ is larger than $1$ (in particular if $\ent(s) >d/2$), then $\hatent(s) =+\infty$.
Thanks to Lemma~\ref{lemma1} (up to a scaling), this shows that
for any $\ga\in (0,d)$, $\hat\cT_{\gb}$ is a.s.\ finite for all~$\gb>0$.
\end{remark}

\subsubsection*{Positivity of $\cT_{\beta}$ (and of $\hat \cT_\gb$)}

Let us consider the random set, for any $k\in \bbZ$
\begin{equation}
\cG_k:=\Big\{ (x,w) \in \cP \,,  \,   2^{2k-1}  \leq \frac{d}{2} \|x\|^2 \leq  2^{2k} \, , \,  \gb w > 2^{2k+1}
\Big\} \,.
\end{equation}
If $\cG_k \neq \emptyset$, let $(x,w)\in \cG_k$ and $s$ be the straight line from the origin to $x$, which verifies $\ent(s) = \frac d2 \|x\|^2$. Therefore, on the event $\cG_k \neq \emptyset$, we have
\begin{equation}
\cT_\beta \ge \gb w - \frac{d}{2}\|x\|^2 \ge 2^{2k}\, .
\end{equation}
Notice that $(|\cG_k|)_{k\in \bbZ}$ are independent Poisson random variables
with mean 
$$
\eta(|\cG_k|)= \int_{ 2^{k} \sqrt{1/d}   \leq  \| x\| \leq 2^{k} \sqrt{2/d} } \dd x \int_{\beta^{-1}2^{2k+1}}^\infty \ga w^{-\ga-1} \dd w 
  = c_{\ga,d,\beta}\, 2^{k(d-2\ga)} \, .
$$
Hence, we get that
$\lim_{k\to -\infty}\bbP(\cG_k \neq \emptyset) = 1$
if $\ga> d/2$:
this proves that $\bbP(\exists k , \cT_{\gb} \geq 2^{2k}) =1$ (by $\cT_\gb<2^{2k}\Rightarrow\cG_k=\emptyset$ and the independence of $|\cG_k|$),
so that $\cT_{\gb}>0$ almost surely, for any $\ga\in (\frac{d}{2},d)$.
\qed

\begin{remark}
\rm
The same argument can be used to prove that $\hat \cT_{\gb}>0$ almost surely for any $\ga\in (\frac{d}{2},\ga)$,
using that $\mathrm{J}_d(x) \sim \frac{d}{2} \|x\|^2$
as $\|x\| \downarrow 0$, see Lemma~\ref{lem:rate}.
This proves the last part of Proposition~\ref{prop:transition}, \textit{i.e.}\ that $\gb_c=0$ a.s.\ when $\ga \in (\frac d2, d)$.
\end{remark}

\subsubsection*{Infiniteness of $\cT_\gb$ for $\ga\in(0,\frac{d}{2}]$}
As far as the case $\ga\leq d/2$ is concerned,
since $(|\cG_k|)_{k\geq 1}$ are independent,
 the Borel-Cantelli lemma ensures that
a.s.\ $\cG_k \neq \emptyset$ for infinitely many $k\geq 1$.
This proves that almost surely, $\cT_{\gb} \geq 2^{2k}$ for infinitely many $k \geq 1$, that is $\cT_{\gb} =+\infty$ a.s.\ when $\ga \in (0,\frac d2]$. \qed

\begin{remark}\label{rem:compHatT-T_1}
\rm
Note that we have proven that $\hat \cT_{\gb} <+\infty$
a.s.\ for all $\gb>0$, for any $\ga\in (0,d)$. The reason is that 
in that case trajectories $s$ with $\hatent(s)<+\infty$ cannot exit the ball of radius~$1$; there is no contradiction with the fact that $\cT_{\gb} =+\infty$ when $\ga\leq d/2$.
\end{remark}

\subsubsection*{Case $\ga \in (0,\frac d2)$, proof of $\gb_c >0$}
We now prove the part $\ga\in (0,\frac d2)$ in Proposition~\ref{prop:transition}, \textit{i.e.}\ that  $\hat \cT_{\gb} =0$ for $\gb$ sufficiently small when $\ga\in (0,\frac d2)$, or equivalently $\gb_c>0$ a.s.
We proceed as in \cite[Section 6]{T14}.

Let $\hat s_\beta$ be a maximizer of $\hat\cT_\beta$,
which a.s.\ exists (and is unique) ---~the proof of the existence and uniqueness of a maximizer is identical to that given in~\cite[Sec.~4.6]{BT19b}. In the following preliminary result we show that for any $\ga \in (0,d)$, if $\beta$ is small, then $\hat s_\beta$ has a small entropy  (and is therefore confined in a small ball around the origin). We denote the length of $s$ by $\|s\|_\infty$.

\begin{lemma}\label{lemma:prel}
For any fixed $\alpha\in (0,d)$, $\mathbb{P}$-a.s., for any $\gep>0$ there exists $\beta_0=\beta_0(\gep)\in (0,1)$ such that $\hatent(\hat{s}_{\beta})<\gep$ for all $\beta<\beta_0$; in particular, since $\hatent(\hat{s}_{\beta}) \geq \frac 12 \|\hat s_{\gb}\|_{\infty}^{2}$
we get that $\|\hat s_{\gb}\|_{\infty} \leq \sqrt{2\gep}$.
\end{lemma}

\begin{proof} 
We proceed by contradiction.
Suppose that there exists $\gep>0$ and a sequence $(\gb_k)_{k\geq 1}$ with $\lim_{k\to\infty}\beta_k = 0$  such that $\hatent(\hat{s}_{\beta_k})\geq\gep$.
Using \eqref{eq:compHatT-T_1}, we have
\begin{equation}
0\le \hat \cT_{\gb_k} 
\leq  \sup_{s\in \cD, \ent(s) \leq d/2 } \big\{  \pi(s)\big\} -  \frac{\gep}{\gb_k} .
\end{equation}
By Lemma \ref{lemma1}, the right hand side becomes negative as $k$ gets large, which leads to a contradiction.
\end{proof}

We can now deduce that there exists $\beta_c>0$ such that $\hat\cT_\beta=0$ if $\beta<\beta_c$. 
Let $\gep >0$ and $\beta_0=\beta_0(\gep)\in (0,1)$ as in Lemma \ref{lemma:prel}, that is $\hatent( \hat{s}_{\beta})<\varepsilon$ for all $\beta<\beta_0$.
Moreover define $ \hat \gep_{\gb}:=\hatent( \hat{s}_{\beta}) \leq \gep$.
Then, since $\hatent(s)\ge \frac1d\ent(s)$, we get that if $\beta <\beta_0$
\begin{equation}
\label{upperT}
\hat\cT_\gb \le \sup_{s\in \cD, \, \hatent(s)\le \hat \gep_{\gb}} \big\{\pi(s)\big\} -  \frac{\hat{\gep}_\gb}{\beta} \leq \sup_{s\in \cD, \, \ent(s)\le d \hat \gep_{\gb}} \big\{\pi(s)\big\} -  \frac{ \hat \gep_{\gb}}{\beta}  \, .
\end{equation}
For any $\epsilon>0$, since $\bbP(\beta_{0}(\epsilon)>0)=1$, we can find some $\delta>0$ such that $\bbP(\beta_{0}(\epsilon)>\delta)>1-\epsilon$. Then with probability large than $1-\epsilon$,
	\begin{equation}\label{upperT1}
	\hat{\cT}_{\delta}
	\leq\sup\limits_{s\in \cD, \, \ent(s)\le d \hat \gep_{\gd}} \big\{\pi(s)\big\} -  \frac{\hat{\gep}_{\gd}}{\delta}\,.
	\end{equation}
Now,	 if $\hat{\epsilon}_{\gd}=0$, then $\beta_{c}\geq\delta>0$. Otherwise, if $\hat{\epsilon}_{\gd}>0$, we use the following union bound for $\bbP(\hat \cT_{\gd} >0 )$
\begin{equation}
\label{summation}
\begin{split}
	\bbP\Big(\sup\limits_{s, \, \ent(s)\le d \hat \gep_{\gd}} \big\{\pi(s)\big\} > \hat{\epsilon}_{\gd} / \delta \Big)
		 & =\sum\limits_{k=0}^{\infty}\bbP\Big(\sup\limits_{s, \, \ent(s)\le d \hat \gep_{\gd} }\big\{\pi(s)\big\} >  \hat{\epsilon}_{\gd} / \delta \, ,\,  \hat{\epsilon}_{\gd} \in (2^{-(k+1)},2^{-k}]\, \epsilon\Big)\\
	&\leq \sum\limits_{k=0}^{\infty}\bbP\Big(\sup\limits_{s, \, \ent(s)\le d 2^{-k}\epsilon} \big\{\pi(s)\big\} > 2^{-(k+1)} \gep /\gd\Big).
	\end{split}
\end{equation}
Now, setting $t=2^{-1-k(1-d/2\alpha)}\epsilon^{1-d/2\alpha}/\delta$ and $\epsilon_{k}=2^{-k}\epsilon$, we get that for any $k\geq 0$
\begin{equation*}
\bbP\Big(\sup\limits_{s, \, \ent(s)\le d2^{-k}\epsilon} \big\{\pi(s)\big\} > 2^{-(k+1)} \gep /\gd\Big)
 = \bbP\Big(\sup\limits_{s, \, \ent(s)\le d\epsilon_{k}}\big\{\pi(s)\big\}> t (\epsilon_k)^{d/(2\alpha)} \Big) \leq  c t^{-\ga/2} \, ,
\end{equation*}
where for the last inequality we used  the scaling $\phi_{a}$ above with $a=\sqrt{\epsilon_k}$
to get that $\sup_{ \ent(s)\leq d \gep_k}\pi (s)$ has the same law as $\epsilon_k^{d/(2\alpha)} \sup_{\ent(s)\leq d}\pi_{1}(s)$, together with Lemma~\ref{lemma1}.
With the definition of $t$, we end up with
\begin{equation}
\bbP\Big(\sup\limits_{s , \, \ent(s)\le d \hat \gep} \big\{\pi(s)\big\} \geq \frac{\hat{\epsilon}}{\delta}\Big)\leq 
c   \gd^{\ga/2}  \gep^{\frac{\alpha}{2}(\frac{d}{2\alpha}-1)} \sum_{k=0}^{\infty}  2^{k \frac{\ga}{2}(1-\frac{d}{2\ga})}
 \leq  c'_{\ga} \delta^{\ga/2} \epsilon^{ d-\ga/2} \, ,
\end{equation}
where we used that $\alpha<d/2$ to see that the sum is finite.
Finally, we have
\begin{equation}
\bbP(\beta_{c}>0)\geq1-\epsilon- c'_{\ga} \delta^{\ga/2}\epsilon^{d/4-\ga/2} \, .
\end{equation}
Letting first $\delta\to0$ and then $\epsilon\to0$, we conclude that $\bbP(\beta_{c}>0)=1$ for $\alpha\in(0,\frac{d}{2})$.
\qed

\subsection{Well-posedness of $\cX$}
\label{sec:finiteSum}

We now prove Proposition~\ref{prop:finiteSum}.
Notice that the condition $\ga> d/(d-2)$ ensures that $\ga>1$
so in particular $\mu := \bbE[\go_0]$ exists.
Let us denote $\bar \go_x := \go_x -\mu$ for simplicity. We prove that $\cX = \sum_{x\in \bbZ^d} \bar \go_x \bP(x\in \cR_{\infty})$ converges $\bbP$-a.s.,
using Kolmogorov's three-series theorem.

Before we turn to the control of the three series, let us stress that 
when $\ga>d/(d-2)$
we have
\begin{equation}
\label{eq:series1}
\sum_{x\in \bbZ^d} \bP(x\in \cR_{\infty})^{\ga} \leq \sum_{x\in \bbZ^d} \mathrm{G}(x)^{\ga}  <+\infty \, ,
\end{equation}
where $\mathrm{G}(x) = \sum_{n=0}^{\infty} \bP(S_n=x)$ is the Green function.
The fact that the sum is finite is due to the fact that
$\mathrm{G}(x) \sim  c \|x\|^{2-d}$ as $x\to\infty$ in dimension $d\geq 3$ and that $\ga >  d/(d-2)$.

We now control the three series in Kolmogorov's three-series theorem.

(i) The first series is
\[
\sum_{x\in \bbZ^d} \bbP\big( |\bar \go_x| \bP(x\in \cR_{\infty}) \geq 1  \big) \leq C  \sum_{x\in \bbZ^d}  \bP(x\in \cR_{\infty})^{\ga}  <+\infty \, ,
\]
where we used Assumption \eqref{eq:tailOmega} for the first inequailty.

(ii) The second series is
\[
 \sum_{x\in \bbZ^d} \Big| \bbE\Big[ \bar \go_x \bP(x\in \cR_{\infty}) \ind_{\{|\bar \go| \bP(x\in \cR_{\infty}) \leq 1\}} \Big]
\Big|
\leq C \sum_{x\in \bbZ^d}  \bP(x\in \cR_{\infty})^{\ga}  <+\infty \,,
\]
where we used that $\bbE[\bar \go]=0$ and $\bE[ \bar \go \ind_{\{|\go| >u\}}] \leq c u^{1-\ga}$ for any $u\geq 1$.

(iii) The third series is
\[
\sum_{x\in \bbZ^d} \bP(x\in \cR_{\infty})^2 \mathbb{V}\mathrm{ar}\big( \bar \go_x  \ind_{\{|\bar \go| \bP(x\in \cR_{\infty}) \leq 1\}} \big) 
\leq
\begin{cases}
 \sum_{x\in \bbZ^d} \bP(x\in \cR_{\infty})^{\ga}  <+\infty &\   \text{ if } \ga <2 \, ,\\
 \sum_{x\in \bbZ^d} \bP(x\in \cR_{\infty})^2 \log \frac{1}{\bP(x\in \cR_{\infty})}  <+\infty&\ \text{ if } \ga  \geq 2 \, .
\end{cases}
\]
where we used that $\bbE[\bar \go^2 \ind_{\{|\bar \go| \leq u\}} ]$
is bounded by a constant times $u^{2-\ga}$ if $\ga <2$, 
by a constant times $\log (1/u)$ if $\ga=2$ and by a constant if $\ga >2$.
The fact that the second series is finite comes from the fact that  
$\bP(x\in \cR_{\infty}) \leq \mathrm{G}(x)$ with $ \mathrm{G}(x) \sim c \|x\|^{2-d}$ as $x\to\infty$ and $2>\frac{d}{d-2}$ with $d\geq 5$.
This concludes the proof of Proposition~\ref{prop:finiteSum}.
\qed

\subsection{Well-posedness of $\cW_{\gb}$}\label{sec:wellposedWbeta}

In this part, we prove Propositions \ref{prop:finiteW} and \ref{prop:finiteW0}, starting
with the second one.
Let us stress that by \eqref{def:f}, $f$ is a radially decreasing function, with the following asymptotics:
\begin{align}
\label{asympfinfty}
f(x)&=O(e^{-c\|x\|^2}), \qquad\text{ as } x\to\infty \,.\\
f(x)&\sim c\, \begin{cases}
\log(1/\|x\|),&\quad\mbox{for}~d=2,\\
 \|x\|^{2-d}  ,&\quad\mbox{for}~d\geq 3,
\end{cases}
\qquad \text{ as } x\to 0\, .
\label{asympf0}
\end{align}

Recall that $\cW_0$ and $\cW_{\gb}$ are defined in~\eqref{def:W0} and~\eqref{eq:defWbeta} respectively.
 
\begin{proof}[Proof of Proposition~\ref{prop:finiteW0}]
We start with the case $\ga\in (0,1)$.
By \cite[Theorem 10.15]{K97}, $\cW_0$ is finite
if and only if the following integral is finite:
\begin{align*}
\int_{\bbR^d\times\bbR_+}\left( wf(x) \wedge 1\right) w^{-(\alpha+1)}\dd x \, \dd w
 & = \int_{\bbR^d} f(x) \int_{0}^{1/f(x)} w^{-\alpha} \dd x\, \dd w
 + \int_{\bbR^{d}}  \int_{1/f(x)}^{\infty} w^{-(1+\alpha)} \dd x\,  \dd w\\
 & = \Big(\frac{1}{1-\ga} +\frac{1}{\ga}  \Big) \int_{\bbR^d} f(x)^{\ga} \dd x\, ,
\end{align*}
where we used that $\ga\in(0,1)$ to compute both integrals on the first line.
This is always finite, thanks to the asymptotics~\eqref{asympfinfty}-\eqref{asympf0} of $f$.

For the case $\ga\in (1,2)$,thanks again to \cite[Theorem 10.15]{K97}, $\cW_0$ is finite
if and only if
\[
\int_{\bbR^d\times\bbR_+}\left( w^2f(x)^2\wedge w f(x)\right) w^{-(\alpha+1)}\dd x\, \dd w =  \Big(\frac{1}{2-\ga}  + \frac{1}{\ga-1} \Big) \int_{\bbR^d} f(x)^{\ga} \dd x 
\]
is finite (the calculation is the same as above, using $\ga\in (1,2)$).
Thanks to the asymptotics~\eqref{asympfinfty}-\eqref{asympf0} of $f$, this is finite if and only if $\ga< \frac{d}{d-2}$.
\end{proof}

\begin{proof}[Proof of Proposition \ref{prop:finiteW}]
Recall that we deal here with the case $\ga \in (\frac d2, 2)$ with dimensions $d=2,3$.
Recalling the definition~\eqref{eq:defWbeta} of $\cW_{\gb}$,
and using Proposition~\ref{prop:finiteW0},
we simply need to show that 
\begin{equation*}
\int_{\bbR^d\times\bbR_+}\frac{1}{\beta}\left(e^{\beta w}-1-\beta w\right)f(x)\cP(\dd x, \dd w) <+\infty \,.
\end{equation*}
To simplify notation, we only treat the case $\gb=1$;  the case $\gb>0$ is identical.
Now, by \cite[Theorem 10.15]{K97}, we need to show that
\begin{equation}\label{eq:proofproWgb_12}
\int_{\bbR^d\times\bbR_+}\left(\big(e^{w}-1- w\big)f(x)\wedge1\right) w^{-(1+\alpha)}\dd x\dd w <+\infty \,.
\end{equation}
Note that when $w\in [0,1]$ we have $e^{w}-1-w \leq cw^2$: hence, the restriction of the above intergral to $\bbR^d\times [0,1]$ is bounded by
\begin{align*}
\int_{\bbR^d\times [0,1]}  (w^2f(x) \wedge 1) w^{-(1+\alpha)}\dd x\dd w
&\leq  \int_{\bbR^d} f(x) \int_{0}^{1} w^{1-\alpha} \dd w \dd x = \frac{1}{2-\ga}  \int_{\bbR^d} f(x)   \dd x \, ,
\end{align*}
which is finite thanks to the asymptotics~\eqref{asympfinfty}-\eqref{asympf0} of $f$.

On the other hand, using that $e^{w}-1-w \leq e^{w}$, we also get that 
the restriction of the integral in~\eqref{eq:proofproWgb_12} to $\bbR^d\times (1,\infty)$ is bounded by
\begin{equation}
 \label{Wtwointegrals}
 \begin{split}
\int_{\bbR^d\times (1,\infty)} &  (e^{w} f(x) \wedge 1) w^{-(1+\alpha)}\dd x\, \dd w \\
&\leq C \int_{1}^{\infty}  \Big(\int_{\|x\|^2 > \frac{2w}{c}}  e^{w} e^{-c \|x\|^2} \dd x \Big) w^{-(1+\alpha)} \dd w + \int_{1}^{\infty}  \Big(\int_{\|x\|^2\leq  \frac{2w}{c}}  \dd x \Big) w^{-(1+\alpha)} \dd w \\
 & \leq C'  \int_{1}^{\infty} w^{-(1+\alpha)} \dd w  + C'' \int_1^{\infty} w^{\frac{d}{2}-(\ga+1)} \dd w \, .
\end{split}
\end{equation}
For the first inequality we used that there are constants $C,c$ such that $f(x) \leq Ce^{-c\|x\|^2}$ for all $\|x\|^2 \geq 2/c$. 
For the second inequality, we used that: (i) in the first term we have $e^{w} e^{-c \|x\|^2}  \leq e^{-\frac{c}{2} \|x\|^2}$, which is integrable on $\bbR^d$;
(ii) in the second term we bounded the volume of the ball of radius $\sqrt{2w/c}$ by a constant times $w^{d/2}$.
To conclude, we get that the first integral in~\eqref{Wtwointegrals} is always finite, while we use that $\ga>d/2$ to get that the second integral is finite. 
\end{proof}


\section{Discrete approximation of the variational problems}
\label{sec:varprofCont}

In this section, we present the main auxiliary results that we use to prove Theorems~\ref{thm:A}-\ref{thm:B}; they are also of independent interest. They extend  analogous results obtained in \cite{BT19b, BT19c} to the $d$-dimensional case. 
We start with the definition of  important quantities  that are used throughout the rest of the article, then we state the convergence result.

\subsection{Discrete energy-entropy variational problems}

Let us introduce the discrete analogue of the variational problem 
in~\eqref{eq:VarPr}.
It will appear naturally when considering the log-partition function
with trajectories restricted to staying in a ball of radius $r_N \ll N$; when the scale is $r_N =N$ some other entropy functional need to be considered, analogously to~\eqref{def:entropyCont2} ---~we will treat that case afterwards.

\subsubsection*{Discrete entropy of a set of points.}
For any set $\Delta=(x_{1},\dots,x_{k})\in (\bbR^{d})^k$ of ordered distinct points (with a slight abuse of notation, we sometimes interpret $\Delta$ as a subset of $\R^d$), we define the \textit{entropy} related to the set $\Delta$ by
\begin{equation}\label{eq:entropydef}
\ent(\Delta) =
\frac{d}{2}\bigg(\sum\limits_{i=1}^{k}\|x_{i}-x_{i-1}\|\bigg)^{2},
\end{equation}
where $x_{0}=0$ by convention.
Note that the entropy of a set $\Delta$ depends on the order of
the points of $\Delta$. 
Note also that if we
consider $s$ the linear interpolation of the points of $\Delta$,
we recover the continuous entropy~\eqref{def:entropyCont}.

We also introduce the entropy for a $N$-step random walk
to visit $\Delta$ (see Lemma~\ref{AppendixL2}):
\begin{equation}
\label{def:entropyN}
\ent_N(\Delta) :=
\inf_{ 0 =t_0 < t_1 < \cdots < t_k \leq N }  \sum_{i=1}^k \frac{d}{2}  \frac{\| x_i-x_{i-1}\|^2}{t_i-t_{i-1}} 
=  \frac{1}{N} \ent(\Delta) \, ,
\end{equation}
The second identity is due to the fact that the infimum
is attained for $t_i-t_{i-1} = \frac{\| x_i-x_{i-1}\|}{ \sum_{i=1}^k \|x_i-x_{i-1}\|} N$.

\subsubsection*{Energy (and truncated energy) of a set of points.}
For $r>0$, let 
\begin{equation*}
\Lambda_{r}=\{x\in\bbZ^{d}: \|x\|\leq r\}
\end{equation*}
 be the ball in $\bbZ^{d}$ with radius $r$ centered at the origin.
We can write the random environment $\omega$ in $\Lambda_{r}$ using its \textit{ordered statistic}: we let $M_{i}^{(r)}$ be the $i$-th largest value of $(\omega_{x})_{x\in\Lambda_{r}}$ and $Y_{i}^{(r)}$ its position. Then $(Y_{i}^{(r)})_{i=1}^{|\Lambda_{r}|}$ is a random permutation of $\Lambda_{r}$ and
\begin{equation}\label{eq:ordstatdiscr}
(\omega_{x},x)_{x\in\Lambda_{r}}=(M_{i}^{(r)},Y_{i}^{(r)})_{i=1}^{|\Lambda_{r}|}.
\end{equation}

The energy collected by a set $\Delta  =(x_1,\ldots, x_k) \subset  \Lambda_{r}$ and its contribution from the $\ell$ largest weights ($1\leq \ell\leq |\Lambda_r|$) are defined by
\begin{equation}
\label{def:Omegar}
\Omega_{r}(\Delta):=\sum\limits_{i=1}^{|\Lambda_{r}|}M_{i}^{(r)}\mathbbm{1}_{\{Y_{i}^{(r)}\in\Delta\}},\qquad 
\Omega_{r}^{(\ell)}(\Delta):=\sum\limits_{i=1}^{\ell}M_{i}^{(r)}\mathbbm{1}_{\{Y_{i}^{(r)}\in\Delta\}} \,,
\end{equation}
where $\{y\in\Delta\}$ means that there is some $1\leq i \leq k$
such that $x_i=y$, 
with a slight abuse of notation.
Let us also set $\Omega_{r}^{(>\ell)}(\Delta):=\Omega_{r}(\Delta)-\Omega_{r}^{(\ell)}(\Delta)$.

\subsubsection*{Discrete variational problem.}
Let us now define the discrete variational problem
\begin{equation}
\label{def:intDiscrVarProb1}
\cT_{N,r}^{\beta}:=\max\limits_{\Delta\subset\Lambda_{r}} \big\{\beta\Omega_{r}(\Delta)-\ent_N(\Delta) \big\} \, .
\end{equation}
We also define the analogous variational problems
restricted to the $\ell$ largest weights and beyond the $\ell$-th largest weight by
\begin{align}
\label{def:intDiscrVarProb2}
\cT_{N,r}^{\beta,(\ell)}&:=\max\limits_{\Delta\subset\Lambda_{r}}  \big \{\beta\Omega_{r}^{(\ell)}(\Delta)-\ent_N(\Delta)\},\\\label{def:intDiscrVarProb3}
\cT_{N,r}^{\beta,(>\ell)}&:=\max\limits_{\Delta\subset\Lambda_{r}}\big\{\beta\Omega_{r}^{(>\ell)}(\Delta)-\ent_N(\Delta)\big\}.
\end{align}

The following result gives an explicit bound on the tail of the (discrete) variational problem.
It will be useful to prove that small weights have a negligible
contribution to the variational problems, 
uniformly in $N$.
Its proof relies on E-LPP estimates and is very similar to~\cite[Prop.~2.6]{BT19b}; its proof is included in Appendix~\ref{sec:ELPP}.

\begin{proposition}\label{tail:disT}
The following hold true:
\begin{itemize}
\item There exists a constant $c$, such that for any $N\geq1$, $r\geq1$, $\beta>0$, $1\leq\ell\leq|\Lambda_{r}|$ and $t>1$,
\begin{equation}\label{tail:T<ell}
\bbP\left(\cT_{N,r}^{\beta,(\ell)}\geq tN\times(\beta r^{\frac{d}{\ga}-1})^{2}\right) \leq c \, t^{-\frac{\alpha d}{2(\alpha+d)}}.
\end{equation}

\item There exists a constant $c$, such that for any $N\geq1$, $r\geq1$, $\beta>0$, $1\leq\ell\leq|\Lambda_{r}|$ and $t>1$,
\begin{equation}\label{tail:T>ell}
\bbP\left(\cT_{N,r}^{\beta,(>\ell)}\geq tN\times(\beta r^{\frac{d}{\ga}-1}  \ell^{\frac1d -\frac1\ga} )^{2}\right)\leq c\, t^{- \frac{\alpha\ell d}{2(\alpha\ell+d)}}.
\end{equation}
\end{itemize}
\end{proposition}

\subsubsection*{Adaptation for trajectories at scale $N$.}

We also define entropy arising when 
considering trajectories at a scale $N$ instead of $N^{\xi}$.
For $\Delta = (x_1,\ldots, x_k) \in \mathbb{R}^d$ and $N \geq 1$,
let us define the $N$-step entropy
\begin{equation}
\label{eq:hatentropydef}
\hatent_N(\Delta) := \inf_{ 0 =t_0 < t_1 < \cdots < t_k \leq N }  \sum_{i=1}^{k}  (t_i-t_{i-1})  \mathrm{J} _d\Big(  \frac{x_i-x_{i-1}}{ t_i-t_{i-1} } \Big) \, ,
\end{equation}
where $\mathrm{J}_d$ is the large deviation rate function for the random walk, see Lemma~\ref{lem:rate}.
Also, we let $\hatent (\Delta) := \hatent_1 (\Delta)$.
Note that if we
consider $s$ the linear interpolation of the points of $\Delta$,
$\hatent (\Delta)$ recovers the continuous entropy~\eqref{def:entropyCont}.

Let us stress that we have 
$\hatent_N(\Delta) \geq \frac1d \ent_N(\Delta)$ for all $N\geq 1$,
since $\mathrm{J}_d (x) \geq \frac{d}{2} \|x\|^2$, cf.\ Lemma~\ref{lem:rate}.
Then, we can define the analogous variational problems as
in~\eqref{def:intDiscrVarProb1}, \eqref{def:intDiscrVarProb2}
and~\eqref{def:intDiscrVarProb3},
replacing the entropy $\ent_N(\Delta)$ with $\hatent(\Delta)$.
One difference here is that 
we only consider trajectories at scale $N$: 
we therefore take $r=N$
(notice also that $\hatent_N(\Delta) = +\infty$ if $\Delta$ 
has one point outside~$\Lambda_N$,
recalling that $\mathrm{J}_d(x) = +\infty$ if $\|x\|_1>1$).
We define
\begin{align}
\label{def:intDiscrVarProb1-hat}
\hat \cT_{N}^{\beta}&:=\max\limits_{\Delta\subset\Lambda_{N}} \big\{\beta\Omega_{N}(\Delta)-\hatent_N(\Delta)\big\},\\\label{def:intDiscrVarProb2-hat}
\hat \cT_{N}^{\beta,(\ell)}&:=\max\limits_{\Delta\subset\Lambda_{N}}\big\{\beta\Omega_{N}^{(\ell)}(\Delta)-\hatent_N(\Delta)\big\},\\\label{def:intDiscrVarProb3-hat}
\hat \cT_{N}^{\beta,(>\ell)}&:=\max\limits_{\Delta\subset\Lambda_{N}}\big\{\beta\Omega_{r}^{(>\ell)}(\Delta)-\hatent_N(\Delta\big)\}.
\end{align}
Using that $\hatent_N(\Delta) \geq \frac{1}{d} \ent_N(\Delta)$,
we get that $\hat \cT_{N}^{\beta} \leq  \frac1d \cT_{N,N}^{d\gb}$,
and similarly for $\hat \cT_{N}^{\beta,(\ell)}$ and $\hat \cT_{N}^{\beta,(>\ell)}$.
Proposition~\ref{tail:disT} therefore remains valid
for the variational problems~\eqref{def:intDiscrVarProb2-hat}-\eqref{def:intDiscrVarProb3-hat}, up to a change in the constants.

\subsection{Convergence of the discrete variational problem to the continuous one}

In this section, we prove that the discrete variational problems~\eqref{def:intDiscrVarProb1} and \eqref{def:intDiscrVarProb1-hat},
when properly rescaled, converge to their continuous counterparts
\eqref{eq:VarPr} and~\eqref{def:varprob1}.

Recall the definition~\eqref{def:Dq}
of  $\cD_q$, the set of paths staying in $\mathbf{\Lambda}_q$
(the ball of radius $q$), 
and the definition~\eqref{eq:energy2} of  $\pi_q^{(\ell)}(s)$,
the contribution to $\pi(s)$ from the $\ell$ largest weights in $\mathbf{\Lambda}_q$.
We have the following convergences

\begin{proposition}\label{thm2}
Let $\ga \in (\frac{d}{2}, d)$, let $q>0$ and $\xi >0$.  
Let $\gb_N$ 
be such that $\lim_{N\to\infty} N^{d\xi/\ga-(2\xi-1)} \beta_{N}   = \beta \in (0,+\infty)$.
Then, for any $\ell\in \mathbb N$ we have that 
\begin{equation}
\label{con1}
\frac{1}{N^{2\xi-1}} \cT_{N,qN^{\xi}}^{\gb_N,(\ell)} \xrightarrow[N\to\infty]{(d)} \cT_{\beta,q}^{(\ell)} := \sup\limits_{s\in \cD_q}\big\{\beta\pi_{q}^{(\ell)}(s)-\ent(s)\big\} \, .
\end{equation}
We also have that, for any fixed $q>0$,
\begin{equation}
\label{con2}
\lim_{\ell \to +\infty} \cT_{\beta,q}^{(\ell)} = \cT_{\beta,q} 
:= \sup\limits_{s\in \cD_q}\big\{\beta\pi (s)-\ent(s)\big\}   \qquad \text{a.s.}
\end{equation}
Finally,  by monotonicity, we have $\lim_{q\to\infty} \cT_{\beta,q} = \cT_{\beta}$ almost surely.
\end{proposition}

%

We have the analogous statement for the discrete variational
problem~\eqref{def:intDiscrVarProb2-hat}.
\begin{proposition}\label{thm3}
Let $\ga \in (\frac{d}{2}, d)$.
We let $\gb_N$ 
such that $\lim_{N\to\infty} N^{ \frac{d}{\ga} -1} \beta_{N}  =  \beta \in (0,+\infty]$.
Then, for any $\ell\in \mathbb N$ we have that 
\begin{align}
\label{con3}
\frac{1}{\gb_N N^{d/\ga}} \hat \cT_{N}^{\gb_N,(\ell)} \xrightarrow[N\to\infty]{(d)} \hat \cT_{\beta}^{(\ell)}
:= \sup_{s\in \cD_1, \hatent(s) <+\infty} 
\big\{ \pi^{(\ell)}_1(s) - \tfrac{1}{ \gb} \hatent(s) \big\} \, .
\end{align}
We also have that
$ \lim_{\ell \to+\infty} \hat \cT_{\beta}^{(\ell)} = \hat \cT_{\beta}$ almost surely.
\end{proposition}
Notice that the restriction $s \in \cD_1$  in \eqref{con3} 
is harmless
since we have $\hatent(s) = +\infty$ for all $s\notin \cD_1$.

\begin{proof}[Proof of Propositions~\ref{thm2} and~\ref{thm3}]
Let us define $r_N = q N^{\xi}$ for simplicity.
We label the points $\Lambda_{r_N}$ according to the ordered statistics $(M_{i}^{(r_N)}, Y_{i}^{(r_N)})_{i=1}^{|\Lambda_{r_N}|}$, see \eqref{eq:ordstatdiscr}. 
In the same way, in the ball $\mathbf\Lambda_q$ we label the points of $\cP$ according to the ordered statistic $(\bM_{i}^{(q)}, \bY_{i}^{(q)})_{i\geq1}$, cf. \eqref{eq:energy3}.
Then, by the Skorokhod representation theorem (cf. \cite[Section 9.4]{D81}, also in the same spirit of Remark \ref{rem:Skorokhod}), we have that 
for any fixed $\ell \in \N$,
\begin{align*}
\frac{1}{N^{\xi}} (Y_{1}^{(r_N)},\dots, Y_{\ell}^{(r_N)})&\xrightarrow[N\to\infty]{a.s.}( \bY_{1}^{(q)},\dots, \bY_{\ell}^{(q)}),\\
\frac{1}{ N^{\xi d/\ga}} (M_{1}^{(r_N)},\dots, M_{\ell}^{(r_N)})&\xrightarrow[N\to\infty]{a.s.}( \bM_{1}^{(q)},\dots, \bM_{\ell}^{(q)}) \, .
\end{align*} 
Then, 
for each $\Delta_N = (Y_{i_1}^{(r_N)}, \ldots, Y_{i_k}^{(r_N)})$,
thanks to 
the continuous mapping theorem, we get that 
$N^{-(2\xi-1)} \gb_N \Omega_{r_N}^{(\ell)} (\Delta_N)$
converges to $ \gb \pi_q^{\ell}( s_\Delta)$, where $s_{\Delta}$
is the linear interpolation of $\Delta := ( \bY_{i_1}^{q}, \ldots, \bY_{i_k}^{q})$ (see the definitions~\eqref{def:Omegar} of $\Omega_{r_N}^{(\ell)}$ and \eqref{eq:energy2} of $\pi_q^{(\ell)}(s)$).
We also obtain that $N^{-(2\xi-1)} \ent_N (\Delta_N)$
converges to $\ent(s_{\Delta})$.
Since the maxima  in~\eqref{def:intDiscrVarProb2}
and~\eqref{con1} are over finitely many terms (in fact $2^{\ell}$),
we get~\eqref{con1}.
The convergence~\eqref{con3} in Theorem~\ref{thm3} follows from exactly the same argument taking $r_N=N$ (here $q=1$),  and using that $N^{-1} \hatent_N(\Delta_N)$ converges
to $\hatent(\Delta)$.

Finally, \eqref{con2}  simply follows 
from the monotonicity in $\ell$, 
and the fact that $\cT_{\gb,q} \leq \cT_{\gb}$ which is finite almost surely. The limits $\lim_{q \to\infty} \cT_{\gb,q} = \cT_{\gb}$ 
and $\lim_{\ell\to\infty} \hat \cT_{\gb}^{(\ell)} = \hat \cT_{\gb}$
also follow by monotonicity.
\end{proof}

\section{Region~A: proof of Theorem~\ref{thm:A} and Proposition~\ref{prop:fluctuationsN}}
\label{sec:regionA}

In this section, we prove Theorem \ref{thm:A}.
First of all, notice that for any $h\in \R$ we have
\begin{equation*}
e^{-N\beta_{N}|h|}Z_{N,\gb_N}^{\go,h=0} \leq Z_{N,\gb_N}^{\go,h}  \leq e^{N\beta_{N}|h|} Z_{N,\gb_N}^{\go,h=0} .
\end{equation*}
Hence, we get
\begin{equation*}
-\frac{|h|}{N^{\frac{d-\alpha}{\alpha}}}+\frac{1}{\beta_{N}N^{\frac{d}{\alpha}}}\log Z_{N,\gb_N}^{\go,h=0} \leq\frac{1}{\beta_{N}N^{\frac{d}{\alpha}}}\log Z_{N,\gb_N}^{\go,h}  \leq\frac{|h|}{N^{\frac{d-\alpha}{\alpha}}}+\frac{1}{\beta_{N}N^{\frac{d}{\alpha}}}\log Z_{N,\gb_N}^{\go,h =0} \, .
\end{equation*}
Since $d>\ga$, the terms  $|h| N^{-(d-\ga)/\ga}$ go to $0$ and we therefore only need to prove the result for $Z_{N,\gb_N}^{\go,h=0}$.
In the rest of this section, we will write $Z_{N,\gb_N} :=Z_{N,\gb_N}^{\go,h=0}$ for simplicity.
We only treat the case where $\lim_{N\to\infty} \beta_{N}N^{(d-\ga)/\ga}  = \gb\in (0,+\infty)$ for notational clarity;
the case where $\gb = +\infty$ follows in a similar manner.

\subsection{Convergence of the rescaled log-partition function}

Let us prove Theorem~\ref{thm:A}.
Our strategy is based on three steps: \textit{Step 1}. most of the contribution to $Z_{N,\gb_N}$ comes from some large disorder weights; \textit{Step 2}. the log-partition function, when restricted to finitely many weights and rescaled by $\gb_N N^{d/\ga}$, has a weak limit; \textit{Step 3}. combine the first two steps to conclude the proof of the convergence.


Let us introduce some notation. Recall the order statistics defined in~\eqref{eq:ordstatdiscr}.
 For any $L\in\bbN$, we define, for $\rho \in \bbR_+$
\begin{align}
Z_{N,\rho\beta_{N}}^{(L)}&:=\bE\bigg[\exp\bigg(\rho\beta_{N} \sum\limits_{i=1}^{L}M_{i}^{(N)}\mathbbm{1}_{\{Y_{i}^{(N)}\in\cR_{N}\}}\bigg)\bigg],\label{ZL}\\
Z_{N,\rho\beta_{N}}^{(>L)}&:=\bE\bigg[\exp\bigg(\rho\beta_{N} \sum\limits_{i=L+1}^{|\Lambda_N|}M_{i}^{(N)}\mathbbm{1}_{\{Y_{n}^{(N)}\in\cR_{N}\}}\bigg)\bigg].\label{ZLell}
\end{align}

\smallskip
\noindent
{\bf Step 1. }
For any fixed $L$ and $\eta\in(0,1)$, by H\"{o}lder's inequality, we have
\begin{equation*}
Z_{N,\gb_N}\leq\left(Z_{N,(1+\eta)\beta_{N}}^{(L)}\right)^{\frac{1}{1+\eta}}\left(Z_{N,(1+\eta^{-1})\beta_{N}}^{(>L)}\right)^{\frac{\eta}{1+\eta}},
\end{equation*}
so that for any fixed $L<N$, since the disorder is positive, we get
\begin{equation}\label{ineq:logZ}
\log Z_{N,\beta_{N}}^{(L)}\leq \log Z_{N,\gb_N}\leq\frac{1}{1+\eta}\log Z_{N,(1+\eta)\beta_{N}}^{(L)}+\frac{\eta}{1+\eta}\log Z_{N,(1+\eta^{-1})\beta_{N}}^{(>L)} \, .
\end{equation}

We now show that the last term in \eqref{ineq:logZ} 
can be made arbitrarily small compared with $\gb_N N^{d/\ga}$ by choosing $L$ large:
this is the content of the following lemma.
\begin{lemma}
\label{ZLellvanish+}
Assume that $\lim_{N\to\infty} \gb_N N^{\frac{d}{\ga} -1}  =\gb>0$.
For any $\gep>0$ and any $\rho>0$, there are some $L_0>0$ and some $N_0$ such that for any $L>L_0$ and $N\geq N_0$ we have
\[
\bbP\Big(\frac{1}{\beta_{N}N^{d/\ga}} \log Z_{N,\rho\beta_{N}}^{(>L)}>\epsilon\Big) \leq \gep \, .
\]
\end{lemma}

\begin{proof}
Notice that for any realization $\cR_N$ of the range,
we have
\[
\sum_{i=L+1}^{|\Lambda_N|} M_i^{N} \ind_{\{Y_i^{N} \in \cR_N\}} \leq \sup_ {\Delta \subset \Lambda_N,\,  \ent_N(\Delta) \leq  \frac{d}{2} N } \gO_N^{(>L)} (\Delta) \, ,
\]
where we recall the notation~\eqref{def:Omegar}
for $\gO_N^{(>L)}(\Delta)$.
We have used here that any set of points $(x_1,\ldots, x_k)$
visited by the random walk before time $N$
must verify $\sum_{i=1}^k \|x_i-x_{i-1}\| \leq N$
and hence have a $N$-step entropy
$\ent_N(\Delta) \leq \frac{d}{2} N$ (recall the definitions \eqref{eq:entropydef}-\eqref{def:entropyN}). 
We therefore get that
\[
 \frac{1}{\gb_N N^{d/\ga}} \log Z_{N,\rho\beta_{N}}^{(>L)} 
\leq  \frac{\rho}{N^{d/\ga}}  \sup_{\Delta\,, \ent_N(\Delta) \leq  \frac{d}{2} N } \gO_N^{(>L)} (\Delta) \, .
\]
Now, we simply have to use Lemma~\ref{lem:tail>ell} (noting that $\ent_N(\Delta) = N^{-1} \ent(\Delta)$),
which states that for any $L$ and any $t>1$
\[
\bbP\bigg( \sup_{\Delta\,, \ent(\Delta) \leq  \frac{d}{2} N^2 } \gO_N^{(>L)} (\Delta)  \geq t L^{ \frac{1}{d} - \frac{1}{\ga}}  N^{\frac{d}{\ga}}  \bigg) \leq t^{- \frac{\alpha L d}{\alpha L+d} }\, .
\]
Fixing $t$ large enough so that the upper bound is smaller than $\gep$ (uniformly for $L\geq 1$) and then
choosing $L$ large enough so that $t  L^{ \frac{1}{d} - \frac{1}{\ga}}  \leq \gep/\rho$, 
we obtain the conclusion of Lemma~\ref{ZLellvanish+}.
\end{proof}

\smallskip
\noindent
{\bf Step 2. }
Once the number of weights is fixed, we can prove
the following convergence in distribution.
\begin{proposition}\label{logZL}
For any positive integer $L$ and real number $\rho>0$, we have the following convergence in distribution
\begin{equation*}
\frac{1}{\beta_{N}N^{d/\ga}} \log Z_{N,\rho\beta_{N}}^{(L)}\overset{(d)}{\longrightarrow}\hat{\cT}_{\rho\beta}^{(L)} \, ,
\end{equation*}
where $\hat{\cT}_{\beta}^{(L)} $ is defined in~\eqref{con3}.
\end{proposition}
\begin{proof}
Let us denote $\Upsilon_{L}:=\{Y_{1}^{(N)},\cdots,Y_{L}^{(N)}\}$ the set of the (random) positions of the $L$ largest weights
in $\Lambda_N$.
Then, we can write
\begin{align*}
Z_{N,\rho\beta_{N}}^{(L)} = \sum_{\Delta \subset \Upsilon_L}   \exp\left(\rho\beta_{N}\Omega_{N,N}^{(L)}(\Delta)\right) \bP\left(\cR_{N}\cap\Upsilon_{L}=\Delta\right)\, ,
\end{align*}
where we used the convention that the points of $\Delta$ are ordered, \textit{i.e.} $\Delta=(x_1,\ldots, x_k)$, 
and used some abuse of notation in writing $\Delta\subset \gU_L$.
Also, let us note that we use the notation $\cR_{N}\cap\Upsilon_{L}=\Delta$ to state that the points in $\Delta$ are visited in the correct order by the random walk, and that no other point in $\gU_L$ is visited.
We then have the following lower and upper bounds:
\begin{align*}
 Z_{N,\rho\beta_{N}}^{(L)} 
& \geq \exp\bigg(  \sup_{\Delta\subset \gU_{L} } \Big\{  \rho \gb_N \Omega_{N,N}^{(L)}  + \log \bP\left(\cR_{N}\cap\Upsilon_{L}=\Delta\right)\Big\}  \bigg) \,,  \\
 Z_{N,\rho\beta_{N}}^{(L)} 
& \leq 2^L L! \exp\bigg(  \sup_{\Delta\subset \gU_{L} } \Big\{  \rho \gb_N \Omega_{N,N}^{(L)}  + \log \bP\left( \Delta \subset \cR_{N}  \right)\Big\} \bigg) \, ,
\end{align*}
where again we used the conventions
that $\Delta = (x_1,\ldots, x_k)$ are ordered points in $\gU_L$ and that we denoted $\Delta \subset \cR_N$ the fact  that the points in $\Delta$ are visited in the correct order by the random walk.
Recalling the definition~\eqref{def:intDiscrVarProb2}
of the discrete variational problem~$\hat \cT_N^{\gb,(\ell)}$,
we can therefore rewrite
\begin{equation}
\label{writevaritaional}
\begin{split}
\log  Z_{N,\rho\beta_{N}}^{(L)}  - \hat \cT_N^{\rho\gb_N,(L)}  
& \geq  \inf_{\Delta \subset \gU_L}  \big\{ \hatent_N(\Delta) + \log \bP\left(   \cR_{N} \cap \gU_L =\Delta  \right) \big\} \,, \\
\log  Z_{N,\rho\beta_{N}}^{(L)}  - \hat \cT_N^{\rho\gb_N,(L)} & \leq  \log(2^L L!) + \sup_{\Delta \subset \gU_L}  \big\{ \hatent_N(\Delta) +\log \bP\left( \Delta \subset \cR_{N}  \right) \big\}\, .
\end{split}
\end{equation}
In view of the convergence in Proposition~\ref{thm3},
we therefore only have to prove that the upper and lower bounds are negligible.
More precisely, since there are only finitely many terms in the infimum (or in the supremum) and since $\gb_N N^{d/\ga} \sim \gb N$ with $\gb>0$, we only have to prove that
for any distinct indices $i_1,\ldots, i_k$, the (random) subset $\Delta  = (Y_{i_1}^{(N)}, \ldots, Y_{i_k}^{(N)})$ of $\gU_L$ verifies
\begin{equation}
\label{convPent}
\frac{1}{N}\big| \hatent_N(\Delta) + \log \bP\left(   \cR_{N} \cap \gU_L =\Delta  \right) \big|  \xrightarrow{\bbP} 0
\quad \text{ and } \quad
\frac{1}{N} \big| \hatent_N(\Delta) + \log \bP\left(   \Delta \subset \cR_{N} \right) \big|  \xrightarrow{\bbP} 0 \, .
\end{equation}
Note that for any fixed $L$,
for any $\epsilon>0$ we have
\begin{equation*}
\bbP\Big(\exists\, 1\leq i,j \leq L,\ \mbox{s.t.}~\|Y_i^{(N)}-Y_j^{(N)}\|<\epsilon N \Big)\leq\binom{L}{2}\bbP\left(\|Y_{1}^{(N)}-Y_{2}^{(N)}\|<\epsilon N\right)\leq C_{L}\epsilon^{d}.
\end{equation*}
The points in $\gU_L$ are therefore at distance at least $\gep N$
from each other, with $\bbP$-probability larger than $1-c \gep^d$.
Thanks to Lemma~\ref{AppendixL1bis} ---~and to the definition~\eqref{eq:hatentropydef} of $\hatent_N(\Delta)$~---,
this shows that 
\begin{equation}
\label{approxentropy}
\frac{1}{N} \big| \hatent_N(\Delta) + \log \bP\left(   \Delta \subset \cR_{N} \right) \big|  \to 0
\end{equation}
with high $\bbP$-probability.
Notice also that $\bP( \cR_{N} \cap \gU_L =\Delta) =  \bP( \Delta\subset \cR_N, \cR_{N} \cap (\gU_L \setminus \Delta ) =\emptyset )$
and that $\bP(\cR_{N} \cap (\gU_L \setminus \Delta  ) =\emptyset  \mid \Delta\subset \cR_N) \to 1$ with high $\bbP$-probability 
(using Lemma \ref{AppendixL2} and that the points in $\gU_L$ are distant by at least $\gep N$
with high $\bbP$-probability).
From \eqref{approxentropy}, we therefore also get that
\[
\frac{1}{N}\big| \hatent_N(\Delta) + \log \bP\left(   \cR_{N} \cap \gU_L =\Delta  \right) \big|  \to 0 \, ,
\]
with high $\bbP$-probability.
This proves~\eqref{convPent} and concludes the proof.
\end{proof}

\smallskip
\noindent
{\bf Step 3. Conclusion. }
Note that, as stressed in Proposition~\ref{thm3}, 
we have $\lim_{L\to\infty}\hat{\cT}_{\rho\beta}^{(L)} =\hat{\cT}_{\rho\beta}$ almost surely.
Due to the continuity in $\gb$ (see Section~\ref{subsec:variat}) we  also have that $\lim_{\rho\to 1}\hat{\cT}_{\rho\beta} = \hat{\cT}_{\beta}$ a.s.

The conclusion of the proof is then just a matter of combining the different steps in the correct order. For any fixed $\gep>0$,
(i) we choose $\eta =\eta_{\gep} \in (0,1)$ sufficiently small (in \eqref{ineq:logZ}) and then $L_0>0$ so that
for any $L\geq L_0$ both $\hat{\cT}^{(L)}_{\beta}$
and $ \frac{1}{1+\eta} \hat{\cT}^{(L)}_{ (1+\eta)\beta}$ are at distance less than $\gep$ from $\hat \cT_{\gb}$ with $\bbP$-probability larger than $1-\gep$;
(ii) we fix some $L$ large enough so that the conclusion of Lemma~\ref{ZLellvanish+} holds, with $\rho=1+\eta^{-1}$.
The conclusion then follows by applying Proposition~\ref{logZL}
to both sides of the inequality~\eqref{ineq:logZ}
(by Skorokhod's representation theorem, cf.\ Remark \ref{rem:Skorokhod}, we can work as if the convergence in~Proposition~\ref{logZL}  is an almost sure convergence).
\qed

\subsection{Transversal fluctuations: proof of Proposition~\ref{prop:fluctuationsN}}
\label{sec:transversalA}

Since the transversal fluctuations are at most $N$, we therefore only have to prove that for any $\gep>0$ there exists some $\eta>0$ such that for all $N$ large enough
\[
\bP_{N,\gb_N}^{\go,h}\Big(  \max_{1\leq n \leq N} \|S_n\| \leq \eta\, N\Big) \leq  \gep  \,,
\]
with large $\bbP$-probability conditionally on $\hat\cT_{\gb}>0$.

The same proof as above can easily be adapted to show that
we have the following convergence in distribution
(it can be upgraded to an almost sure convergence by Skorokhod's representation theorem)
\begin{equation}
\label{convergenceeta}
\frac{1}{ \gb_N N^{d/\ga}} \log Z_{N,\gb_N}^{\go,h}\big( \max_{1\leq n \leq N} \|S_n\| \leq \eta\, N \big)  \stackrel{(d)}{ \longrightarrow} \hat\cT_{\gb,\eta} := \sup_{s\in \cD_{\eta}, \hatent(s) <+\infty} 
\big\{ \pi(s) - \tfrac{1}{ \gb} \hatent(s) \big\} \,, 
\end{equation}
where we set $Z_{N,\gb_N}^{\go,h} (\cA) = \bE\big[\exp\big(\sum_{x\in\cR_{N}}\beta_N (\omega_{x}-h) \big) \ind_{\cA}\big]$,
and $\cD_{\eta}$ is defined in~\eqref{def:Dq}.

This shows in particular, using Skorokhod's representation theorem (see Remark \ref{rem:Skorokhod}),
that
\begin{equation}
\label{convforfluct}
\begin{split}
&\frac{1}{ \gb_N N^{d/\ga}}  \log \bP_{N,\gb_N}^{\go,h}\Big(  \max_{1\leq n \leq N} \|S_n\| \leq \eta\, N\Big) \\
 & = \frac{1}{ \gb_N N^{d/\ga}}  \log Z_{N,\gb_N}^{\go,h}\big( \max_{1\leq n \leq N} \|S_n\| \leq \eta\, N \big)
 - \frac{1}{ \gb_N N^{d/\ga}}  \log Z_{N,\gb_N}^{\go,h} \longrightarrow  \hat\cT_{\gb,\eta}  - \hat\cT_{\gb} \, \quad \text{a.s.}\, 
 \end{split}
\end{equation}
Noting that $\lim_{\eta\downarrow 0} \hat\cT_{\gb,\eta} =0$ a.s.,
then conditionally on having $\hat \cT_{\gb} >0$
we can choose
$\eta$ small enough so that
$\hat\cT_{\gb,\eta}  - \hat\cT_{\gb}$ is negative, with high (conditional) $\bbP$-probability.
From the convergence in~\eqref{convforfluct}, on the event $\hat\cT_{\gb,\eta}  - \hat\cT_{\gb}<0$
we have that $\log \bP_{N,\gb_N}^{\go,h} (  \max_{1\leq n \leq N} \|S_n\| \leq \eta\, N )$ goes to $-\infty$,  that is
$\bP_{N,\gb_N}^{\go,h} (  \max_{1\leq n \leq N} \|S_n\| \leq \eta\, N)$ goes to $0$ (exponentially fast).
This concludes the proof.
\qed

\section{Region B: proof of Theorem \ref{thm:B}}
\label{sec:convpartfunc}

In this section we prove Theorem \ref{thm:B}.
Recall that we choose $h=\mu$:
we denote $\bP_{N,\gb_N}^{\go,h=\mu}$ and $Z_{N,\beta_{N}}^{\omega,h=\mu}$ by $\brP_{N,\gb_N}$ and $\bar{Z}_{N,\gb_N}$ respectively.
Recall that $\ga\in (\frac d2, d)$ and that $\lim_{N\to\infty} \gb_N N^{-\gamma} = \gb \in (0+\infty)$ with $\gamma \in (\frac{d-\ga}{\ga}, \frac{d}{2\ga})$.
Define $\xi := \frac{\ga (1-\gamma)}{2\ga -d}$, which turns out to be the end-to-end critical exponent,
and note that we have  $\xi\in (\frac12,1)$ for the range of parameters considered.

Let
$M_{N}:=\max_{1\leq n\leq N}\|S_{n}\|_{\infty}$ (this notation is used in the rest of the paper).
For any~$q>0$, we split $\bar{Z}_{N,\gb_N}$ as follows
\begin{equation}
\label{splitbarZ}
\bar{Z}_{N,\gb_N}=\bar{Z}_{N,\gb_N}\big(M_{N}>q N^{\xi} \big)+\bar{Z}_{N,\gb_N} \big(M_{N}\leq q N^{\xi} \big)\,,
\end{equation}
where $\bar{Z}_{N,{\gb_N}} (\cA) = \bE\big[\exp\big( \sum_{x\in\cR_{N}}\beta_N(\omega_{x}-\mu) \big) \ind_{\cA}\big]$.

We divide the proof into three parts:
\textit{Step~1}.\ we show that $\bar Z_{N,\gb_N}(M_{N}> q N^{\xi} )$ is small for $q$ large with high $\bbP$-probability, which shows that $\bar{Z}_{N,\gb_N}(M_{N}> q N^{\xi})$ is negligible compared with $\bar Z_{N,\gb_N}$; \textit{Step~2}.\ we show that $ \log \bar{Z}_{N,\gb_N} (M_{N}\leq q N^{\xi})$,
when suitably rescaled,
converges in distribution to~$\cT_{\gb,q}$;
\textit{Step~3}.\ we let $q\to\infty$ and we conclude our main result.

The main difference here with respect to Section~\ref{sec:regionA} is the fact that we need to control the partition function with trajectories $M_N \geq q N^{\xi}$   (we had $\xi=1$ in the previous section): this brings many additional technical difficulties and makes the first step much more difficult.

\subsection{Step 1. Controlling $\bar Z_{N,\gb_N}(M_{N}\geq q N^{\xi})$} \label{subsec:step1B}

We prove the following estimate, slightly more general than what we need.

\begin{proposition}\label{prop:Z>Ar}
Suppose that $\gb_N N^{d\xi/\ga} \leq c N^{2\xi-1}$ for all $N$, for some constant $c$.
There exist positive constants $c_{1},c_{2}$ and $\nu>0$, such that for any sequence $A_{N}\geq1$, we can find $N_{0}$ such that for any $N\geq N_{0}$ we have 
\[
\bbP\Big( \bar Z_{N,\gb_N} \big(M_{N}\geq A_{N} \, N^{\xi}\big)\geq e^{-c_{1}A_{N}^{2} N^{2\xi-1}} \Big)  \leq c_2 A_N^{-\nu} 
\]
\end{proposition}
\begin{proof}
We partition the interval $[A_{N}N^{\xi},N]$ into blocks
\begin{equation}
B_{k,N}:=[2^{k-1}N^{\xi},2^{k}N^{\xi}],\quad k=\log_{2}A_{N}+1,\cdots,\log_{2}(N^{1-\xi})
\end{equation}
and we divide the partition function according
to the value of $M_N$:
\begin{equation}\label{Zbardec}
\bar{Z}_{N,\gb_N}\left(M_{N}\geq A_{N}N^{\xi}\right)=\sum\limits_{k=\log_{2}A_{N}+1}^{\log_{2}(N^{1-\xi})}\bar{Z}_{N}\left(M_{N}\in B_{k,N}\right).
\end{equation}
By applying Cauchy--Schwarz inequality, we have
\begin{equation}
\label{CSineq1}
\left(\bar{Z}_{N,\gb_N}\left(M_{N}\in B_{k,N}\right)\right)^{2}\leq\bP\big(M_{N}\geq 2^{k-1}N^{\xi}\big)\times\bar{Z}_{N,2\beta_{N}}\left(M_{N}\leq 2^{k} N^{\xi}\right).
\end{equation}
Then, a simple random walk estimate gives that
there are some constants $C_d,c_d$ such that 
for all $N\geq 1$ and all $k\geq 1$
\begin{equation}
\label{CSineq2}
\bP\big(M_{N}\geq2^{k-1}N^{\xi}\big)\leq C_{d}\,  e^{-c_{d}2^{2k+1}N^{2\xi-1}} \, .
\end{equation}
To obtain this, we can for instance use a union bound on the different coordinates, to reduce to one-dimensional random walk estimates, for which such an inequality is 
classical, see e.g.~\cite{Feller1}.
%

Let $c_{1}>0$ sufficiently small such that
\begin{equation}
\sum\limits_{k=\log_{2}A_{N}+1}^{+\infty} C_{d}^{1/2} \, e^{-c_{d}2^{2k-1}N^{2\xi-1} }\leq e^{-c_{1}A_{N}^{2} \, N^{2\xi-1}}.
\end{equation}
By a union bound, we obtain 
\begin{equation}
\begin{split}
\bbP \Big( \bar Z_{N,\gb_N}  \big(M_{N}\geq A_{N}N^{\xi}\big) &\geq  e^{-c_{1}A_{N}^{2} \, N^{2\xi-1} }\Big) \\
& \leq  \sum\limits_{k=\log_{2}A_{N}+1}^{\log_{2}(N^{1-\xi})}\bbP\left(\bar{Z}_{N,\gb_N}\left(M_{N}\in B_{k,N}\right)> C_{d}^{1/2} \, e^{- c_d 2^{2k-1} N^{2\xi-1}}\right) \\
& \leq \sum\limits_{k=\log_{2}A_{N}+1}^{+\infty}\bbP\left(\bar{Z}_{N,2\beta_{N}}\big(M_{N}\leq 2^{k}N^{\xi}\big)>e^{ c_d 2^{2k-1}N^{2\xi-1}}\right) \, ,
\end{split}
\end{equation}
where we used~\eqref{CSineq1} and~\eqref{CSineq2} for the last inequality.
%
The result then follows directly from Lemma~\ref{lem:S<qr} below.
\end{proof}
%

We formulate the following result in a general manner, since it will also be useful when~$\ga \in (0,\frac d2)$.
We still write $\bar Z_{N,\beta_{N}} = Z_{N,\gb_N}^{\go, h=\mu}$,
with $\mu=\bbE[\go_0]$ when $\ga>1$ and $\mu$  any real number when $\ga\in (0,1)$.

\begin{lemma}\label{lem:S<qr}
Assume that $\ga\in (0,d)$ and
let $(h_N)_{N\geq1}$ be a sequence
verifying
$\lim_{N\to\infty} h_N^2/N =+\infty$
and $h_N\leq N$.
Define $\gep_N := N \gb_N h_N^{d/\ga-2}$.
For any constant $c_0$,
there exist constants $c>0$ and $\nu>0$ such that for $N$ sufficiently large
\begin{equation}
\bbP\left( \bar Z_{N,\beta_{N}}\big(M_{N}\leq h_N\big)> e^{ c_0 h_N^2 /N } \right) \leq  c\, \gep_N^{\nu}  + \Big(\frac{h_N^2}{N}\Big)^{-\nu}.
\end{equation}
As a consequence, when $\alpha>\frac{d}{2}$, if $\gb_N N^{\xi d/\ga} \leq c N^{2\xi-1}$  then letting $h_N = q N^{\xi}$ we get that $\gep_N \leq c q^{\frac{d}{2} -\ga}$.
Hence, one can choose $q_{0} \geq 1$ such that for any $q\geq q_{0}$ we have
\begin{equation}\label{adapt-a}
\bbP\left( \bar Z_{N,2\beta_{N}}\big(M_{N}\leq qN^{\xi}\big)>e^{ c_{0}q^{2}N^{2\xi-1} }\right)\leq c \, q^{-\nu (2-\frac{d}{\ga}) }.
\end{equation}
\end{lemma}
%
%

The proof is analogous to that of Lemma~4.1 in \cite{BT19a};
let us warn the reader 
that it is quite technical.

\begin{proof}[Proof of Lemma~\ref{lem:S<qr}]
Let us note that the bound is trivial if $\gep_N \geq 1$. We will therefore assume that $\gep_N \leq 1$ for all $N$; in particular we have $\gb_N h_N^{d/\ga} \leq  h_N^2/N$.

Let us split $\bar Z_{N,\beta_{N}}(M_{N}\leq h_N)$
into three pieces, that we will control separately.
By H\"{o}lder's inequality (using also that $\mu>0$),
we have
\begin{equation}
\log \bar Z_{N,\beta_{N}}\big(M_{N}\leq h_N\big)\leq\frac{1}{3}\log Z_{N,3\beta_{N}}^{(>Q)}+\frac{1}{3}\log Z_{N,3\beta_{N}}^{((1,Q])}+\frac{1}{3}\log \bar Z_{N,3\beta_{N}}^{(\leq1)}
\end{equation}
where we defined, for any $\rho\in \R_+$,
\begin{align}
\label{Z>Q}
Z_{N,\rho\beta_{N}}^{(>Q)}&:=\bE\Big[\exp\Big(\sum\limits_{x\in\cR_{N}}\rho\beta_{N}\omega_{x}\mathbbm{1}_{\{\beta_{N}\omega_{x}>Q\}}\Big)\mathbbm{1}_{\left\{M_{N}\leq h_N \right\}}\Big],\\
\label{Z(1,Q)}
Z_{N,\rho\beta_{N}}^{(1,Q]}&:=\bE\Big[\exp\Big(\sum\limits_{x\in\cR_{N}}\rho\beta_{N}\omega_{x}\mathbbm{1}_{\{\beta_{N}\omega_{x}\in(1,Q]\}}\Big)\mathbbm{1}_{\left\{M_{N}\leq h_N\right\}}\Big],\\
\label{Z<1}
\bar Z_{N,\rho\beta_{N}}^{(\leq1)}&:=\bE\Big[\exp\Big(\sum\limits_{x\in\cR_{N}}\rho\beta_{N}(\omega_{x}-\mu)\mathbbm{1}_{\{\beta_{N}\omega_{x}\leq 1\}}\Big)\mathbbm{1}_{\left\{M_{N}\leq h_N\right\}}\Big],
\end{align}
where $Q:=Q_{N}$ is a constant depending on $N$ and will be chosen later (it will be of the form $(h_N^2/N)^{\zeta}$ for some $\zeta \in (0,1)$).
We now deal with the three terms separately.

\subsubsection*{First term: \eqref{Z>Q}. }
We want to show that by properly choosing $Q$ of the form $(h_N^2/N)^{\zeta}$ for a well chosen $\zeta \in (0,1)$, 
we have for $N$ large enough
\begin{equation}\label{bound:Z>Q}
\bbP\Big(\log Z_{N,3\beta_{N}}^{(>Q)}\geq c_0 \frac{h_N^2}{N}\Big)\le    c \gep_N^{\nu} + \Big(\frac{h_N^2}{N}\Big)^{-\nu}\, .
\end{equation}

Let $\ell:=\ell_{N}= (h_N^2/N)^{1-\gd}$, where $\gd$ is fixed (small enough). We consider the partition function restricted to the $\ell$ largest weights, as follows
\begin{equation}\label{def:truncatedpart}
Z_{N,3\beta_{N}}^{(\ell)}:=\bE\bigg[\exp\bigg(\sum\limits_{i=1}^{\ell}3\beta_{N} M_{i}^{(h_N)}\mathbbm{1}_{\{ Y_{i}^{(h_N)}\in\cR_{N}\}}\bigg)\bigg],
\end{equation}
where $( M_{i}^{(h_N)}, Y_{i}^{(h_N)})_{i\ge 0}$ is the ordered statistics introduced in \eqref{eq:ordstatdiscr}.

Now, since $\gb_N h_N^{d/\ga} \leq c h_N^2/N$, we have that
\begin{equation}\label{modpartfuncl1}
\begin{split}
\bbP \Big(\beta_{N} & M_{\ell}^{(h_N)} >Q\Big)\leq \bbP\Big( M_{\ell}^{(h_N)}>c_{\gb} Q  N h_N^{-2} h_N^{d/\ga}\Big) \leq \Big( c \ell^{1/\alpha}  N  Q /h_N^2 \Big)^{-\alpha\ell}.
\end{split}
\end{equation}
where we have used Lemma~5.1 in \cite{BT19b}
for the last inequality (see~\eqref{tail:orderstat} in Appendix).
We now choose $Q=\ell^{-\frac{1}{\ga}(1-\delta)^{1/2}}\frac{h_N^2}{N}$; recall that $\ell=\big(\frac{h_N^2}{N}\big)^{(1-\delta)}$. In particular, note that $Q= (h_N^{2}/N)^{\zeta}$ for some $\zeta \in (0,1)$, and so it goes to infinity as $N\to\infty$.

Since $\ell^{1/\alpha}  N  Q /h_N^2 \leq \ell^{\delta/3\alpha}$ if $\delta$ has been fixed small enough, we therefore get that with probability larger than $1-(c\ell^{\delta/3\alpha})^{-\alpha\ell}$  we have
\begin{equation}\label{modpartfuncl2}
\big\{x\in\cR_{N}\colon \beta_{N}\omega_{x}>Q\big\}\subset\Upsilon_{\ell}^{(h_N)}=\big\{Y_{1}^{(h_N)},\ldots,Y_{\ell}^{(h_N)}\big\}\,.
\end{equation}
We stress that on this event we have $Z_{N,3\beta_{N}}^{(>Q)}\leq Z_{N,3\beta_{N}}^{(\ell)}$, and we now turn to $Z_{N,3\beta_{N}}^{(\ell)}$. Recall the notation~\eqref{def:Omegar} of $\gO^{\ell}(\Delta)$. By Lemma \ref{AppendixL2},
we have that 
\begin{equation*}
Z_{N,3\beta_{N}}^{(\ell)}=\sum\limits_{\Delta \subset \gU_{\ell}^{(h_N)}}e^{3\beta_{N}\Omega_{h_N}^{(\ell)}(\Delta)}\bP( \Delta \subset \cR_N  )\leq
\sum\limits_{\Delta \subset\gU_{\ell}^{(h_N)}} (C_{1}  )^{ |\Delta|} \exp\Big( 6\beta_{N}\Omega_{h_N}^{(\ell)}(\Delta)-C_{2}\ent_N(\Delta) \Big)  
\end{equation*}
(recall that we view $\Delta=(x_1,\ldots, x_k)$ as an ordered subset of $\gU_{\ell}^{(h_N)}$).
Therefore, recalling the definition~\eqref{def:intDiscrVarProb2} of the 
variational problem $\cT_{N,r}^{\gb,(\ell)}$, we have, setting $C_3:=6/C_2$
\begin{equation}
\label{decom:Z<ell0}
Z_{N,3\beta_{N}}^{(\ell)} \leq \ell ! (2C_{1}  )^{ \ell} 
\exp\Big( C_2 \cT_{N,h_N}^{ C_3 \beta_{N},(\ell)} \Big) \, .
\end{equation}
Then for $N$ large enough, using the definition of $\ell = (h_N^2/N)^{1-\gd}$,
 we get 
\begin{equation}
\label{decom:Z<ell}
\log Z_{N,3\beta_{N}}^{(\ell)} \leq \ell  \big( \log \ell +\log(2C_{1}) \big)+ C_{2} \cT_{N,h_N}^{C_3 \beta_{N},(\ell)}
  \leq  \frac{c_{0}}{2}\frac{h_N^2}{N}+ C_2 \cT_{N,h_N}^{C_3 \gb_N ,(\ell)} \, ,
\end{equation}
where $c_0$ is a fixed constant (appearing
in~\eqref{bound:Z>Q}).
Recalling the definition $\gep_{N}:=N\gb_N h_N^{d/\ga-2}$, we therefore get that 
\begin{equation}\label{main014}
\begin{split}
\bbP\Big(\log Z_{N,3\beta_{N}}^{(\ell)}\geq c_{0} \frac{h_N^2}{N}\Big)
& \leq  \bbP\Big(\cT_{N,h_N}^{C_3 \beta_{N},(\ell)}\geq\frac{c_{0}}{2C_{2}}\frac{h_N^2}{N}\Big)\\
& \leq\bbP\Big(\cT_{N,h_N}^{C_3 \beta_{N},(\ell)}\geq \frac{c_0}{2C_2C_3^2} \gep_N^{-2}  N \Big( C_3\beta_{N} (h_N)^{\frac{d}{\ga} -1 } \Big)^{2}\Big) \, .
\end{split}
\end{equation}
Then,
one simply needs to use
Proposition \ref{tail:disT} to get that 
the last term is bounded by a constant times
$\gep_N^{ \ga d/(\alpha+d)}$. This, together with \eqref{modpartfuncl1}, establishes \eqref{bound:Z>Q}.

\subsubsection*{Second term:~\eqref{Z(1,Q)}.}
We now show that  there exists some $\zeta\in(0,1)$, such that for any $C>0$ and $N$ large enough,
\begin{equation}\label{eq:Z_bar(1,Q)}
\bbP\left(\log Z_{N,3\beta_{N}}^{((1,Q])}\geq C\Big(\frac{h_N^2}{N}\Big)^\zeta \right)\leq  c \Big(\frac{h_N^2}{N}\Big)^{-\nu}.
\end{equation}

We again decompose $Z_{N,3\beta_{N}}^{((1,Q])}$, according to the contribution of different ranges of values in the environment.
We fix $\gd$ (appearing in the definition of $Q$ and $\ell$ above) small enough so that $\theta:=(1-\delta) \frac{d}{\alpha}>1$ (recall $\alpha<d$). Then, for $j\geq0$, we set  
\begin{align}
\label{def:ell}
\ell_{j} &:=\Big( \frac{h_N^2}{N}\Big)^{\theta^{j}(1-\delta)}=(\ell)^{\theta^{j}},\\
\label{def:Q}
Q_{j}&:=\Big( \frac{h_N^2}{N}\Big) (\ell_{j})^{-\frac{1}{\ga}(1-\delta)^{1/2}}=\Big( \frac{h_N^2}{N}\Big)^{1- \frac{\theta^j}{\ga}(1-\delta)^{3/2}},
\end{align}
where we used the definition of $\ell$ above.
Note that $\ell_0=\ell$, $Q_0 =Q$ and that each pair $(\ell_{j},Q_{j})$ has a similar form to $(\ell,Q)$, and let us stress that $\ell_{j-1}<\ell_j$ and $Q_{j-1} > Q_j$ for any $j\geq 1$.
Now, let $\kappa$ be the smallest integer such that $\theta^{\kappa}(1-\delta)^{3/2}\geq  \ga$, and note that for such $\kappa$ we have $Q_{\kappa}\leq 1$. 

Then, thanks to H\"{o}lder's inequality, we have
\begin{equation}
\log Z_{N,3\beta_{N}}^{((1,Q])}\leq\frac{1}{\kappa}\sum\limits_{j=1}^{\kappa}\log Z_{N,3\kappa\beta_{N}}^{(j)},
\end{equation}
where we have set
\begin{equation}
Z_{N,3\kappa\beta_{N}}^{(j)}:=\bE\bigg[\exp\bigg(\sum\limits_{x\in\cR_{N}}3\kappa\beta_{N}\omega_{x}\mathbbm{1}_{\{\beta_{N}\omega_{x}\in(Q_{j},Q_{j-1}]\}}\bigg)\mathbbm{1}_{\left\{M_{N}\leq h_N\right\}}\bigg] \, .
\end{equation}
%
Thanks to an union bound we only need to show that for any $j=1,\dots, \kappa$,

\begin{equation}\label{logZj}
\bbP\left(\log Z_{N,3\kappa\beta_{N}}^{(j)} \geq  C\Big(\frac{h_N^2}{N}\Big)^\zeta \right)\leq c\Big(\frac{h_N^2}{N}\Big)^{-\nu}.
\end{equation}

 to get \eqref{eq:Z_bar(1,Q)}.

By the same argument used for $Z_{N,3\gb_N}^{(>Q)}$, see \eqref{modpartfuncl1}-\eqref{modpartfuncl2}, we get that 
\begin{equation}\label{step2:subset}
\left\{x\in\cR_{N}:\beta_{N}\omega_{x}>Q_{j}\right\}\subset\Upsilon_{\ell_{j}}:=\big\{Y_{1}^{(h_N)},\cdots,Y_{\ell_{j}}^{(h_N)}\big\} \, .
\end{equation}
with probability larger than $1-(c\ell_{j}^{\delta/3\alpha})^{-\alpha\ell_{j}}$.
Hence, as above, 
we are reduced to controlling
\begin{equation}\label{eq:Z_ell}
Z_{N,3\kappa\beta_{N}}^{(j)}\leq\bE\bigg[\exp\bigg(3\kappa Q_{j-1}\sum\limits_{k=1}^{\ell_{j}}\mathbbm{1}_{\{Y_{k}^{(h_N)}\in\cR_{N}\}}\bigg)\bigg] 
 = \sum\limits_{k=0}^{\ell_{j}}\sum\limits_{\Delta\subset\Upsilon_{\ell_j},|\Delta|=k}e^{3\kappa Q_{j-1}k}\bP(\cR_{N}\cap\Upsilon_{\ell_j}=\Delta). 
\end{equation}
We split the sum in~\eqref{eq:Z_ell} at some level $K =K_j:= (\ell_{j})^{1/d+\delta/6d} $ (this is needed in~\eqref{entropicestimate} below) and
we use Lemma~\ref{AppendixL2} to bound the probability $\bP(\Delta\subset \cR_{N})\geq\bP(\cR_{N}\cap\Upsilon_{\ell_j}=\Delta)$ for $k\geq K$. We get 
\begin{equation}
\begin{split}
Z_{N,3\kappa\beta_{N}}^{(j)} & \leq 
e^{3 \kappa Q_{j-1} K } + 
\sum_{k=K}^{\ell_j} \sum_{\Delta \subset \gU_{\ell_j} \,, |\Delta| =k } (C_{1}  )^{k}  \exp\Big( 3 \kappa Q_{j-1} k  - C_2 \ent_N(\Delta) \Big)  \\
 &
 \le e^{3 \kappa Q_{j-1} K } +
 \cH_j \, ,
 \end{split}
 \label{boundZj}
\end{equation}
where we have set
\[
\cH_j := \sum\limits_{k=K}^{\ell_j}k! \binom{\ell_{j}}{k} (C_{1} )^{k}\exp\Big(6\kappa Q_{j-1}k-C_{2} \inf\limits_{\Delta\subset\Upsilon_{\ell_{j}},|\Delta|\geq k }\ent_N(\Delta) \Big).
\]
We now use Entropic LPP estimates to control the last infimum. We let 
\[
B_{k}:=\frac{12\kappa}{C_{2}}Q_{j-1}k=\frac{12\kappa}{C_{2}}\frac{h_N^2}{N}\ell_{j}^{-\frac{1}{d}(1-\delta)^{-1/2}} k\,,
\]
where we used the definition~\eqref{def:Q} and the fact that $\ell_{j-1} = \ell_j^{1/\theta}$ together with the definition $\theta =(1-\gd) \frac{d}{\ga}$.
Then, thanks to Theorem~\ref{thm1} in Appendix, we obtain
\begin{align*}
\bbP\Big(\inf\limits_{\Delta\subset\Upsilon_{\ell_{j}},|\Delta|\geq k}\ent_N(\Delta)\leq B_{k}\Big)  =  \bbP\left(L_{\ell_{j}}^{(N B_{k})}(h_N)\geq k\right)&  \leq 
 \bigg(C'\frac{ ( N B_k )^{1/2} \ell_{j}^{1/d}}{h_Nk }\bigg)^{dk} \, ,
\end{align*}
so that using the definition of $B_k$, we end up with
\[
\bbP\Big(\inf\limits_{\Delta\subset\Upsilon_{\ell_{j}},|\Delta|\geq k}\ent_N(\Delta)\leq B_{k}\Big)  
\le \Big(c (\ell_j)^{\frac1d-\frac{1}{2d(1-\delta)^{1/2}}}k^{-1/2} \Big)^{dk}\,
\leq (c \ell_j)^{-k \delta /4}\,  ,
\]
where we have used in the last inequality that 
$k\geq K := (\ell_{j})^{1/d+\delta/6d} $ (and took $\gd$ small enough).
By a union bound, this leads to
\begin{equation}
\label{entropicestimate}
\bbP\Big( \inf_{\Delta\subset\Upsilon_{\ell_{j}},|\Delta|\geq k}\ent(\Delta)>B_{k}  \, ,  \forall \, K \leq k \leq \ell_j \Big) \geq 1-\sum\limits_{k= K}^{\infty}(c\ell_{j})^{-k \gd /4}\geq1-(c\ell_{j})^{- c_\delta\ell_{j}^{1/d} }\, .
\end{equation}
Moreover, on this event,
we get that for $N$ large enough
\[
\cH_{j} \leq\sum\limits_{k=K }^{\infty}k!\binom{\ell_{j}}{k} (C_{1}  )^{k}   e^{-6\kappa Q_{j-1}k}
\leq \sum\limits_{k=1}^{\infty} \ell_{j}^k  \exp\Big( - 4\kappa  k (h_N^2/N)^{\vartheta} \Big) \, ,
\]
where we used in the last inequality that $Q_{j-1} \geq Q_{\kappa-1} = (h_N^2/N)^{\vartheta}$ for some $\vartheta = \vartheta_{\kappa} >0$ (recall~\eqref{def:Q}),
and that $(h_N^2/N)^{\vartheta}$ goes to infinity.
Now using that $\ell_j \leq \ell_{\kappa} \leq (h_N^2/N)^{\theta^{\kappa} (1-\delta)}$, we get that for $N$ large enough
\begin{equation}\label{Hjestimate}
\cH_j \leq \sum\limits_{k=1}^{\infty} \exp\Big( - 2 \kappa  k (h_N^2/N)^{\vartheta} \Big)  \leq  \exp\Big(- \kappa   (h_N^2/N)^{\vartheta}   \Big) \, .
\end{equation}
Note also that $K Q_{j-1}\le (h_N^2/N)\ell_j^{-\delta/4d}\leq(h_N^2/N)\ell^{-\delta/4d}$ goes to $0$ as $N\to\infty$: going back to~\eqref{boundZj}, we therefore get that on the event considered in~\eqref{entropicestimate}, for $N$ large enough we have $Z_{N,3\kappa \gb_N}^{(j)} \leq  c(h_N^2/N)^\zeta$.

Combining this with \eqref{step2:subset} and \eqref{entropicestimate}, we therefore obtain
that
\begin{equation}
\bbP\left( Z_{N,3\kappa\beta_{N}}^{(j)}\geq 2 \right)\leq(c\ell_{j})^{-\delta\ell_{j}/3}+(c\ell_{j})^{-c_\delta\ell_j^{1/d}},
\end{equation}
which proves \eqref{logZj}, recalling that $\ell_{j}\geq \ell=\big( h_N^2/N\big)^{1-\delta}$.

\subsubsection*{Third term: \eqref{Z<1}}
We now show that there exists a constant $C$ such that 
\begin{equation}\label{logZ_bar<1}
\bbP\Big(\log \bar Z_{N,3\beta_{N}}^{(\leq1)}>c_{0}\frac{h_N^2}{N}\Big)<C  \Big(\frac{h_N^2}{N} \Big)^{-\nu}.
\end{equation}

To bound \eqref{logZ_bar<1} we note that
there is some $C>0$ such that $e^{x}\leq1+x+Cx^2$ for $|x| \leq 4$.
Therefore, for $N$ large enough (so that $3\gb_N \mu \leq 1$), we obtain
\begin{equation}\label{eq:Zbar<1}
\bar Z_{N,3\beta_{N}}^{(\leq1)}\leq 
\bE\Big[\prod\limits_{x\in\cR_{N}}\left(1
+3\beta_N(\omega-\mu)\mathbbm{1}_{\{\beta_{N}\omega_{x}\le 1\}}
+C'\beta_{N}^2(\omega_{x}-\mu)^2\mathbbm{1}_{\{\beta_{N}\omega_{x}\leq1\}}\right)\Big].
\end{equation}
Note that $\bbE[(\omega-\mu)\mathbbm{1}_{\{\beta_{N}\omega\le 1\}}]\le 0$ as soon as $\beta_N^{-1}\ge \mu$ (our assumption implies that $\gb_N \to 0$).
Hence, using also that $|\cR_N|\leq N$, we get that 
\begin{equation}\label{eq:Zbar<1_E}
\bbE\big[ \bar Z_{N,3\beta_{N}}^{(\leq1)}\big]
\leq \Big(1+ C' \gb_N^2 \bbE\big[ (\omega-\mu)^2\mathbbm{1}_{\{\beta_{N}\omega\le 1\}} \big] \Big)^N 
\leq 
\begin{cases}
\exp\left(CN\beta_{N}^{\alpha}\right)  & \text{ if } \ga<2 \, , \\
\exp\left(CN\beta_{N}^{2}\log(1/\gb_N) \right) & \text{ if } \ga\geq 2  \, ,\\
\end{cases}
\end{equation}
where we used that the truncated expectation is bounded
by a constant times $ (1/\gb_N)^{2-\ga}$ if $\ga<2$, 
by a constant times $ \log(1/\gb_N)$ if $\ga=2$,
and by a constant if $\ga>2$.

Then by Markov's inequality and \eqref{eq:Zbar<1_E}, we have that

\begin{equation}\label{eq:Zbar<1_last}
\bbP\Big(\log \bar Z_{N,3\beta_{N}}^{(\leq1)}>c_{0}\frac{h_N^2}{N}\Big)\leq e^{-c_0\frac{h_N^2}{N}}\bbE\big[Z_{N,3\beta_N}^{(\leq1)}\big]\leq e^{-c_0\frac{h_N^2}{N}}\times\begin{cases}
\exp(CN\beta_N^\ga),\quad&\text{if}~\ga<2,\\
\exp(CN\beta_N^2\log(1/\beta_N)),\quad&\text{if}~\ga\geq2.
\end{cases}
\end{equation}
It suffices to show that $h_N^{-2}N^2\beta_N^{2\wedge\ga}\log(1/\beta_N)^{\mathbbm{1}_{\{\ga\geq2\}}}\overset{N\to\infty}{\longrightarrow}0$.
Since $\gb_N h_N^{d/\alpha}  \leq c h_N^2 /N$, 
we get that:

$\ast$ If $\ga<2$ then
\begin{equation}
\label{tailstep3}
\frac{N^2}{h_N^2} \beta_{N}^{\alpha} 
\leq c'  \Big(\frac{h_N^2}{N}\Big)^{\alpha-2} h_N^{2-d} \leq   c'\Big(\frac{h_N^2}{N}\Big)^{\alpha-2} \, , 
\end{equation}
 because $d\geq 2$.

$\ast$ If on the other hand $\ga\geq 2$,  we can choose $\epsilon>0$ sufficiently small, such that
\begin{equation}
\frac{N^2}{h_N^2} \beta_{N}^{2} \log(1/\gb_N)\leq c\frac{N^2}{h_N^2}\beta_N^{2-\epsilon}\leq ch_N^{2(1-\frac{d}{\alpha})+\frac{d}{\ga}\epsilon}\Big(\frac{h_N^2}{N}\Big)^{-\epsilon}\leq N^{1-\frac{d}{\alpha}+\frac{d}{2\alpha}\epsilon}\leq\Big(\frac{h_N^2}{N}\Big)^{1-\frac{d}{\alpha}+\frac{d}{2\alpha}\epsilon}
\end{equation}
using that $h_N\gg N^{1/2}$ for the third inequality and that $h_N^2/N \leq N$ for the last inequality. All together, this proves \eqref{logZ_bar<1}.


\smallskip
The conclusion of Lemma~\ref{lem:S<qr} then simply follows by
combining~\eqref{bound:Z>Q}, \eqref{eq:Z_bar(1,Q)} and \eqref{logZ_bar<1}.
\end{proof}

\subsection{Step 2: convergence of $\log \bar Z_{N}(M_{N}\leq qN^{\xi})$}\label{subsec:step2B}

We now prove the following convergence.

\begin{proposition}
\label{prop:convZq}
Assume that $\lim_{N\to\infty} \gb_N N^{d\xi/\ga} N^{-(2\xi-1)} = \gb \in (0,+\infty)$.
Then  for all $q\in(0,+\infty)$ we have the following convergence in distribution
\[
\frac{1}{N^{2\xi-1}} \log \bar Z_{N,\gb_N}(M_{N}\leq qN^{\xi}) 
\overset{(d)}{\longrightarrow}\cT_{ \beta,q} \, ,
\]
where $\cT_{ \beta,q}$ is defined in~\eqref{con2}.
\end{proposition}

Once trajectories are restricted to staying in a ball of radius $qN^{\xi}$, the proof of the convergence is very similar to
the proof in Section~\ref{sec:regionA}.
We follow the same steps: (i) first we show that most of the contribution to the partition function comes from large disorder weights; (ii) we prove that the log-partition function, when restricted to finitely many weights, converges in distribution;
(iii) we send the number of weights to infinity.

\smallskip
\noindent
{\bf Step 2.(i)-a.}
First of all, we show that we can restrict the partition function
to weights $\go_x \geq 1/\gb_N$.
Recall the notations  $\bar{Z}_{N,\rho \beta}^{(\leq1)}$ in \eqref{Z<1},
and define analogously $\bar{Z}_{N,\rho\beta}^{(>1)}$ by replacing $\ind_{\{\gb_N \go_x \leq 1\}}$ with $\ind_{\{\gb_N \go_x > 1\}}$ inside~\eqref{Z<1}; take $h_N =q N^{\xi}$. For any $\eta\in(0,1)$, by H\"{o}lder's inequality, we have
\begin{equation}\label{upper:Z<}
\bar{Z}_{N,\gb_N}(M_{N}\leq qN^{\xi})\leq\left(\bar{Z}_{N,(1+\eta)\beta_{N}}^{(>1)}\right)^{\frac{1}{1+\eta}}\left(\bar{Z}_{N,(1+\eta^{-1})\beta_{N}}^{(\leq1)}\right)^{\frac{\eta}{1+\eta}},
\end{equation}
and by applying H\"{o}lder's inequality to $\bar{Z}_{N,(1-\eta)\beta_{N}}^{(>1)}$,  we also have for any $\eta\in (0,1)$
\begin{equation}\label{lower:Z>}
\bar{Z}_{N,\gb_N}(M_{N}\leq qN^{\xi})\geq\left(\bar{Z}_{N,(1-\eta)\beta_{N}}^{(>1)}\right)^{\frac{1}{1-\eta}}\left(\bar{Z}_{N,(1-\eta^{-1})\beta_{N}}^{(\leq1)}\right)^{\frac{\eta}{\eta-1}}.
\end{equation}
Note that if $\beta_{N}\omega>1$ and $N$ large enough, we have $$(1-2\eta)\beta_{N}\omega<(1-\eta)\beta_{N}(\omega-\mu)\quad \text{and}\quad (1+\eta)\beta_{N}(\omega-\mu)\leq(1+\eta)\beta_{N}\omega, $$
 so that we can replace $\bar{Z}_{N,(1+\eta)\beta_{N}}^{(>1)}$ and $\bar{Z}_{N,(1-\eta)\beta_{N}}^{(>1)}$ by $Z_{N,(1+\eta)\beta_{N}}^{(>1)}$ and $Z_{N,(1-2\eta)\beta_{N}}^{(>1)}$ respectively,
where
\begin{equation}
Z_{N,\rho\beta_{N}}^{(>1)} := \bE\bigg[\exp\bigg( \rho\beta_{N}\sum\limits_{x\in\cR_{N}} \omega_{x}\mathbbm{1}_{\{\beta_{N}\omega_{x}>1\}}\bigg)\mathbbm{1}_{\left\{M_{N}\leq qN^{\xi}\right\}}\bigg].
\end{equation}

 We need to control the upper bound for \eqref{upper:Z<} and the lower bound for \eqref{lower:Z>} (note that $1-\eta^{-1}, \eta-1<0$). Now notice that adapting~\eqref{eq:Zbar<1_last},
we easily get that for any $\rho>0$, any $\gep>0$

\begin{equation*}
\bbP \big( N^{-(2\xi-1)}\log\bar{Z}_{N,\rho\beta_{N}}^{(\leq1)}  > \gep \big)\leq\exp(C_\rho N\beta_N^{2\wedge\ga}\log(1/\beta_N)-\epsilon N^{2\xi-1}),
\end{equation*}
which goes to $0$ as $N\to\infty$ (the proof of~\eqref{eq:Zbar<1_E} is carried out for $\rho=3$ but it is easily adapted to any $\rho>0$).  For $\rho<0$, a slight modification in \eqref{eq:Zbar<1}-\eqref{eq:Zbar<1_E} is needed. Note that we have $\bE[-\beta_N(\omega-\mu)\mathbbm{1}_{\{\beta_N\omega_x\leq1\}}]=\bE[\beta_N(\omega-\mu)\mathbbm{1}_{\{\beta_N\omega_x>1\}}]\leq C\beta_N^\ga$, which gives an extra term in \eqref{eq:Zbar<1_E} for $\rho<0$ and it follows that
\begin{equation*}
\bbP \big( -N^{-(2\xi-1)}\log\bar{Z}_{N,\rho\beta_{N}}^{(\leq1)}  < -\gep \big)\leq\exp(C_\rho N\beta_N^{2\wedge\ga}\log(1/\beta_N)-\epsilon N^{2\xi-1})\overset{N\to\infty}{\longrightarrow}0.
\end{equation*}

From~\eqref{upper:Z<}-\eqref{lower:Z>} we simply need to prove the
convergence in distribution of
$N^{-(2\xi-1)} \log Z_{N,\rho\beta_{N}}^{(>1)}$ to $\cT_{\rho\gb,q}$.

Notice also that thanks to~\eqref{eq:Z_bar(1,Q)},
we also have that for any $\rho>0$
\[
\frac{1}{N^{2\xi-1}} \log Z_{N,\rho\gb_N}^{(1,Q]} \longrightarrow 0\,, \quad \text{in probability}.
\]
Using H\"older's inequality as in~\eqref{upper:Z<} and~\eqref{lower:Z>},
we are therefore reduced to showing that for any $\rho>0$ we have the following convergence in distribution
\begin{equation}
\frac{1}{N^{2\xi-1}} \log Z_{N,\rho\gb_N}^{(>Q)}   \overset{(d)}{\longrightarrow}\cT_{ \rho\beta,q} \, .
\end{equation}

\smallskip
\noindent
{\bf Step 2.(i)-b.} We now show that we can restrict the partition function
to a finite number of weights.
Analogously to Section~\ref{sec:regionA} (see~\eqref{ZL}-\eqref{ZLell}), we define for any $L\in \N$,
\begin{align}
Z_{N,\rho\beta_{N}}^{(L)}&:=\bE\bigg[\exp\bigg(\rho\beta_{N} \sum\limits_{i=1}^{L}M_{i}^{(qN^{\xi})}\mathbbm{1}_{\{Y_{i}^{(qN^{\xi})}\in\cR_{N}\}}\bigg)\bigg],\label{ZL2}\\
Z_{N,\rho\beta_{N}}^{(L,\ell)}&:=\bE\bigg[\exp\bigg(\rho\beta_{N} \sum\limits_{i=L+1}^{\ell} M_{i}^{(qN^{\xi})}\mathbbm{1}_{\{Y_{n}^{(qN^{\xi})}\in\cR_{N}\}}\bigg)\bigg]\, ,\label{ZLell2}
\end{align}
where we consider the order statistics $(M_i^{(qN^{\xi})},Y_i^{(q N^{\xi})})$
in the (discrete) ball of radius $qN^{\xi}$.
Recall that $\ell := (h_N^2/N)^{1-\gd} = (q^2 N^{2\xi-1})^{1-\gd}$.

Thanks to H\"{o}lder's inequality, we get that for any $\eta\in (0,1)$, analogously to~\eqref{ineq:logZ},
then for $N$ large enough we have
\begin{equation}
\log Z_{N,\rho\beta_{N}}^{(L)}\leq \log Z_{N,\rho\gb_N}^{(>Q)} \leq\frac{1}{1+\eta}\log Z_{N,(1+\eta)\beta_{N}}^{(L)}+\frac{\eta}{1+\eta}\log Z_{N,(1+\eta^{-1})\beta_{N}}^{(L,\ell)} 
\end{equation}
with probability going to $1$ as $N\to\infty$ (more precisely on the
event  $\{\beta_N M_L^{(qN^\xi)}\geq Q>\beta_N M_\ell^{(qN^\xi)}\}$. Recall \eqref{modpartfuncl1} and $N^{-d\xi/\ga}M_L^{(qN^\xi)}{\longrightarrow}\bM_L^{(q)}$ in Proposition \ref{thm3}).
We now prove the analogous of Lemma~\ref{ZLellvanish+} to show that the last term is negligible. 

\begin{lemma}\label{Z_L_fix}
For any $\epsilon>0$  and any $\rho>0$, we have
\begin{equation*}
\limsup_{L\to\infty} \limsup_{N\to+\infty}\bbP\Big(\frac{1}{N^{2\xi-1}}\log Z_{N,\rho\beta_{N}}^{(L,\ell)}>\epsilon\Big) =0 \, .
\end{equation*}
\end{lemma}
\begin{proof}
Note that we easily have, as in the calculation leading to~\eqref{decom:Z<ell0},
\begin{equation}
Z_{N,\rho\beta_{N}}^{(L,\ell)} \leq\sum\limits_{\Delta\subset\Upsilon_{\ell}^{(qN^{\xi})} } e^{\rho\beta_{N}\Omega_{qN^{\xi}}^{(>L)}}\bP\left(\Delta\subset\cR_{N}\right) \leq \ell! (2 C_{1}  )^{\ell}
 \exp\left( C_2 \cT_{N,qN^{\xi} }^{C_{\rho} \beta_{N},(>L)}\right).
\end{equation}
Then
\begin{equation*}
 N^{-(2\xi-1)} \log Z_{N,\rho\beta_{N}}^{(L,\ell)}\leq \ell  N^{-(2\xi-1)} \log( \ell 2 C_{1}  )+ C_2 N^{-(2\xi-1)} \cT_{N,qN^{\xi}}^{\frac{\rho\beta_{N}}{C_{2}},(>L)}.
\end{equation*}
Note that since $\ell = (q^2 N^{2\xi-1})^{1-\gd}$ with $q$ fixed,
the first term goes to $0$ as $N\to\infty$ (recall $\xi >1/2$).
For the second term, 
using that $\gb_N \sim \gb N^{2\xi-1 - d\xi/\ga}$
we get that for any $\gep>0$
\begin{align*}
\bbP\left(\cT_{N,qN^{\xi}}^{ C_\rho \beta_{N},(>L)}> (2C_2)^{-1} \gep N^{2\xi-1} \right) 
&= \bbP\bigg( \cT_{N,qN^{\xi}}^{ C_\rho\beta_{N} ,(>L)}> C' \epsilon L^{2(\frac{1}{\alpha}-\frac{1}{d})} N  \times \Big( C_{\rho} \gb_N  (qN^{\xi})^{\frac{d}{\ga}-1}  L^{\frac{1}{d}-\frac{1}{\alpha}} \Big)^2 \bigg)\,,
\end{align*}
where the constant $C'$ depends on $\rho,q,\gb$. Then,
using Proposition~\ref{tail:disT}-\eqref{tail:T>ell}, provided that $L$
is large enough so that $C' \epsilon L^{2(\frac{1}{\alpha}-\frac{1}{d})} >1$,
we get
\[
\bbP\left(C_2 \cT_{N,qN^{\xi}}^{ C_\rho \beta_{N},(>L)}>  \tfrac12 \gep N^{2\xi-1} \right) 
\leq c\left(C' \epsilon L^{2(\frac{1}{\alpha}-\frac{1}{d})} \right)^{ {- \frac{\alpha L d}{2(\alpha L+d)}} } \leq c_{\gep} L^{-a},
\]
for some exponent $a>0$.
This concludes the proof of Lemma~\ref{Z_L_fix}.
\end{proof}

\smallskip
\noindent
{\bf Step 2.(ii).} Once the number of weights is fixed, we can prove
the following convergence in distribution. The proof is identical to
that of Proposition~\ref{logZL} (replacing $\hatent$ with $\ent$)
and is omitted.

\begin{proposition}
\label{prop:convergence}
For any positive integer $L$ and any $\rho>0$, we have the following convergence
\begin{equation*}
\frac{1}{N^{2\xi-1}} \log Z_{N,\rho\gb_N}^{(L)}\overset{(d)}{\longrightarrow}\cT_{\rho\beta,q}^{(L)} \, ,
\end{equation*}
where $\cT_{\beta,q}^{(L)}$ is defined in \eqref{con1}.
\end{proposition}

\smallskip
{\bf Step 2.(iii). }
As mentioned in Proposition~\ref{thm2}, 
we have that $\cT_{\rho\beta,q}^{(L)} $ converges to 
$\cT_{\rho\beta,q}$ as $L\to\infty$. Thanks to the continuity in $\gb$,
we also have that $\cT_{\rho\beta,q}$ converges to $\cT_{\rho\beta,q}$
as $\rho\to 1$.
The conclusion of the proof of Proposition~\ref{prop:convZq}
then follows from combining Proposition~\ref{prop:convergence}
with Lemma~\ref{Z_L_fix} (and Step 2.(i)-a), letting $L\to\infty$ then $\rho$ to $1$.

\subsection{Step 3: letting $q\to\infty$}

Going back to~\eqref{splitbarZ},
the conclusion of the convergence in Theorem~\ref{thm:B}
simply follows from Step 1 (Proposition~\ref{prop:Z>Ar}) and Step 2 (Proposition~\ref{prop:convZq}).
Indeed, thanks to Proposition \ref{prop:Z>Ar},
we get that if $q$ has been fixed large enough,  then with probability $1-c q^{-\nu}$ we have for $N$ large enough
\[
\bar{Z}_{N,\gb_N}(M_{N}\leq qN^{\xi}) \leq  \bar{Z}_{N,\gb_N} \leq   \bar{Z}_{N,\gb_N} (M_{N}\leq qN^{\xi}) + 1 \, .
\]
Using Proposition~\ref{prop:convZq} and noting that $\cT_\beta>0$ a.s. and then letting $q\to\infty$
concludes the proof of~\eqref{def:varprob1} (recall that $\cT_{\gb,q}$ converges to $\cT_{\gb}$ as $q\to\infty$, by monotonicity).

\subsection{Transversal fluctuations}

%

Let us first prove that the transversal fluctuations are \emph{at least}
of order $N^{\xi}$, analogously to what is done in Section~\ref{sec:transversalA}.
Notice that for any $\eta>0$ we have thanks to  Proposition~\ref{prop:convZq} and~\eqref{def:varprob1} (using Skorokhod's representation theorem, cf.\ Remark \ref{rem:Skorokhod})
\begin{equation}
\label{convforfluct2}
\frac{1}{ N^{2\xi-1}}  \log \overline{\bP}_{N,\gb_N}\big( M_N \leq \eta\, N^{\xi}\big)  = \frac{1}{ N^{2\xi-1}} \Big( \log \bar Z_{N,\gb_N}\big( M_N \leq \eta\, N^{\xi} \big)
 -  \log \bar Z_{N,\gb_N}  \Big)\longrightarrow \cT_{\gb,\eta}  - \cT_{\gb} \, \quad \text{a.s.}
\end{equation}
Since we have $ \cT_{\gb} >0$ and  $\lim_{\eta\downarrow 0} \cT_{\gb,\eta} =0$ a.s.,
we can choose
$\eta$ small enough so that
$\cT_{\gb,\eta}  - \cT_{\gb}$ is negative with high $\bbP$-probability.
From~\eqref{convforfluct2}, on the event $\cT_{\gb,\eta}  -\cT_{\gb}<0$
we have that $\bar \bP_{N} ( M_N \leq \eta\, N^{\xi})$ goes to $0$:
this concludes the proof that transversal fluctuations are at least of order $N^{\xi}$.

Then we prove that the transversal fluctuations are \emph{at most} of order $N^{\xi}$. We have that
\begin{equation*}
\frac{1}{ N^{2\xi-1}}  \log \overline{\bP}_{N,\gb_N}\bigg( M_N \geq\frac{1}{\eta} N^{\xi}\bigg) = \frac{1}{N^{2\xi-1}}\log \bar{Z}_{N,\beta_N}\bigg(M_N\geq \frac{1}{\eta}N^\xi\bigg) - \frac{1}{N^{2\xi-1}}\log \bar{Z}_{N,\beta_N}.
\end{equation*}
The second term above converges to $\cT_{ \beta}$, which is positive a.s.. For the first term, we have that $\bar{Z}_{N,\beta_N}(M_N\geq \eta^{-1}N^\xi)<\exp(-c_1\eta^{-2}N^{2\xi-1})$ with $\bbP$-probability larger than $1-\eta^\nu$ by Proposition \ref{prop:Z>Ar}. Therefore, $\overline{\bP}(M_N\geq \eta^{-1}N^\xi)$ vanishes exponentially fast with $\bbP$-probability close to $1$ for small enough $\eta$, which concludes the result.

\section{Region~C: proof of Theorems~\ref{thm:C1}}
\label{sec:regC}

We will prove the result in the case $\lim_{N\to\infty} \gb_N N^{\frac{d}{2\ga}} = \gb \in [0,+\infty)$.
We start with the case 
$\ga>2$ before we turn to the case $\ga<2$.
We again use the notation $\bar Z_{N,\gb_N}$ for $Z_{N,\gb_N}^{\go, h=\mu}$.

\subsection{The case $\ga \in ( 2\vee \frac{d}{2}, d)$}\label{sec:thmCpart1}

We first focus on the convergence~\eqref{conv:C11}
of Theorem~\ref{thm:C1}.
For simplicity, we treat the case where $\gb_N$ decays at most polynomially; the case where $\gb_N$ decays faster (and thus $\beta=0$) is even simpler.
Note that the condition $\ga \in ( 2\vee \frac{d}{2}, d)$ implies in particular that $d\ge 3$, and recall the definition~\eqref{def:Chi} of~$\cX$ in dimension $d\geq 5$.

To simplify the notation, let us suppose that $\bbV\mathrm{ar}(\omega_x)=1$ and observe that the convergence~\eqref{conv:C11} 
is equivalent to 
\begin{equation}
\label{removelog}
\frac{1}{a_N \gb_N} \Big( e^{-\frac 12 \gb_N^2 \bE[|\cR_N|]} \bar Z_{N,\gb_N} -  1 \Big) \xrightarrow{(d)}
 \begin{cases}
\cN(0,\sigma_d^2) & \text{ if } d =3,4 \,,\\
\cX & \text{ if } d\geq 5\, .
\end{cases} 
\end{equation}
since both imply that $e^{-\frac12 \gb_N^2 \bE[|\cR_N|]}  \bar Z_{N,\gb_N}$ converges to $1$ in probability.


\smallskip
{\bf Step 1. Truncation of the environment.}
Let us set
\begin{equation}
\label{def:kN}
k_N:= (\log N)^\eta N^{\frac{d}{2\ga}}  
\end{equation}
where $\eta \in (\frac{d}{2\ga},1)$ and let us define the truncated environment 
\begin{equation}
\label{tildeomega}
\tilde\go=(\go-\mu) \mathbbm 1_{\{\go\le k_n\}} \,.
\end{equation}
Let us also define $\lambda_N = \log \bbE[ e^{\gb_N \tilde \go_0}]$.
The next lemma compares the partition function $\bar Z_{N,\gb_N}$ with its counterpart with the truncated environment $\tilde \go$.

\begin{lemma}
\label{lem:truncate}
Let $\ga \in( 2\vee\frac d2, d)$ .
Assume that
$ \gb_N  \leq  c N^{-d/2\ga}$ for some constant $c$.
Then for any $a>0$, we have that $N^{a} (\bar Z_{N,\gb_N} - Z_{N,\gb_N}^{\tilde \go}) \to 0$ in $\bbP$-probability.
\end{lemma}

\begin{proof}
First, we show that
we can reduce the partition function to  trajectories with transversal fluctuations at most $\sqrt{N\log N}$. 
Assume that $\gb_N N^{d/2\ga} \leq c$, with $\ga> d/2$, then
for any $A_N\geq 1$
\begin{equation}
\label{reduction0}
\bbP\left(  \bar Z_{N,\gb_N} \big( M_{N}\geq A_N \,\sqrt{N\log N} \big)\geq e^{-c_{1}A_N^{2} \log N}  \right)\leq c_{2} A_N^{-\nu} \, .
\end{equation}
This result is identical to Proposition~\ref{prop:Z>Ar}, replacing $N^{\xi}$ with $\sqrt{N\log N}$ (so $N^{2\xi-1}$ is replaced with $\log N$): the proof is identical and relies on Lemma~\ref{lem:S<qr} (in which we can take $h_N = \sqrt{N\log N}$), so we omit it.
Thanks to~\eqref{reduction0}, we can therefore choose a sequence $A_N$ going to infinity (let us take $A_N =\log\log N$) in order to get that for any $a>0$, $N^a Z_{N,\gb_N}^{\go,h=\mu}(M_{N}\geq A_N \,\sqrt{N\log N})$
goes to $0$ in probability.
 
On the other hand, using Cauchy-Swartz inequality we also have that 
\begin{equation}
\label{reduction3}
\begin{split}
\bbE \big[ Z_{N,\gb_N}^{\tilde \go}\big(M_N > A_N\sqrt{N\log N}\big) \big] 
&  = \bE\big[ e^{\lambda_N |\cR_N|} \ind_{\{M_N > A_N \sqrt{N\log N}\}}\big] \\
&\leq  \bE\big[ e^{2\lambda_N |\cR_N|} \big]^{1/2} \bP \big(M_N > A_N\sqrt{N\log N}\big)^{1/2} \, .
\end{split}
\end{equation}
Now, since $\ga>2$, $\lambda_N \sim \frac12 \gb_N^2$ (see Lemma \ref{lem:moments}).
In particular, since $\gb_N \leq c N^{-d/2\ga}$, we get that 
 $|\lambda_N | \leq C' N^{-d/\ga}$,
so $\lambda_N |\cR_N| \leq \lambda_N N $ goes to $0$ as $N\to\infty$ since $\ga<d$.
Also, the last probability in~\eqref{reduction3} is bounded by $\exp(- cA_N^2 \log N) = N^{-c A_N^2}$.  
By Markov's inequality,
we therefore get that $N^a  Z_{N,\gb_N}^{\tilde \go}\big(M_N > A_N\sqrt{N\log N}\big)$ goes to $0$ in probability, using again that $A_N\to \infty$.

All together, it remains to prove that
\begin{equation}
\label{truncated}
N^a \Big(  \bar Z_{N,\gb_N}  \big(M_N \leq A_N\sqrt{N\log N}\big) -  Z_{N,\gb_N}^{\tilde \go}\big(M_N \leq A_N\sqrt{N\log N}\big) \Big) \to 0 \qquad \text{in $\bbP$-probability}.
\end{equation}
But a union bound simply gives that
\begin{equation}
\label{probatrunc}
\begin{split}
\bbP\Big( \exists \, x , \|x\| \leq A_N\sqrt{N\log N} \,, 
\tilde \go_x \neq \go_x-\mu \Big)
&\leq A_N^d (N \log N)^{d/2} \bbP(\go >k_N)\\
&\leq c' A_N^d (N \log N)^{d/2}   (\log N)^{- \ga \eta} N^{-d/2} \, .
\end{split}
\end{equation}
where we used that $k_N  =  (\log N)^{\eta} N^{d/2\ga}$ and our assumption~\eqref{eq:tailOmega}. 
Since $\eta >d/2\ga$, choosing~$A_N=\log \log N$ we have that this probability goes to~$0$.
This proves~\eqref{truncated} and concludes the proof of Lemma~\ref{lem:truncate}
\end{proof}

\smallskip
{\bf Step 2. Convergence of the partition function with truncated environment.}
To conclude the proof of Theorem~\ref{thm:C1}-\eqref{conv:C11}, and in view of~\eqref{removelog} and Lemma~\ref{lem:truncate} (recall we assumed that~$\gb_N$ decays at most polynomially),
we simply need to show the following lemma.
For simplicity, let us assume that $\bbV\mathrm{ar}(\go)=1$ and let us denote
$\Xi_N := \frac12 \bbV\mathrm{ar}(\go) \gb_N^2 \bE[|\cR_N|]$ the centering term in Theorem~\ref{thm:C1}-\eqref{conv:C11}; 
note that (cf. Remark \ref{rem:betaNERN}) $\lim_{N\to\infty} e^{\Xi_N} =1$.
Recall the definition~\eqref{def:Chi} of $\cX$, and recall that 
we have set $a_N= N^{1/4}$ for $d=3$, $a_N =(\log N)^{1/2}$ for $d=4$ and $a_N =1$ for $d\geq 5$.
\begin{lemma}
\label{lem:reduction0}
Assume that $\gb_N \leq c N^{-d/2\ga}$.
If $\ga\in (2\vee \frac{d}{2}, d)$, then we have the following convergence in distribution
\begin{equation}
\label{reduction1}
 \frac{1}{a_N \gb_N } \Big( Z_{N,\gb_N}^{\tilde \go} -  e^{\Xi_N} \Big)
\overset{(d)}{\longrightarrow}
\begin{cases}
\cN(0,\sigma_d^2)  &\quad \text{ if } d=3,4, \\
\cX & \quad \text{ if } d\geq 5 \, .
\end{cases}
\end{equation}
If $d\geq 5$ and $\ga \in (\frac{d}{d-2}, \frac{d}{2})$, then we have 
$\frac{1}{\gb_N }( Z_{N,\gb_N}^{\tilde \go}  - 1 )
  \overset{(d)}{\longrightarrow}  \cX $.
\end{lemma}

%

\begin{remark}We included here the result for $\ga \in (\frac{d}{d-2},\frac{d}{2})$ since the proof is identical to that of~\eqref{reduction1}; this will be useful to treat the case $\ga<d/2$. \end{remark}

As a preliminary to the proof of Lemma~\ref{lem:reduction0}, let us collect some important estimates.
Thanks to Lemma~\ref{lem:moments}, we get that
$|e^{\lambda_N} -1| \leq C \gb_N^2$ and $\lambda_N \sim \frac12 \gb_N^2$ when $\ga >2$ and
$|e^{\lambda_N} -1| \leq C  e^{c (\log N)^{\eta}} \gb_N^\ga$
when $\ga \in (1,2]$.
In particular, we have 
\begin{equation}
\label{lambdaN}
\lambda_N \leq C \gb_N^2  \ \text{ if } \ga \in (2,d) \,, \qquad
\lambda_N \leq C e^{\gb_N k_N} \gb_N^\ga  \ \text{ if } \ga \in (1,2] \, . 
\end{equation}
Note that~\eqref{lambdaN} implies that $\lim_{N\to\infty} \lambda_N N =0$.

Also, define $\epsilon_x:=\exp(\beta_N\tilde{\omega}_x-\lambda_N)-1$, which will be used in both lemmas.
Note that $\bbE[ \gep_x ]=0$ and that thanks to Lemma~\ref{lem:moments}, we get as for $\lambda_N$
\begin{equation}
\label{epsilonN}
\bbE[(\gep_0)^2] \leq C \gb_N^2  \ \text{ if } \ga \in (2,d) \,, \qquad
\bbE[(\gep_0)^2]  \leq C e^{\gb_N k_N} \gb_N^\ga  \ \text{ if } \ga \in (1,2] \, . 
\end{equation}
In fact, we have that $\bbE[(\gep_0)^2] \sim \gb_N^2$ when $\ga >2$.
The bounds~\eqref{lambdaN}-\eqref{epsilonN}, combined 
with the fact that  $\gb_N k_N \leq c (\log N)^{\eta}$,  will be used extensively throughout the proof. 

\begin{proof}[Proof of Lemma~\ref{lem:reduction0}]\label{proofLemreduction0}
Note that $e^{u\ind_{\{x\in \cR_N\}}}=1+(e^u-1)\ind_{\{x\in \cR_N\}}$ for any $u\in \bbR$.
Writing this with $u = \gb_N \tilde \go_x-\lambda_N$ so that $e^{u}-1 =\gep_x$,
we expand the product in the partition function as follows:
\begin{align}\label{polexpansion}
 Z_{N,\gb_N}^{\tilde\go}  &=
 \bE\Big[e^{\lambda_N |\cR_{N} |}\prod_{x\in  \bbZ^d  }(1+\epsilon_x \mathbbm{1}_{\{x\in\cR_{N}\}})\Big]
 =\bE \big[  e^{\lambda_N |\cR_{N} |} \big]  + \sum_{x\in \bbZ^d} \gep_x \bE\big[ e^{\lambda_N |\cR_{N}|}  \ind_{\{x\in \cR_{N}\}}  \big]  +\bY_{N}\,, 
\end{align}
where we set 
\begin{equation}\label{def:RN}
\bY_{N}:=\sum_{k\ge 2}\sumtwo{(x_1,\dots,x_k)\in ( \bbZ^d )^k}{x_i\neq x_j,\,i,j=1,\dots,k}\Big(\prod_{i=1}^k\epsilon_{x_i}\Big)\bE\Big[ e^{\lambda_N |\cR_{N}  |}\ind_{\{ (x_1,\dots,x_k)\in\cR_{N} \}}\Big] \,,
\end{equation}
where $(x_1,\dots,x_k)\in \cR_{N}$ means that the points $x_1, \ldots, x_k$ are visited in this given order; in particular, we do not have the combinatorial term $k!$.

We now control all terms, starting with $\bY_N$.

\smallskip
{\it Last term in~\eqref{polexpansion}.} We show that 
$\frac{1}{\gb_N a_N} \bY_N$ 
converges in probability to zero.
Since $\bbE[\gep_x]=0$, we only need to control the second moment of $\bY_{N}$.
Because the $\gep_x$ are independent, 
we easily get that 
\begin{equation}
\label{firstERN2}
\bbE\big[( \bY_{N} )^2\big]=\sum_{k\ge 2}
\sumtwo{(x_1,\dots,x_k)\in ( \bbZ^d )^k}{x_i\neq x_j,\,i,j=1,\dots,k}\bbE \big[ (\gep_{0})^2 \big]^k \bE\Big[ e^{\lambda_N |\cR_{N}|}\ind_{\{ (x_1,\dots,x_k)\in\cR_{N} \}}\Big]^2  \, .
\end{equation}
Now, using the fact that $|\cR_{N}|\leq N$ and Markov's property, we get
\begin{align*}
\sumtwo{(x_1,\dots,x_k)\in ( \bbZ^d)^k}{x_i\neq x_j,\,i,j=1,\dots,k}\bE\Big[ e^{\lambda_N |\cR_{N}|}\ind_{\{ (x_1,\dots,x_k)\in\cR_{N} \}}\Big]^2 
& \leq e^{2 N\lambda_N} \sumtwo{(x_1,\dots,x_k)\in (\bbZ^d)^k}{x_i\neq x_j,\,i,j=1,\dots,k}\bP ( (x_1,\dots,x_k)\in\cR_N  )^2 \\
& \leq  C \Big( \sum_{x\in \Z^d} \bP(x\in \cR_N)^2 \Big)^k  \, ,
\end{align*}
where we also used that $\lambda_N N$ goes to $0$, as seen above.
We denote $J_N := \sum_{x\in \Z^d} \bP(x\in \cR_N)^2$.
Then, we have $J_N \sim c_d a_N^2$ as $N\to\infty$,
as outlined in Lemma~\ref{lem:JN} in the appendix.
Using the bound~\eqref{epsilonN} on $\bbE[(\gep_0)^2]$, we conclude that by $\bbE[(\gep_x)^2]J_N\to0$,
\begin{equation}
\label{secondERN2}
\frac{1}{\gb_N^2 a_N^2}\bbE\big[(\bY_{N})^2\big]\le \frac{C}{\gb_N^2 a_N^2}\sum_{k\ge 2}  \big(\bbE \big[ (\gep_{0})^2 \big] J_N \big)^k
\le C 
\begin{cases}
\gb_N^2  J_N \  &\text{ if } \ga \in (2,d) \, ,\\
 e^{2c (\log N)^{\eta}} \gb_N^{2(\ga-1)}  J_N \  &\text{ if } \ga \in (1,2]\, .
\end{cases}
\end{equation}
This readily goes to $0$ as $N\to\infty$.

\smallskip
{\it First term in~\eqref{polexpansion}.}
We  now show that 
\begin{equation}
\label{replacelambdaN}
\lim_{N\to\infty} \frac{1}{a_N \gb_N}\Big|  \bE\big[ e^{\lambda_N \cR_N} \big] - e^{\Xi_N \ind_{\{\ga >2\}}}  \Big| =0 \,.
\end{equation}
Let us start with the case $\ga>2$.
 Lemma~\ref{lem:moments}  gives in that case that $|\lambda_N - \tfrac12 \gb_N^2|\leq C  e^{ c' (\log N)^{\eta}}  \gb_N^{3\wedge \ga} $.
Hence, we get
\begin{align*}
\frac{1}{\gb_N a_N} \Big| \bE\big[e^{\lambda_N | \cR_{N} |} \big]  - \bE\big[e^{ \frac12 \gb_N^2 | \cR_{N} |} \big] \Big|
& \leq \frac{1}{\gb_N a_N}   \big| \lambda_N - \tfrac12 \gb_N^2 \big| \, N \leq C e^{c'(\log N)^{\eta}} \frac{1}{a_N} \gb_N^{2\wedge (\ga-1)} N  \,,
\end{align*}
which goes to $0$ as $N\to\infty$.
Indeed, if $\ga\in (2,3)$, we have  that $\gb_N^{\ga-1} \leq c N^{- \frac{d(\ga-1)}{2\ga}}$: in dimension $d \geq4$ we get that
$\gb_N^{\ga-1} N \leq c N^{1- 2 (\ga-1)/\ga}$ which goes to $0$ since $\ga>2$;
in dimension $d=3$, using that $a_N=N^{1/4}$ we get that $a_N^{-1}\gb_N^{\ga-1} N \leq N^{\frac{3}{4} (1 - 2(\ga-1)/\ga)}$ which also  goes to $0$ since $\ga>2$.
If $\ga \in [3,d)$, we simply use that $\gb_N^2 N \leq C N^{1-d/\ga}$, which goes to $0$ since $\ga<d$.
Additionnally, recalling that $\Xi_N = \frac12 \gb_N^2 \bE[|\cR_N|]$, we have
\begin{align*}
\frac{1}{\gb_N a_N} \Big| \bE\Big[e^{ \frac12 \gb_N^2 | \cR_{N} |} \Big] -  e^{ \Xi_N }  \Big|
&\leq \frac{C}{\gb_N a_N}  e^{ \frac12 \gb_N^2 \bE[ | \cR_{N} | ] }  \times \gb_N^4 \mathbf{V}\mathrm{ar}\big( |\cR_N|  \big)  \notag \\
&\leq C' \gb_N^3 N \log N \,,
\end{align*}
where we expanded $e^{\frac{1}{2} \gb_N^2 (|\cR_N| - \bE[|\cR_N|])}-1$ to get the first inequality (recall that $\gb_N^2 N$ goes to $0$), 
and used that $\mathbf{V}ar(|\cR_N|)  \sim c_3 N \log N$ in dimension $d=3$ and $\mathbf{V}ar(|\cR_N|) \sim c_4 N$ in dimension $d\geq 4$ (see~\cite{JP71}), together with the definition of $a_N$.
This proves~\eqref{replacelambdaN} in the case $\ga \in(2,d)$.

In the case $\ga \in (\frac{d}{d-2},2]$ (in particular $d\geq 5$), we use~\eqref{lambdaN} to get that
\[
\frac{1}{\gb_N}  \Big| \bE\big[e^{\lambda_N | \cR_{N} |} \big]  - 1 \Big|
\leq c \gb_N^{-1} \lambda_N  N \leq C e^{c(\log N)^{\eta}} \gb_N^{\ga-1} N \, ,
\] 
which goes to $0$ as $N\to\infty$, because $\frac{d}{2\ga} (\ga-1) >1$.
This proves~\eqref{replacelambdaN} in the case $\ga \in (\frac{d}{d-2},2]$.

\smallskip
{\it Second term in~\eqref{polexpansion}.}
Let us rewrite it as
\begin{equation}
\label{convnormal1}
 \frac{1}{a_N} \sum_{x\in \bbZ^d} \frac{\gep_x}{\gb_N}  \bP(x\in \cR_N) +  \frac{1}{\gb_N a_N}\sum_{x\in  \bbZ^d} \gep_x \bE\Big[\big (e^{\lambda_N |\cR_{N}|} -1 \big) \ind_{\{x\in \cR_{N}\}}  \Big] \, .
\end{equation}

The second term in~\eqref{convnormal1} goes to zero in probability. Indeed, it is centered in expectation, and its second moment is bounded by
\[
\frac{1}{\gb_N^2 a_N^2}   \sum_{x\in \bbZ^d} \bbE[ (\gep_0)^2 ] \big (e^{\lambda_N N} -1 \big)^2  \bP(x\in \cR_N)^2 \leq 
C \frac{\bbE[ (\gep_0)^2] (\lambda_N N)^2 J_N}{\gb_N^2 a_N^2} \, .
\]
Using that  $J_N \sim c_d a_N^2$ as $N\to\infty$ (cf.\ Lemma~\ref{lem:JN}) and that $\lambda_N ,\bbE [ (\gep_{0})^2 ] \leq C e^{c (\log N)^{\eta}}\gb_N^{\ga\wedge 2}$, see~\eqref{lambdaN} and~\eqref{epsilonN}, this is bounded by a constant times $ e^{c(\log N)^{\eta}} \gb_N^{3 (\ga\wedge 2) -2} N^2$.
If $\ga>2$ this is bounded by $N^{o(1)} (\gb_N^2 N)^2$ so this goes to $0$; if $\ga \in (\frac{d}{d-2}, 2]$ this is bounded by $N^{o(1)}(\gb_N^{\ga-1} N)^2$, so this also goes to~$0$, as seen above.

We now
rewrite the first term in~\eqref{convnormal1} as
\begin{equation}
\label{gotozero}
\frac{1}{a_N} \sum_{x\in \bbZ^d} ( \go_x -\mu) \bP(x\in \cR_N)
 + \frac{1}{a_N \gb_N}  \sum_{x\in \bbZ^d} \big( \gep_x -\gb_N( \go_x -\mu) \big) \bP(x\in \cR_N) \, ,
\end{equation}
and we now show that the second term goes to $0$ in probability.

If $\ga \in (\frac{d}{d-2},2]$ ($d\geq 5$), we control the expectation of its absolute value. Using that $|\gb_N \tilde \go_x|\leq \gb_N k_N \leq c (\log N)^{\eta}$ and recalling that $\gep_x = e^{\gb_N \tilde \go_x -\lambda_N} -1$,
we can perform a Taylor expansion to get that 
\begin{align*}
\frac{1}{\gb_N} \bbE\big[ \big| \gep_x - \gb_N( \go_x -\mu) \big| \big]
&\leq \frac{1}{\gb_N} \Big(  \gb_N \bbE[ |\go_x-\mu| \ind_{\{\go_x>k_N\}}  ]  + \lambda_N+ Ce^{2c (\log N)^{\eta}} ( \gb_N^2\bbE[ \tilde \go_x^2] + \lambda_N^2 ) \Big) \\
& \leq k_N^{1-\ga} +  e^{c' (\log N)^\eta} \gb_{N}^{\ga -1}  \, ,
\end{align*}
where we used that $\bbE[ \tilde \go_x^2]\leq C k_N^{2-\ga} \log k_N \leq C' \gb_N^{\ga-2}  (\log N)^{C} $, together with a similar bound on $\lambda_N$, see~\eqref{lambdaN}.
Therefore, the expectation of the absolute value of the second term in~\eqref{gotozero}
is bounded by a constant times $k_N^{1-\ga} N + e^{c' (\log N)^\eta}  \gb_{N}^{\ga -1} N$: this goes to $0$ as $N\to\infty$, thanks to the condition $\ga > d/(d-2)$.
By Markov's inequality, the second term in~\eqref{gotozero} goes to $0$ in probability.

If $\ga\in (2,d)$, since the expectation of the second term in~\eqref{gotozero} is $0$, we control its second moment: by independence of the $\go_x$, it is equal to
\[
\frac{1}{a_N^2 \gb_N^2} \sum_{x\in \bbZ^d} \bbE\Big[ \big( \gep_x -\gb_N ( \go_x -\mu) \big)^2 \Big]\bP(x\in \cR_N)^2 \leq \frac{C}{ \gb_N^2} \bbE\Big[ \big( \gep_x -\gb_N ( \go_x -\mu) \big)^2 \Big]\,, 
\]
where we have used that $J_N = \sum_{x\in \bbZ^d} \bP(x\in \cR_N) \leq c a_N^2$, see Lemma~\ref{lem:JN}.
Now, we can use the same Taylor expansion as above, to get that 
\begin{align*}
\frac{1}{\gb_N^2}\bbE\Big[  \big( \gep_x - \gb_N (\go_x -\mu) \big)^2\Big] &\leq \frac{C}{\gb_N^2} \Big(  \gb_N^2 \bbE\big[ (\go_x-\mu)^2 \ind_{\{\go_x>k_N\}} \big]  +  \lambda_N^2 + Ce^{4c(\log  N)^{\eta}}  \big( \gb_N^4 \bbE[\tilde \go_x^4] +\lambda_N^4 \big)\Big) \\
&\leq  C k_N^{2-\ga} + Ce^{c'(\log N)^{\eta}}  \gb_N^{\ga\wedge 4 -2} \, .
\end{align*}
where we used here that $\lambda_N \sim \frac12 \gb_N^2$ as $N\to\infty$ and that $\bbE[ \tilde \go_x^4]$ is bounded by $C k_N^{4-\ga} \log k_N \leq C \gb_N^{\ga-4} (\log N)^{C'}$
if $\ga\leq 4$ and by a constant if $\ga>4$.
This shows that the variance of the second term in~\eqref{gotozero}
goes to $0$, so this term goes to $0$ in probability.

\smallskip
All together,  combining~\eqref{polexpansion} with~\eqref{secondERN2}-\eqref{replacelambdaN} and \eqref{convnormal1}-\eqref{gotozero}, we have shown that 
\[
\frac{1}{a_N\gb_N}\Big( Z_{N,\gb_N}^{\tilde \go} - e^{\Xi_N \ind_{\{\ga>2\}}}  - \sum_{x\in \bbZ^d}  \gb_N ( \go_x -\mu) \bP(x\in \cR_N) \Big)
\xrightarrow{\bbP} 0 \, .
\]
It therefore only remains to show that
\begin{equation}
\label{interm}
\frac{1}{a_N} \sum_{x\in \bbZ^d} ( \go_x -\mu) \bP(x\in \cR_N)  \xrightarrow{(d)}
\begin{cases}
\cN(0,\sigma_d^2) &\  \text{ if }  d=3,4,\\
\cX & \ \text{ it } d\geq 5 \, .
\end{cases}
\end{equation}
Note that in~Lemma~\ref{lem:reduction0}, when $d\geq 5$ and $\ga \in (2,\frac d2)$ we used a centering equal to $1$ instead of~$e^{\Xi_N}$.
This is not an issue since in that case we have that $\gb_N^{-1}(e^{\Xi_N} -1) \leq   c \gb_N N$ goes to $0$ as $N\to\infty$.

\smallskip
In dimension $d\geq 5$, \eqref{interm}
follows simply by monotone convergence, recalling that $\cX$
is a.s.\ finite if $\ga> \frac{d}{d-2}$, see Proposition~\ref{prop:finiteSum}.

In dimension $d=3,4$, 
the random variables $a_N^{-1} ( \go_x -\mu) \bP(x\in \cR_N)$ 
are independent, centered and have variance $a_N^{-2} \bP(x\in \cR_N)^2$.
Since $a_N \to\infty$, this goes to $0$ uniformly for $x\in \bbZ^d$.
Hence, we directly get~\eqref{interm}, for instance with a characteristic function analysis
with the constant $\sigma_d^2:= \lim_{N\to\infty} a_N^{-2}\sum_{x\in \bbZ^d} \bP(x\in\cR_N)^2$, recall Lemma~\ref{lem:JN}.
\end{proof}

\subsection{The case $\ga \in (\frac{d}{2},2)$}
\label{sec:thmCpart2}
Note that this regime is nonempty only for $d=2,3$.
Also, since $\ga>\frac{d}{2}\geq1$, we have $\mu:=\bbE[\go_0]<+\infty$.
Let us stress that, similarly as~\eqref{removelog} above, the convergence~\eqref{conv:C12}
is equivalent to the convergence in distribution
\begin{equation}
\label{removelog2}
\frac{v_N}{ \gb_N N^{\frac{d}{2\ga}}} \big( \bar Z_{N,\gb_N}  -1\big) \xrightarrow{(d)} \cW_{\gb} \,.
\end{equation}
The steps are the same as in Section~\ref{sec:thmCpart1}. The case of the dimension $d=3$ is very similar, but some adaptations are needed in dimension $d=2$.

\smallskip
{\bf Step 1. Truncation of the environment.}
Let us set
\begin{equation}
\label{def:kN2}
k_N := \begin{cases}
 (\log N)^{\eta} N^{\frac{d}{2\ga}}  &\quad \text{if } d=3 \, ,\\
 (\log \log N)^{\eta} N^{\frac{d}{2\ga}}  &\quad \text{if } d=2 \, ,
\end{cases}
\end{equation}
where $\eta \in (\frac{d}{2\ga},1)$.
We again define the truncated environment 
$\tilde\go=(\go-\mu) \mathbbm 1_{\{\go\le k_n\}}$
and  we set $\lambda_N = \log \bbE[ e^{\gb_N \tilde \go_0}]$.
The next lemma is analogous to Lemma~\ref{lem:truncate}.

\begin{lemma}
\label{lem:truncate2}
Let $\ga \in (2,\frac d2)$.
Assume that $\gb_N  \leq c  N^{- d/2\ga}$ for some constant $c$.
Then for any $a>0$, we have that:
\begin{itemize}
\item[(i)] if $d=3$, $N^a ( \bar Z_{N,\gb_N} - Z_{N,\gb_N}^{\tilde \go}) \to 0$ in $\bbP$-probability;
\item[(ii)] if $d=2$, $(\log N)^a (\bar Z_{N,\gb_N} - Z_{N,\gb_N}^{\tilde \go})\to 0$ in $\bbP$-probability.
\end{itemize}
\end{lemma}

Before we start the proof, let us stress that Lemma~\ref{lem:moments} gives some estimates on~$\lambda_N$: since $\ga\in(1,2)$, we have
\begin{equation}
\label{lambdabound}
\lambda_N \leq C e^{\gb_N k_N} \gb_N^{\ga}
\leq C e^{\gb_N k_N} N^{-d/2}\,,
\quad 
\text{with } \gb_N k_N =
\begin{cases}
(\log N)^{\eta}  & \text{ if } d=3 \,, \\
(\log \log N)^{\eta} & \text{ if } d=2 \, .
\end{cases}
\end{equation}

\begin{proof}
In the case of dimension $d=3$, the proof is identical  to that of Lemma~\ref{lem:truncate}:
the only difference here is that we do not have $\lambda_N \leq C\gb_N^2$. Instead, we use~\eqref{lambdabound} instead which gives that $\lambda_N N \leq Ce^{(\log N)^{\eta}}  N^{1-\frac{d}{2}}$  goes to $0$ as $N\to\infty$; this was used to bound~\eqref{reduction3}.

We therefore focus on the case of dimension $d=2$;
the idea of the proof is identical, with some adaptation.
First of all, analogously to~\eqref{reduction0}, we get that 
for any sequence $A_N\geq 1$
\begin{equation}
\label{reduction01}
\bbP\left(  \bar Z_{N,\gb_N} \big( M_{N}\geq A_N \,\sqrt{N\log \log N} \big)\geq e^{-c_{1}A_N^{2} \log \log N} = (\log N)^{-c_1 A_N^2} \right)\leq c_{2} A_N^{-\nu} \, .
\end{equation}
Again, this is identical to Proposition~\ref{prop:Z>Ar}, replacing $N^{\xi}$
with $\sqrt{N \log \log N}$: in particular, this relies on Lemma~\ref{lem:S<qr} in which we can take $h_N = \sqrt{N\log N \log N}$.
Therefore, choosing a sequence $A_N \to \infty$ (we take $A_N = \log \log \log N$), 
we get that $(\log N)^a Z_{N,\gb_N}^{\go,h=\mu}(M_{N}\geq A_N \,\sqrt{N\log \log N})$
goes to $0$ in probability.

Then, as in~\eqref{reduction3},
we have that 
\begin{equation}
\label{reduction31}
\begin{split}
\bbE \big[ Z_{N,\gb_N}^{\tilde \go}\big(M_N > A_N\sqrt{N\log \log  N}\big) \big]
&\leq  \bE\big[ e^{2\lambda_N |\cR_N|} \big]^{1/2} \bP \big(M_N > A_N\sqrt{N\log \log  N}\big)^{1/2} \\
& \leq \bE\big[ e^{2\lambda_N |\cR_N|} \big]^{1/2} e^{ - c A_N^2 \log \log N}
 \, .
 \end{split}
\end{equation}
The main difference here is that we cannot use the bound
$\lambda_N |\cR_N| \leq \lambda_N N$
since we do not have $\lambda_N N \to 0$ as $N\to\infty$, see~\eqref{lambdabound}.
However, we may use the fact that $|\cR_N|$ is of the order $N/\log N$  to obtain that $\lambda_N |\cR_N| \to 0$ with high $\bP$-probability.
More precisely, using that $\bE[|\cR_N|] \leq C N/\log N$ (see e.g.~\cite[Eq.~(5.3.39)]{X15})
we have that, for a sufficiently large $C'$
\begin{equation}\label{probadev}
\bP\Big( |\cR_N| \geq   2C' \frac{N}{\log N}\Big) \leq 
\bP\Big( |\cR_N|-\bE[|\cR_N|]\geq C'\frac{N\log\log N}{(\log N)^{2}} \Big) \leq e^{- c (\log N)^4} \, .
\end{equation}
where we used~\cite[Thm.~8.5.1]{X15} for the last inequality.
We therefore get that 
\[
\bE\big[ e^{2\lambda_N |\cR_N|} \big] \leq e^{4C' \lambda_N N/\log N } + e^{2\lambda_N N} e^{- c (\log N)^4} \, .
\]
Now, thanks to~\eqref{lambdabound}, we get that $\lambda_N N = o(\log N)$  as $N\to\infty$.
Hence, from~\eqref{reduction31}, we get
that
\[
\bbE \big[ Z_{N,\gb_N}^{\tilde \go}\big(M_N > A_N\sqrt{N\log \log  N}\big) \big] \leq C (\log N)^{- cA_N^2}
\]
so that $(\log N)^a Z_{N,\gb_N}^{\tilde \go}\big(M_N > A_N\sqrt{N\log \log  N}\big) \to 0$ in $\bbP$-probability thanks to Markov's inequality.

Finally, notice that similarly to~\eqref{probatrunc}, we have that
\begin{equation}
\bbP\Big( \exists \, x , \|x\| \leq A_N\sqrt{N\log \log  N} \,, 
\tilde \go_x \neq \go_x-\mu \Big)
\leq c' A_N^d (N \log \log N)^{d/2}   (\log \log N)^{- \ga \eta} N^{-d/2} \, ,
\end{equation}
which goes to $0$ as $N\to\infty$ since $\ga \eta >d/2$, and $A_N = \log \log\log N$.
Therefore, we get that 
$\bar Z_{N,\gb_N}\big(M_N > A_N\sqrt{N\log \log  N}\big)$ is equal to
$Z_{N,\gb_N}^{\tilde \go}\big(M_N > A_N\sqrt{N\log \log  N}\big)$
with $\bbP$-probability going to $1$.
This concludes the proof of Lemma~\ref{lem:truncate2}.
\end{proof}

\smallskip
{\bf Step 2. Convergence of the partition function with truncated environment.}
To conclude the proof of Theorem~\ref{thm:C1}-\eqref{conv:C12},
and in view of~\eqref{removelog2} and Lemma~\ref{lem:truncate2},
it remains to prove the following lemma.
(In dimension $d=2$, we need to assume that $\gb_N N^{d/2\ga}$ decays slower than any power of $\log N$, but this does not hide anything deep; one simply needs to use a more restrictive truncation, but this actually simplifies many of the arguments so we do not treat this case.)

\begin{lemma}
\label{lem:reduction11}
Let $\ga\in (\frac{d}{2},2)$.  If $\lim_{N\to\infty} \gb_N N^{d/2\ga} =\gb \in[0,+\infty)$ then we have the following convergence in distribution
\begin{equation*}
\frac{v_N}{\beta_N N^{d/2\alpha}} \big(Z_{N,\gb_N}^{\tilde \go} -1 \big)\overset{(d)}{\longrightarrow}\cW_{\beta}\,.
\end{equation*}
\end{lemma}

\begin{proof}
We use the same notation as in Step~2 of Section~\ref{sec:thmCpart1}.
Defining $\epsilon_x:=\exp(\beta_N\tilde{\omega}_x-\lambda_N)-1$, and using a polynomial chaos expansion similar  to~\eqref{polexpansion}, we get
\begin{equation}\label{polychaos}
Z_{N,\beta_N}^{\tilde{\omega}} -1=
\bE \big[  e^{\lambda_N |\cR_{N} |} -1 \big]  + \sum_{x\in \bbZ^d} \gep_x \bE\big[ e^{\lambda_N |\cR_{N}|}    \ind_{\{x\in \cR_N\}} \big]
+ \bY_{N}.
\end{equation}
where we defined $\bY_N$ as in~\eqref{def:RN}. We first prove that the last term is negligible: the computation in $d=3$ follows closely the computations \eqref{firstERN2}--\eqref{secondERN2}, while dimension $d=2$ is more delicate. In the second part of the proof we show that the first two terms give the main contribution to the convergence. This part is quite technical and we split it into three steps.

\smallskip
{\it Last term in~\eqref{polychaos}.}
Let us prove that $v_N \bY_{N} /(\beta_N N^{d/2\alpha})$ 
goes to $0$ in probability.

We start with the case of dimension $d=3$, which we essentially already treated.
Reasoning as in~\eqref{firstERN2}--\eqref{secondERN2}, we computed the second moment of $\bY_N$
in the case $\ga \in (1,2)$: we have 
that 
\begin{equation}
\label{YNd=3}
\bbE[\bY_N^2]  \leq C \sum_{k\geq 2} \big( \bbE[ \gep_0^2 ] J_N \big)^{k} \leq  C e^{2c(\log N)^{\eta}} \gb_N^{2\ga} N \,,
\end{equation}
where we bounded $|\cR_N|\le N$, used that  $J_N \leq c N^{1/2}$ in dimension $d=3$ (see Lemma~\ref{lem:JN}) and the bound~\eqref{epsilonN}.  Since $v_N = N^{1/2}$, we obtain
\begin{align}
\Big(\frac{v_N}{\beta_N N^{d/2\alpha}} \Big)^2\bbE[\bY_{N}^{2}]
\leq C  e^{2(\log N)^{\eta}}  N^{1- \frac{d}{\ga}}  \gb_N^{2(\ga-1)} 
\leq C' e^{c(\log N)^{\eta}} N^{2-d} \,.
\end{align}
This term goes to $0$, which concludes the proof in the case of dimension $d=3$.

Turning to the case of dimension $d=2$, we need to be more careful: we cannot simply bound $|\cR_N|$ by $N$ since we do not have $\lambda_N N \to 0$ anymore.
The calculation of $\bbE[\bY_N^2]$ in~\eqref{firstERN2} remains valid. Then, decomposing the expectation according to whether $|\cR_N| \leq 2C' N /\log N$ or not, we obtain
\begin{equation}
\label{xinRd=2}
\begin{split}
\bE & \Big[ e^{\lambda_N |\cR_{N}|}\ind_{\{ (x_1,\dots,x_k)\in\cR_{N} \}}\Big]^2  \\
& \leq \bigg( e^{2C' \frac{\lambda_N N}{\log N}} \bP\big((x_1,\dots,x_k)\in\cR_{N} \big)   + e^{\lambda_N N}  \bP\Big((x_1,\dots,x_k)\in\cR_{N} ; |\cR_N| \geq  \frac{2C' N}{\log N}\Big) \bigg)^2 \\
& \leq  2 e^{ 4C'  \frac{ \lambda_NN}{\log N}} \bP\big((x_1,\dots,x_k)\in\cR_{N} \big)^2 + 2 e^{2\lambda_N N} \bP\big((x_1,\dots,x_k)\in\cR_{N} \big) \bP\Big( |\cR_N| \geq  \frac{2C' N}{\log N}\Big) \,,
\end{split}
\end{equation}
where for the last inequality we used that $(a+b)^2 \leq 2 (a^2 + b^2)$ together with Cauchy--Schwarz inequality.
Now, we have that  $\lambda_N N = o(\log N)$  (recall~\eqref{lambdabound}) and also that $ \bP( |\cR_N| \geq2 C' N/\log N )\leq e^{- c (\log N)^4}$ if $C'$ has been fixed large enough, thanks to~\eqref{probadev}.
Hence, summing~\eqref{xinRd=2} over $(x_1,\ldots,x_k)$ and using Markov's property, we finally get that 
\begin{align*}
\sumtwo{(x_1,\dots,x_k)\in ( \bbZ^d)^k}{x_i\neq x_j,\,i,j=1,\dots,k}\bE\Big[ e^{\lambda_N |\cR_{N}|}\ind_{\{ (x_1,\dots,x_k)\in\cR_{N} \}}\Big]^2  & \leq C  \Big( \sum_{x\in \bbZ^d} \bP(x\in \cR_N)^2\Big)^k
 + e^{-c' (\log N)^4}  \Big( \sum_{x\in \bbZ^d} \bP(x\in \cR_N)\Big)^k
\end{align*}
All together, going back to~\eqref{firstERN2},
we obtain that when $d=2$ the bound~\eqref{YNd=3}
is replaced with
\[
\bbE[\bY_{N}^{2}]\leq  C \sum_{k\geq 2} \big( \bbE[(\gep_0)^2] J_N\big)^k 
+ e^{-c' (\log N)^4} \sum_{k\geq 2} \big( \bbE[(\gep_0)^2] \bE[|\cR_N|]\big)^k \, .
\]
Now, we have that $\bbE[(\gep_0)^2] \leq C e^{(\log \log N)^{\eta}} \gb_N^{\ga}$  by Lemma~\ref{lem:moments}, that $J_N \leq C\frac{N}{(\log N)^2}$ by Lemma~\ref{lem:JN} and that $\bE[|\cR_N|] \leq C \frac{N}{\log N}$ by \cite[Eq.~(5.3.39)]{X15}.
Using that  $v_N =\log N$, we obtain
\begin{equation}
\label{YNto0}
 \Big(\frac{v_N}{ \gb_N N^{\frac{d}{2\ga}}} \Big)^2 \bbE[\bY_{N}^{2}]
\leq C (\log N)^2 e^{2 (\log \log N)^{\eta}} \gb_N^{2(\ga-1)} N^{2-\frac{d}{\ga} } \Big( \frac{1}{(\log N)^4} + e^{-c' (\log N)^4} \frac{1}{(\log N)^2}\Big) \, ,
\end{equation}
and note that $\gb_N^{2(\ga-1)} N^{2-d/\ga} \leq c$.
This term goes to $0$, which concludes the proof when $d=2$.

\smallskip
{\it  First terms in~\eqref{polychaos}.}
We can rewrite these terms as
\begin{align}
\bE& \big[  e^{\lambda_N |\cR_{N} |} -1 \big]  + \sum_{x\in \bbZ^d} \gep_x \bE\big[ \big( e^{\lambda_N |\cR_{N}|}  -1\big) \ind_{\{x\in \cR_N\}}  \big] + \sum_{x\in \bbZ^d} \gep_x \bP(x\in \cR_N) \notag \\
& =
\bE \big[  e^{\lambda_N |\cR_{N} |} -1  - (1-e^{-\lambda_N}) |\cR_N|\big] +  \sum_{x\in \bbZ^d} \gep_x \bE\big[ \big( e^{\lambda_N |\cR_{N}|}  -1\big) \ind_{\{x\in \cR_N\}}  \big]
\label{threetermsd=3} \\
& \hspace{8cm}  +e^{-\lambda_N} \sum_{x\in \bbZ^d} (e^{\gb_N \tilde \go_x} -1) \bP(x\in \cR_N) \,, \notag
\end{align}
where we used that $\gep_x = e^{-\lambda_N} (e^{\gb_N \tilde \go_x} -1) -1+e^{-\lambda_N}$ for the last identity.
Let us denote $\mathbf{I}$, $\mathbf{II}$ and~$\mathbf{III}$ the three terms in~\eqref{threetermsd=3}.
We show that the first two terms are negligible.

\smallskip
{\it Term {\bf I}.}
For the first term in~\eqref{threetermsd=3}, let us first treat the case of dimension $d=3$, which is simpler.
Using that $\lambda_N \to 0$ and $\lambda_N |R_N| \leq \lambda_N N \to 0$, a Taylor expansion gives that it is bounded by a constant times $(\lambda_N N)^2$.
Hence, using that  $\lambda_N \leq C e^{c(\log N)^{\eta}} \gb_N^{\ga}$, we get that
\[
\frac{v_N}{\gb_N N^{\frac{d}{2\ga}}} \mathbf{I}  
\leq  C'  e^{2c(\log N)^{\eta}}  v_N \gb_N^{2\ga-1} N^{2-\frac{d}{2\ga}} \leq   C'  e^{c(\log N)^{\eta}}  N^{\frac52 - d} \,,
\]
where we used that $v_N=N^{1/2}$ in dimension $d=3$. This goes to $0$ as $N\to\infty$.

The case of dimension $d=2$ is a bit more delicate.
On the event $\{|\cR_N| \leq 2C' N/\log N\}$, we can again use a Taylor expansion, since $\lambda_N N/\log N \to 0$, so that the term inside the expectation is bounded in absolute value by a constant times
$\lambda_N^2 N^2 / (\log N)^2$.
On the event $\{|\cR_N| >  2C' N/\log N\}$, we bound the absolute value of the term 
inside the expectation by $e^{N\lambda_N} - 1+c \lambda_N N  \leq C \lambda_N N e^{o(\log N)}$.
We therefore get that 
\begin{align*}
|\mathbf{I}| \leq  C \frac{\lambda_N^2 N^2}{(\log N)^2} +  
  \lambda_N N e^{o(\log N)} \bP\Big( |\cR_N| >  2C' \frac{N}{\log N}\Big) 
\leq \frac{e^{ 2 (\log \log N)^{\eta}} }{(\log N)^2} \gb_N^{2 \ga} N^2
 +  e^{-c' (\log N)^4}  \gb_N^{\ga} N\, ,
\end{align*}
where we used~\eqref{probadev} for the last inequality.
Since $v_N =\log N$, we get that 
\begin{equation}
\label{termI}
\frac{v_N}{\gb_N N^{\frac{d}{2\ga}}} \mathbf{I}   \leq  
\frac{e^{ 2 (\log \log N)^{\eta}} }{\log N}  \gb_N^{2 \ga-1} N^{2-\frac{d}{2\ga}}
 + \log N e^{-c' (\log N)^4}  \gb_N^{\ga-1} N^{1-\frac{d}{2\ga}}
\end{equation}
which goes to $0$ as $N\to\infty$, noticing that $\gb_N^{\ga-1} N^{1-d/2\ga} \leq c$,  $\gb_N^{2 \ga-1} N^{2-\frac{d}{2\ga}}\leq c$ and $\eta\in(\frac{d}{2\ga},1)$.

\smallskip
{\it Term {\bf II}.}
We prove that $v_N \mathbf{II} / (\gb_NN^{d/2\ga})$ goes to $0$ in probability, by controlling its second moment.
Since the $\epsilon_x$ are centered and independent, we have
\begin{equation}\label{IIIsquare}
\bbE[(\mathbf{II})^{2}]\leq\bbE[\epsilon_0^2]\sum_{x\in\bbZ^{d}}\bE\Big[\big(e^{\lambda_N|\cR_N|}-1\big)\mathbbm{1}_{\{x\in\cR_N\}}\Big]^2\, .
\end{equation}

We start with the simpler case of dimension $d=3$.
Using that $e^{\lambda_N|\cR_N|}-1 \leq c \lambda_N N$ and recalling the definition $J_N$, we get that
\[
\bbE[(\mathbf{II})^2] \leq c \bbE[(\epsilon_0)^2]  (\lambda_N N)^2 J_N 
\leq C e^{ 3c (\log N)^{\eta}} \gb_N^{3\ga} N^{\frac 52} \,,
\]
where we used the bounds~\eqref{lambdaN}-\eqref{epsilonN}
together with the fact that $J_N\leq N^{1/2}$ in dimension $d=3$, 
see Lemma~\ref{lem:JN}.
 Hence, using also that $v_N=N^{1/2}$
we get that 
\begin{equation*}
\Big(\frac{v_N }{\beta_N N^{d/2\alpha}}\Big)^2\bbE[(\mathbf{II})^{2}]
\leq C'  e^{ c'  (\log N)^{\eta}} \gb_N^{3\ga-2} N^{ \frac{7}{2} -\frac{d}{\ga} } \leq  C'  e^{ c'  (\log N)^{\eta}} N^{ \frac{7}{2}-\frac{3d}{2}} \, ,
\end{equation*}
which goes to $0$ as $N\to\infty$.

In the case of dimension $d=2$, we proceed as above.
Writing $e^{\lambda_N |\cR_N|} -1 \leq \lambda_N |\cR_N| e^{\lambda_N |\cR_N|}$ and 
decomposing according to whether we have  $|\cR_N| \leq 2C' N/\log N$ or not,
we get similarly to~\eqref{xinRd=2}
\begin{align*}
\bE\left[\left(e^{\lambda_N|\cR_N|}-1\right)\mathbbm{1}_{\{x\in\cR_N\}}\right]^2
\leq C\Big( \frac{\lambda_N N}{\log N}\Big)^2 \bP(x\in \cR_N)^2 + \lambda_N N e^{\lambda_N N}  \bP(x\in \cR_N) \bP\Big( |\cR_N| \geq 2C' \frac{N}{\log N}\Big) \, .
\end{align*}
Summing over $x$ and using the bound~\eqref{probadev} together with $\lambda_NN =o(\log N )$ and $\bE[|\cR_N|] \leq  N$, we get from~\eqref{IIIsquare}
that 
\begin{align*}
\bbE[(\mathbf{II})^{2}]
& \leq C \bbE[\epsilon_0^2]  \frac{(\lambda_N N)^2 }{(\log N)^2}  J_N + C \bbE[\epsilon_0^2]  \lambda_N N^2 e^{- c (\log N)^4} \\
& \leq C' e^{3(\log \log N)^{\eta}}  \gb_N^{2\ga} N^2 \Big( (\log N)^{-4} + e^{- c' (\log N)^4}  \Big)\, .
\end{align*}
For the second inequality, we used the bounds~\eqref{lambdaN}-\eqref{epsilonN}  (and also that $\gb_N^{\ga}N \leq c$)
and that $J_N\leq C N/(\log N)^2$, see Lemma~\ref{lem:JN}.
 Since $v_N =\log N$, we get that 
\begin{equation}
\label{termII}
\Big(\frac{v_N }{\beta_N N^{d/2\alpha}} \Big)^2\bbE[(\mathbf{II})^{2}] \leq 
 C' e^{3(\log \log N)^{\eta}}  \gb_N^{2(\ga-1)} N^{2-\frac{d}{\ga}} \Big( (\log N)^{-2} + (\log N)^2e^{- c' (\log N)^4}  \Big) \, ,
\end{equation}
which goes to $0$ as $N\to\infty$, using that  
$\gb_N^{\ga-1} N^{1- d/2\ga} \leq c$

\smallskip
{\it Term {\bf III}.}
We show the following lemma to control the third term in~\eqref{threetermsd=3}.
\begin{lemma}
\label{lem:termIII}
Let $d=2,3$ and $\ga \in (\frac d2,2)$. We have the following convergence in distribution:
\begin{equation}
\label{conv:termIII}
\frac{v_N }{\beta_N N^{d/2\alpha}}  \sum_{x\in \bbZ^d} (e^{\gb_N \tilde \go_x} -1) \bP(x\in \cR_N)
 \xrightarrow{(d)} \cW_{\gb}\, ,
\end{equation}
with $\cW_{\gb}$ defined in~\eqref{eq:defWbeta}.
\end{lemma}

Going back to~\eqref{polychaos}, and combining~\eqref{YNd=3}-\eqref{YNto0} with~\eqref{threetermsd=3} and the estimates~\eqref{termI}-\eqref{termII}, then Lemma~\ref{lem:termIII} 
concludes the proof of Lemma~\ref{lem:reduction11},
noting also that $e^{-\lambda_N}$ goes to $1$ as $N\to\infty$.
\end{proof}

\begin{proof}[Proof of Lemma~\ref{lem:termIII}]
To prove~\eqref{conv:termIII},
we adapt the method developed in \cite[Thm.~1.4]{DZ16}.
We consider only the case $\lim_{N\to\infty} \beta_N N^{d/2\alpha} =\gb \in (0,+\infty)$; the case $\gb=0$ can be dealt with similarly, one simply needs to keep track of the dependence in $\gb$ in all estimates --- note also that we have $\cW_{\gb} \to \cW_0$ as $\gb\to 0$.
Also, for $x\in \bbZ^d$, we use the notation $f_N(x) = v_N \bP(x\in \cR_N)$.

We prove the convergence of $ \sum_{x\in\bbZ^d} (\gb_N N^{d/2\ga})^{-1} (e^{\beta_N\tilde{\omega}_{x}}-1 ) f_N(x)$. 
As in Equation~(39) in~\cite{DZ16}, we fix $\gep>0$, $K>0$,
and we define $\tilde{\omega}_{x}^{\epsilon}=(\omega_x-\mu)\mathbbm{1}_{\{(\omega_x-\mu)\leq\epsilon N^{d/2\ga}\}}$. Recall the definition of $\tilde{\omega}_{x}$ in \eqref{def:kN2}.
We use the following decomposition of 
$\sum_{x\in\bbZ^d}  (\gb_N N^{d/2\ga})^{-1}  (e^{\beta_N\tilde{\omega}_{x}}-1 ) f_N(x)$:
\begin{equation}\label{decomp:main}
\begin{split}
 &  \sumtwo{\|x\|\leq K\sqrt{N} }{ (\omega_x-\mu)>\epsilon N^{d/2\ga} }    \frac{1}{\gb_N N^{d/2\ga}}\left(e^{\beta_N\tilde{\omega}_{x}}-1\right) f_N(x) +N^{-d/2\ga} \bbE[\tilde{\omega}_{0}^{\gep} ]\sum\limits_{\|x\| \leq K\sqrt{N}} f_N(x) \\
&  + \sum\limits_{\|x\| \leq K\sqrt{N}}  \frac{1}{\gb_N N^{d/2\ga}}\left( e^{\gb_N \tilde \go_x^{\gep}}  - \bbE\big[ e^{\gb_N \tilde \go_0^{\gep}} \big]\right) f_N(x) \\
&  +   
\frac{1}{\gb_N N^{d/2\ga}}\bbE\left[e^{\beta_N\tilde{\omega}_{0}^{\epsilon}}-1-\beta_N\tilde{\omega}_{0}^{\epsilon}\right]\sum\limits_{\|x\|\leq K\sqrt{N}} f_N(x) 
 + \sum\limits_{\|x\| >K \sqrt{N}}  \frac{1}{\gb_N N^{d/2\ga} }\left(e^{\beta_N\tilde{\omega}_{x}}-1\right) f_N(x) \, .
\end{split}
\end{equation}
We control each term separately: the first two terms bring the main contribution; we show that the last three terms can be made arbitrarily small by taking $\gep$ small or $K$ large.

\smallskip
\textit{$1$st term in~\eqref{decomp:main}.}
We split the term into
\begin{equation}\label{splitterm1}
\sumtwo{\gep^\ga\sqrt{N}\leq\|x\|\leq K\sqrt{N} }{ (\omega_x-\mu)>\epsilon N^{d/2\ga} }    \frac{1}{\gb_N N^{d/2\ga}}\left(e^{\beta_N\tilde{\omega}_{x}}-1\right) f_N(x)+\sumtwo{\|x\|\leq \gep^\ga\sqrt{N} }{ (\omega_x-\mu)>\epsilon N^{d/2\ga} }    \frac{1}{\gb_N N^{d/2\ga}}\left(e^{\beta_N\tilde{\omega}_{x}}-1\right) f_N(x).
\end{equation}

For the first term we partition $B_{K} \times(\epsilon,\infty)$ into rectangles with diameter $\delta>0$, where $B_K := \{x \in \bbR^d: \|x\|\leq K\}$. By denoting this partition $\sP_{\delta}$, $\pi$ a patch of $\sP_{\delta}$ and $(x_{\pi},w_{\pi})$ the center of $\pi$, we can write the first term above as 
\begin{equation}
\label{eq:1sttermPoisson}
\begin{split}
\sum\limits_{\pi \in \sP_{\gd}} \sum_{\gep^\ga\sqrt{N}\leq\|x\| \leq K\sqrt{N}} &
\frac{1}{\gb_N N^{d/2\ga}}(e^{\gb_N \tilde \go_x} -1) f_N(x) \ind_{\{ (\frac{x}{\sqrt{N}},\frac{\go_x}{N^{d/2\ga} } ) \in \pi\}} \\
& = 
 \sum\limits_{\pi \in \sP_{\gd}}  (1+o_\gd(1))   \frac{1}{\gb}(e^{\gb w_{\pi}} -1) f(x_{\pi}) 
 \sum_{\gep^\ga\sqrt{N}\leq\|x\| \leq K\sqrt{N}} \ind_{\{ (\frac{x}{\sqrt{N}},\frac{\go_x}{N^{d/2\ga} } ) \in \pi\}} \, ,
 \end{split}
\end{equation}
where we also used that $\lim_{N\to\infty} \gb_N N^{-d/2\ga} = \gb \in [0,\infty)$ (with by convention $\frac{1}{\gb}(e^{\gb w_{\pi}} -1) = w_{\pi}$ when $\gb=0$) and  $\lim_{N\to\infty} f_N( z \sqrt{N} ) = f(z)$, see~\eqref{def:f}; here $o_{\delta}(1)$ denotes errors that are negligible as $\delta \searrow 0$ for fixed $\epsilon$.

It is now not hard to check that (see \cite[p.~4036]{DZ16})
\begin{equation}
\label{eq:2ndtermPoisson}
\sum_{0<\|x\| \leq K\sqrt{N}} \ind_{\{ (\frac{x}{\sqrt{N}},\frac{\go_x}{N^{d/2\ga} } ) \in \pi\}}  
\longrightarrow
\cP(\pi)\quad\text{weakly},
\end{equation}
where $\cP(\pi)$ is a Poisson point measure on $\pi$ with intensity $\eta(\dd x,\dd w)=\alpha w^{-(\alpha+1)}\dd x\dd w$.

Thus, by first sending $N\to\infty$ and then letting $\delta\to0$, we have the following convergence in distribution:
\begin{equation}
\label{term1}
\sumtwo{\epsilon^{\alpha}\sqrt{N}\le \|x\|\leq K\sqrt{N} }{ (\omega_x-\mu)>\epsilon N^{d/2\ga} } \frac{1}{\gb_N N^{d/2\ga}} \left(e^{\beta_N\tilde{\omega}_{x}}-1\right) f_N(x)
\xrightarrow{(d)}
\int_{ (B_K\backslash B_{\gep^\ga}) \times(\epsilon,\infty)} \frac{1}{\gb}(e^{\beta w}-1)f(x)\cP(\dd x,\dd w) \, .
\end{equation}
Note that \cite[Theorem 10.15]{K97} ensures that the integral is well-defined (in particular at $w=+\infty$, recall also the proof of Proposition~\ref{prop:finiteW} in Section~\ref{sec:wellposedWbeta}).

%
For the second term in \eqref{splitterm1} we show that uniformly on $N\in \N$
\begin{equation}\label{term1_II}
\bbP\Bigg(\sumtwo{\|x\|\leq \gep^\ga\sqrt{N} }{ \bar \omega_x>\epsilon N^{d/2\ga} }    \frac{1}{\gb_N N^{d/2\ga}}\left(e^{\beta_N\tilde{\omega}_{x}}-1\right) f_N(x) >\gep^{\ga/2}\Bigg)\le c_\alpha \gep^{\ga/2}\log (1/\gep).
\end{equation}
Since $\tilde \omega_x\le \bar \omega_x$, we replace $\tilde \omega_x$ by $\bar \omega_x$. Moreover, we consider the event $\mathcal A=\Big\{ \max\limits_{\|x\|\leq \gep^\ga\sqrt{N}}\bar{\omega}_{x}\le N^{d/2\ga} \Big\}$ and split \eqref{term1_II} according to $\mathcal A$: 
the probability in the left-hand side of~\eqref{term1_II}
is bounded by
\begin{align}
\label{term1-II-I}
\bbP\Bigg(\sum_{\|x\|\leq \gep^\ga\sqrt{N} }   \frac{1}{\gb_N N^{d/2\ga}}\left(e^{\beta_N\bar{\omega}_{x}}-1\right) \ind_{\{{ \bar \omega_x> \epsilon N^{d/2\ga} } \}}f_N(x) > \gep^{\ga/2}\, ;\,  \mathcal A \Bigg) +\mathbb P\big(\mathcal A^c\big).
\end{align}
Since $\beta_N N^{d/2\alpha}\le c$, we have that on $\mathcal A$, $(e^{\beta_N\bar{\omega}_{x}}-1) \le  C\, \beta_N N^{d/2\alpha}$. Then, note that we have 
\begin{equation}\label{localerror}
|f_N( x) - f(x/\sqrt{N})|  = \Big| v_N \bP(x\in \cR_N) - f(x/\sqrt{N}) \Big| \leq c f(x/\sqrt{N})
\end{equation}
for all  $x\in \bbZ^d \setminus \{0\}$ with $\|x\| \leq K \sqrt{N}$ by \cite[Thm.~1.8]{Uch11} (in dimension $d\geq 3$) 
and \cite[Thm.~1.6]{Uch11}  (in dimension $d=2$) and by a Riemann sum approximation (and dominated convergence), we have
\begin{equation*}\begin{split}
& \bbE\Bigg[\sum_{\|x\|\leq \gep^\ga\sqrt{N} }\ind_{\{\bar \omega_x>\epsilon N^{d/2\alpha}\}} f_N(x) \Bigg]
\le C \gep^{-\alpha }N^{-d/2}\sum_{\|x\|\leq \gep^\ga\sqrt{N} }  f_N(x) \le C' \gep^{-\alpha } \int_{B_{\gep^{\alpha}}}f(x) \dd x,
\end{split}
\end{equation*}
where we used that $\bbP\big(\bar \omega_x>\epsilon N^{d/2\ga}\big)\le C\epsilon^{-\alpha}N^{-d/2}$. Then  from asymptotic estimates on $f$, see~\eqref{asympf0}, we easily obtain that this is bounded by a constant times $\gep^{2\alpha}  \log(1/\epsilon)$.
Therefore, by Markov's inequality, the first term of \eqref{term1-II-I} is bounded by $C''\epsilon^{-\alpha} \epsilon^{3\ga /2}  \log(1/\epsilon)\le \epsilon^{\ga /2}  \log(1/\epsilon)$. 

Finally, since $\mathbb P\big(\mathcal A^c\big)=\mathbb P\big( \max\limits_{\|x\|\leq \epsilon^\ga\sqrt{N}}\bar{\omega}_{x}> N^{d/2\ga} \big) \le c \gep^{\ga d}$, we complete the proof of \eqref{term1_II}.

\smallskip
\textit{$2$nd term in~\eqref{decomp:main}.}
Using that $\omega_0-\mu$ is centered, we easily get from~\eqref{eq:tailOmega} that
\begin{equation*}
N^{\frac{d}{2} (1-\frac{1}{\ga})} \bbE[\tilde{\omega}_{0}^{\epsilon}]=- N^{\frac{d}{2} (1-\frac{1}{\ga})} \bbE\big[(\omega_0-\mu)\mathbbm{1}_{\{(\omega_0-\mu)>\epsilon N^{d/2\alpha}\}}\big]\overset{N\to\infty}{\longrightarrow}-\alpha\int_{\epsilon}^{\infty} w  \cdot w^{-(1+\ga)}\dd w \, .
\end{equation*}

Again by \eqref{localerror} and a Riemann-sum approximation (and dominated convergence), we get that
\begin{equation}
\label{locallimitf}
\lim_{N\to\infty} N^{-\frac{d}{2}}  \sum_{\|x\| \leq K\sqrt{N}} f_N(x)
 = \int_{\|x\|\leq K} f(x) \dd x \, .
\end{equation}
All together, we get that the second term in~\eqref{decomp:main}
verifies
\begin{equation}
\label{term2}
\lim_{N\to\infty}  N^{-d/2\ga} \bbE[\tilde{\omega}_{0}^{\gep} ]\sum\limits_{\|x\| \leq K\sqrt{N}} f_N(x) = - \int_{ B_K \times (\gep,+\infty)} w f(x) \eta (\dd x, \dd w) \,.
\end{equation}
Combining~\eqref{term1} with~\eqref{term2}, we obtain that 
the sum of the first term in \eqref{splitterm1} and the second term in~\eqref{decomp:main} converges in distribution to
\begin{equation}
\label{terms12}
\begin{split}
\int_{(B_K\backslash B_{\epsilon^\alpha}) \times (\gep,+\infty)} \frac{1}{\gb} & (e^{\gb w} -1-\gb w) f(x)
 \cP(\dd x, \dd w) + \int_{(B_K\backslash B_{\epsilon^\ga}) \times (\gep,+\infty)} w f(x)
 (\cP-\eta)(\dd x, \dd w)\\
&-\int_{B_{\gep^\ga} \times(\epsilon,\infty)}\omega f(x)\eta(\dd x,\dd\omega) \,.
\end{split}
\end{equation}
Note that the third integral above is bounded by $c_\ga\epsilon^{1+\ga}\log(1/\epsilon)$. Letting $\gep \downarrow 0$ and $K\uparrow+\infty$ and by \eqref{term1_II}, one recovers
$\cW_{\gb}$, see~Proposition~\ref{prop:finiteW};
when $\gb=0$, the first term is equal to zero.
It therefore remains to control the remaining terms in~\eqref{decomp:main}.

\smallskip
{\it $3$rd term in~\eqref{decomp:main}.}
We control the second moment of this term, since it has zero mean.
By a Taylor expansion, using that $|\beta_N\tilde{\omega}_x^\epsilon|\leq C\epsilon\beta$, we have that
\[
\bbE\Big[ \big(e^{\beta_N\tilde{\omega}_{0}^{\epsilon}}-\bbE\big[e^{\beta_N\tilde{\omega}_{0}^{\epsilon}}\big] \big)^2 \Big]
= \bbE\big[ e^{2 \beta_N\tilde{\omega}_{0}^{\epsilon}} \big] -
\bbE\big[e^{\beta_N\tilde{\omega}_{0}^{\epsilon}}\big]^2
\leq C\beta_N^2\bbE\left[(\tilde{\omega}_{0}^\epsilon)^2\right]
 \leq C' \gep^{2-\ga} \gb_N^2 N^{\frac{d}{2\ga} (2-\ga)} \, ,
\]
where the last inequality holds since $\ga<2$.
Hence,  recalling that $f_N(x) = v_N \bP(x\in \cR_N)$, 
we get that the third term of~\eqref{decomp:main} has a second moment 
\begin{equation}
\label{term3}
\bbE\bigg[ \Big( \sum\limits_{\|x\| \leq K\sqrt{N}}  \frac{1}{\gb_N N^{d/2\ga}} \left( e^{\gb_N \tilde \go_x^{\gep}}  - \bbE\big[ e^{\gb_N \tilde \go_x^{\gep}} \big]\right) f_N(x) \Big)^2 \bigg] \leq c  \gep^{2-\ga} N^{-\frac{d}{2}}   v_N^2  J_N \leq c' \gep^{2-\ga} \, ,
\end{equation}
where for the last inequality we used the definition of $v_N$ (see Section~\ref{results:C}) together with Lemma~\ref{lem:JN} for $d=2,3$.

\smallskip
{\it $4$th term in~\eqref{decomp:main}}.
Again, by a Taylor expansion, and using $\bbE\left[(\tilde{\omega}_{0}^\epsilon)^2\right] \leq c \gep^{2-\ga} N^{\frac{d}{2\ga}(2-\ga)}$, we have
\begin{equation*}
\frac{1}{\gb_N N^{d/2\ga}} \bbE\big[e^{\beta_N\tilde{\omega}_{0}^{\epsilon}}-1-\beta_N\tilde{\omega}_{0}^{\epsilon}\big]\leq C \frac{\gb_N}{N^{d/2\ga}} \bbE[(\tilde{\omega}_{0}^{\epsilon})^2] \leq 
C \epsilon^{2-\alpha} \gb_N N^{\frac{d}{2\ga}} \,   N^{-\frac{d}{2}} \, .
\end{equation*}
Now, using~\eqref{locallimitf} and the fact that $\gb_N N^{d/2\ga} \leq c$, we get that the $4$th term in~\eqref{decomp:main} verifies
\begin{equation}
\label{term4}
0\leq \frac{1}{\gb_N N^{d/2\ga}}   \bbE\big[e^{\beta_N\tilde{\omega}_{0}^{\epsilon}}-1-\beta_N\tilde{\omega}_{0}^{\epsilon}\big] \sum\limits_{\|x\|\leq K\sqrt{N}} f_N(x) \leq C  \epsilon^{2-\alpha} \,.
\end{equation}

\smallskip
{\it $5$th term in~\eqref{decomp:main}.}
Here, we follow the line of proof of~\cite[Lem.~5.1]{DZ16}:
 we prove that (recall  the notation $\bar \go_x = \go_x-\mu$)
\begin{equation}
\label{term5}
\bbP\Big( \sum\limits_{\|x\| >K \sqrt{N}} \frac{1}{\gb_N N^{d/2\ga}} \left(e^{\beta_N \bar \omega_{x}}-1\right) f_N(x)  > e^{-K}  \Big) \leq   C  K^{d-2\ga} \,,
\end{equation}
which can be made arbitrarily small by taking $K$ large, since $\ga>d/2$.
Note that since $\tilde \go_x \leq \bar\go_x$, this enables us to control the last term in~\eqref{decomp:main}.

To prove~\eqref{term5}, we decompose the sum according to whether $\bar \go_x$
is smaller or larger than $1/\gb_N$.

To control the part of the sum with $\bar \go_x \leq 1/\gb_N$,
we use Lemma \ref{lem:moments} to get that 
$$
\bbE\Big[ \big| e^{\beta_N\bar{\omega}_{x} \ind_{\{\bar \go_x \leq 1/\gb_N\}}}-1 \big| \Big] \leq c'  \gb_N^{\ga} .
$$
Hence,  the sum over the terms with $\bar \go_x \leq 1/\gb_N$ has $L^1$ norm bounded by 
\begin{align*}
 \bbE\left[  \sum\limits_{\|x\| >K \sqrt{N}} \frac{1}{\gb_N N^{d/2\ga}}\big| e^{\beta_N\bar \omega_{x} \ind_{\{ \bar \go_x \leq 1/\gb_N \}}}-1 \big| f_N(x)  \right]
& \leq (\beta_N N^{d/2\alpha})^{\alpha-1} N^{-\frac d2}  \sum\limits_{\|x\| >K \sqrt{N}} f_N(x) 
\\ & \leq  c'   \int_{\|x\|>K} f(x) \dd x \, ,
\end{align*}
where we used $\gb_N \leq c N^{-d/2\ga}$ and the same Riemann-sum approximation
as in~\eqref{locallimitf}. Then  from asymptotic estimates on $f$, see~\eqref{asympfinfty}, we easily obtain that this is bounded by a constant times $e^{- c K^{2}}$.
Hence, applying Markov's inequality, 
we get that the probability that this part of the sum is larger than $e^{-K}/2$ is bounded by $c e^{K} e^{-c K^2} \leq c K^{d-2\ga}$.

Let us now control the sum over the terms  with $\bar \go_x >1/\gb_N$.
We write
\begin{multline}
\label{eq:probasum}
\bbP \bigg(   \sum\limits_{\|x\| >K \sqrt{N}} \frac{1}{\gb_N N^{d/2\ga}} \big( e^{\beta_N\bar{\omega}_{x} }-1 \big)  \ind_{\{ \bar \go_x  > 1/\gb_N \}} f_N(x)  >  e^{-K}/2 \bigg) \\
\leq \sum_{j=1}^{K^{-1} \sqrt{N}}  \bbP\bigg(  \sum\limits_{\|x\| \in (j , j+1 ] K\sqrt{N}} e^{\beta_N\bar{\omega}_{x} } \ind_{\{ \bar \go_x  > 1/\gb_N \}}   f_N(x)  > \gb_N N^{d/2\ga} e^{-K}   2^{-j-2 }\bigg) \,.
\end{multline}
Now,  we set $N_{j}^K = \big|\{x \in \bbZ^d \,,  \bar \go_x > 1/\gb_N\,,   \|x\| \in(j, j+1] K \sqrt{N} \}\big|$,
and we bound each probability in the sum by
\begin{multline}
\label{twoproba}
\bbP\Big( N_{j}^K > e^{K}  j^{1+d}  \gb_N N^{d/2\ga} \Big)\\
 + \bbP\bigg(  \sum\limits_{\|x\| \in (j , j+1] K \sqrt{N}}  e^{\beta_N\bar{\omega}_{x} }  \ind_{\{ \bar \go_x  > 1/\gb_N \}}   f_N(x)  > \gb_N N^{d/2\ga}  e^{-K} 2^{-j-2 } \, ;\, N_{j}^K \leq e^{K} j^{1+d}  \gb_N N^{d/2\ga} \bigg)\,.
\end{multline}
First of all, by Markov's inequality, since
the number of sites verifying $ \|x\| \in (j, j+1] K \sqrt{N}$ is bounded by  a constant times 
$ K^d j^{d-1} N^{d/2}$ and since $\bbP(\bar \go_x > 1/\gb_N) \leq c \gb_N^{\ga}$, we get that
\begin{equation}\label{boundN}
\bbP\Big( N_{j}^K > e^K j^{d+1}  \gb_N N^{d/2\ga} \Big) \leq c e^{-K} j^{-(d+1)}    (\gb_N N^{d/2\ga})^{-1}
\times K^{d} j^{d-1} N^{\frac{d}{2}}  \gb_N^{\ga} \leq c'  K^d e^{-K} j^{-2}\, ,
\end{equation}
where we used that $\gb_N^{\ga-1} \leq c N^{d (\ga-1)/2\ga}$.
For the second probability in~\eqref{twoproba}, we use that $f_N(x)$ is bounded above by $e^{- c j^{2} K^2} $ uniformly for $\|x\| >  j K \sqrt{N}$ (recall~\eqref{asympfinfty}), we bound the number of terms in the sum by $e^K j^{1+d} \gb_N N^{d/2\ga}$ (using the condition on $N_j^K$) and we bound $e^{\gb_N \bar\go_x}$ by its maximum.
We then obtain that the second probability in~\eqref{twoproba} is bounded by
\[
 \bbP\Big(  e^{K} j^{1+d} e^{- c j^{2} K^2}  \max_{\|x\| \in (j,j+1] K\sqrt{N}} \big\{  e^{\beta_N\bar{\omega}_{x} }  \big\}  >  e^{-K}  2^{-j-2 }  \Big)   \, .
 \] 
Now, using that  $e^{-2K} j^{-(1+d)}2^{-j} e^{ c j^{2} K^2} \geq  e^{c' j^2 K^2}$  uniformly in $j\geq 1$ and noting that  $N^{d/2} \gb_N^{\ga}\leq c$,
we get by a union bound that the above probability is bounded by
\begin{equation}
\label{boundmax}
 c j^{d-1} K^d N^{\frac{d}{2}} 
 \bbP\big(    e^{ c \gb_N \bar \go_x}    >    e^{c' j^2 K^2} \big)  
 \leq   c j^{d-1 -2\ga}  K^{d-2\ga} N^{\frac{d}{2}} \gb_N^{\ga}  \leq  c' j^{d-1 -2\ga}  K^{d-2\ga} \, .
\end{equation}

All together, summing over $j$ for \eqref{boundN}  and~\eqref{boundmax} (recall that $\ga>d/2$), we get that the probability in~\eqref{eq:probasum} is bounded by
a constant times $K^{d} e^{-K}+ K^{d-2\ga}$. This concludes the proof of~\eqref{term5}.

\smallskip
{\it Conclusion of the proof.}
Recall that the sum of the first two terms in~\eqref{decomp:main}
converge to~\eqref{terms12}.
Then, we let $\gep\downarrow 0$ and we use \eqref{term3}-\eqref{term4}
to get that the $3$rd and $4$th term in~\eqref{decomp:main}
become negligible.
Finally, letting $K\uparrow \infty$, \eqref{term5}
makes sure that the $5$th term in~\eqref{decomp:main}
goes to $0$ in probability.

\subsection{Transversal fluctuations}
\label{fluct4}
Let us first show that the transversal fluctuations are \textit{at least} of order $\sqrt{N}$.  Since $\bar{Z}_{N,\beta_N}(M_N\leq\eta\sqrt{N})=\bar{Z}_{N,\beta_N}\overline{\bP}_{N,\beta_N}(M_N\leq\eta\sqrt{N})$ and because $\bar{Z}_{N,\beta_N}\to1$ in $\bbP$-probability, we can focus on $\bar{Z}_{N,\beta_N}(M_N\leq\eta\sqrt{N})$.

By the same argument as \eqref{truncated}-\eqref{probatrunc}, $\bar{Z}_{N,\beta_N}(M_N\leq\eta\sqrt{N})-Z_{N,\beta_N}^{\tilde\omega}(M_N\leq\eta\sqrt{N})\to0$ in $\bbP$-probability. We then have
\begin{equation*}
\begin{split}
\bbP\big(\bar{Z}_{N,\beta_N}(M_N\leq\eta\sqrt{N})>\epsilon\big)\leq&\bbP\bigg(\big|\bar{Z}_{N,\beta_N}(M_N\leq\eta\sqrt{N})-Z_{N,\beta_N}^{\tilde\omega}(M_N\leq\eta\sqrt{N})\big|>\frac{\epsilon}{2}\bigg)\\
&+\bbP\bigg(Z_{N,\beta_N}^{\tilde\omega}(M_N\leq\eta\sqrt{N})>\frac{\epsilon}{2}\bigg),
\end{split}
\end{equation*}
where the second term is bounded above by $C_\epsilon(1-e^{-c\eta^2})$, thanks to the Markov's inequality and reasoning as in \eqref{reduction3}, which can be made arbitrarily small by choosing small enough $\eta$: this concludes the proof that transversal fluctuations are at least of order $\sqrt{N}$.

The proof for transversal fluctuations are \textit{at most} of order $\sqrt{N}$ is very similar, where one just needs to restrict the path to $\{M_N\leq A_N\sqrt{N\log N}\}$ first (see Lemma \ref{lem:truncate} for more details), so we omit it.

\end{proof}



\section{Region C: proof of Theorem~\ref{thm:C2}}
\label{sec:C2}

Here we have $\lim_{N\to\infty} \gb_N N^{\frac{d}{\ga}-1} = \gb \in [0,+\infty )$,
and  we work conditionally on the event $\gb\leq \gb_c$,
\textit{i.e.}\ $\hat \cT_{\gb}=0$ in Theorem~\ref{thm:A}.
Recall that if $\ga>1$ we set $\mu = \bbE[\go_0]$,
whereas when $\ga<1$ we let $\mu$ be any real number;
we again use the notation $\bar Z_{N,\gb_N} = Z_{N,\gb_N}^{\go,h=\mu}$.

\subsection{Preliminary: paths cannot stay at scale $N$ nor at an intermediate scale}

\label{sec:regCprel8}

First of all, let us stress that conditionally on having $\hat \cT_{\gb} =0$, 
paths cannot stay at scale $N$.
Indeed, similarly to what is done in~Section~\ref{sec:regionA} (see in particular~\eqref{convergenceeta}), we have that for any $\eta>0$
\[
\frac{1}{\gb_N N^{d/\ga}} \log Z_{N,\gb_N}^{\go,h} \big( M_N \geq  \eta N \big)  
\longrightarrow  \hat \cT_{\gb,\geq \eta} :=    \sup_{ s\in \cD, \sup_{t\in [0,1]}\|s(t)\|\geq \eta , \hatent(s) <+\infty} 
\big\{ \pi(s) -  \tfrac{1}{\gb} \hatent(s) \big\} \, .
\]
Now, on the event that $\hat \cT_{\gb}=0$
we have $\hat \cT_{\gb,\geq \eta}<0$ a.s., see Lemma \ref{lemma:techHatTbeta} in Appendix.
Hence, if $\lim_{N\to\infty} \gb_N N^{d/\ga -1} =\gb>0$, we get that for any $\eta>0$,
\begin{equation}
\label{notscaleN}
\lim_{N\to\infty} \bbP\Big(   Z_{N,\gb_N}^{\go,h} \big( M_N \geq  \eta N \big)  \geq e^{- N^{1/2}} \, \big| \,  \hat \cT_{\gb} =0\Big) =0\, .
\end{equation}
If  $\lim_{N\to\infty} \gb_N N^{d/\ga -1} =0$,  in view of Remark~\ref{rem:1/beta}, one has
\[
\lim_{N\to\infty} \frac1N \log Z_{N,\gb_N}^{\go, h}\big( M_N \geq  \eta N \big)
 = \lim_{N\to\infty} \frac1N \log \bP\big( M_N \geq  \eta N \big) <0 \, ,
\]
so the conclusion~\ref{notscaleN} remains valid.

Let us now prove here the analogous of Proposition~\ref{prop:Z>Ar}, stating that paths cannot stay at some intermediate scale between $N^{1/2}$ and $N$.

\begin{lemma}
\label{lem:intermscale}
Assume that $\lim_{N\to\infty} \gb_N N^{\frac{d-\ga}{\ga}} = \gb \in [0,+\infty )$. Then, there exists some $\nu'>0$ and a constant $c_1$ such that for any  $A_N \in [\sqrt{\log N}, N^{1/4}]$ and any $\eta \in(0,1)$
we have that for $N$ sufficiently large
\[
\bbP\Big( \log \bar Z_{N,\gb_N} \big( M_N \in [A_N \sqrt{N}, \eta N ) \big)  > - c_1 A_N^{2}   \Big)  \leq  c \eta^{\nu'}  + c A_N^{-\nu'}\, .
\]
\end{lemma}

\begin{proof}
The proof is very similar to that of Proposition~\ref{prop:Z>Ar}.
First,
we divide the partition function as
\begin{equation}
\label{split8}
\bar Z_{N,\gb_N} \big( M_N \in [A_N \sqrt{N}, \eta N ) \big)
= \bar Z_{N,\gb_N} \big( M_N \in [A_N \sqrt{N}, N^{3/4} ) \big)
+\bar Z_{N,\gb_N} \big( M_N \in [N^{3/4}, \eta N) \big) \, .
\end{equation}

For the first term, we divide the partition function as
\begin{equation}
\label{decomp8}
\begin{split}
\bar Z_{N,\gb_N} \big( M_N \in [A_N \sqrt{N}, N^{3/4} ) \big)
 & = \sum_{k= \log_2 A_N}^{ \frac14 \log_2 N} \bar Z_{N,\gb_N} \big( M_N \in [ 2^k \sqrt{N}, 2^{k+1} \sqrt{N} ) \big)  \\
& \leq \sum_{k= \log_2 A_N}^{ \frac14 \log_2 N}  e^{ - c 2^{2k} } \sqrt{\bar Z_{N,\gb_N} \big( M_N \leq  2^{k+1} \sqrt{N}  \big)} \, .
\end{split}
\end{equation}
where we used Cauchy--Schwarz inequality for the second inequality,
together with the fact that $\bP(M_N \geq 2^{k} \sqrt{N}) \leq \exp(- c 2^{2k})$.
Now, applying Lemma~\ref{lem:S<qr} with $h_N = 2^{k} \sqrt{N}$, we get for the range of $k$ considered that for any fixed constant $c_0>0$
\[
\bbP\Big( \bar Z_{N,\gb_N} \big( M_N \leq  2^{k+1} \sqrt{N}  \big)  \geq \exp( c_0 2^{2k} \big) \Big) \leq    c'' N^{-\frac{1}{4} (\frac{d}{\ga} -2)\nu} + c' 2^{- 2\nu k}\, ,
\]
where we used that $\gep_N  = N \gb_N h_N^{d/\ga-2} \leq c_{\gb} N^{-\frac14 (d/\ga -2)}$ since $h_N \leq N^{3/4}$ and $\ga<\frac{d}{2}$.

Together with~\eqref{decomp8},
a union bound therefore shows that, provided that $c_1$ has been chosen small enough, we have
\begin{equation}
\label{8.1-term1}
\bbP\Big(  \bar Z_{N,\gb_N} \big( M_N \in [A_N \sqrt{N}, N^{3/4} ) \big)  \geq e^{- c_1 A_N^2}\Big) \le   c'' A_N^{- 2\nu}  + c (\log N) N^{-\frac14 (\frac d\ga -2)} \, .
\end{equation}
Note that the second term is bounded by $A_N^{- \nu'}$ since $A_N \leq N^{1/4}$.

For the second term in~\eqref{split8}, we proceed analogously: we decompose it as
\begin{equation}
\label{decomp8-2}
\begin{split}
\bar Z_{N,\gb_N} \big( M_N \in [N^{3/4} , \eta N) \big)
 & = \sum_{k= -\log_2 \eta}^{ \frac14 \log_2 N} \bar Z_{N,\gb_N} \big( M_N \in [ 2^{-k-1} N, 2^{-k} N ) \big)  \\
& \leq \sum_{k= -\log_2 \eta}^{ \frac14 \log_2 N}  e^{ - c 2^{ -2k} N } \bar Z_{N,\gb_N} \big( M_N \leq  2^{- k} N \big) \, .
\end{split}
\end{equation}
Again, we can apply Lemma~\ref{lem:S<qr} with $h_N = 2^{-k} N$ to get for the range of~$k$ considered that
\begin{equation}
\bbP\Big( \bar Z_{N,\gb_N} \big( M_N \leq  2^{- k} N \big) \geq \exp\big(  c_0 2^{-2k } N  \big) \Big) \leq c 2^{-k(\frac{d}{\ga} -2)\nu} + N^{- \nu/2}\,,
\end{equation}
where we used $\gep_N = N\gb_N h_N^{d/\ga-2}  \leq c_{\gb} 2^{-k(d/\ga-2}$, together with the fact that $h_N^2/N \geq N^{1/2}$.

Again, together with~\eqref{decomp8-2},
a union bound therefore shows that, provided that $c_1$ has been chosen small enough, we have
\begin{equation}
\label{8.1-term2}
\bbP\Big(  \bar Z_{N,\gb_N} \big( M_N \in [N^{3/4} , \eta N) \big)  \geq e^{- c_1 N^{1/2}}\Big) \le   c'' \eta^{(\frac{d}{\ga}-2) \nu}  + c (\log N) N^{-\nu/2}\, .
\end{equation}
Note again that the second term is bounded by $A_N^{- \nu'}$ since $A_N \leq N^{1/4}$.
Combining~\eqref{8.1-term1} and \eqref{8.1-term2}  therefore concludes the proof of Lemma~\ref{lem:intermscale}.
\end{proof}

\subsection{The case $\ga\in (0,1)$}

We first prove~\eqref{convC2-2}, in the case $\ga\in (0,1)$.
We show the following: if $\ga\in (0,1)$, then for any $h\in \bbR$,
\begin{equation}
\label{convergence<1}
\frac{v_N}{\beta_N N^{d/2\alpha}}  \Big( Z_{N,\gb_N}^{\go,h} - 1 \Big)
\overset{(d)}{\longrightarrow}\cW_{0}  \, .
\end{equation}
Recall that $v_N = \log N$ in dimension $d=2$ and $v_N = N^{\frac{d}{2}-1}$ in dimension $d\geq 3$.

\begin{proof}[Proof of~\eqref{convergence<1}]
First of all, using~\eqref{notscaleN} and Lemma~\ref{lem:intermscale}, we get that 
for any $a>0$ there is some $A>0$ such that
\[
\lim_{N\to\infty} \bbP\Big( N^{a} \big( Z_{N,\gb_N}^{\go, h}  - Z_{N,\gb_N}^{\go, h} (M_N \leq A\sqrt{N\log N})  \big)  >\gep  \, \Big| \, \hat \cT_{\gb} =0 \Big) =0 \, .
\]
Therefore, we may focus on the convergence of $Z_{N,\gb_N}^{\go, h} (M_N \leq A\sqrt{N\log N} )$.
 We have
\[
e^{- |h|\gb_N N } \leq  \frac{ Z_{N,\gb_N}^{\go, h} \big(M_N \leq A\sqrt{N\log N}  \big)  }{Z_{N,\gb_N}^{\go, h=0} \big(M_N \leq A\sqrt{N\log N} \big) }\leq   e^{ |h| \gb_N N}  \, ,
\]
so that, for any $h\in \bbR$, noting that $\beta_N N\leq cN^{2-d/\ga}\to0$, we have
\begin{equation}
\label{eq:comparethetwo}
\bigg|  \frac{Z_{N,\gb_N}^{\go, h} \big( M_N \leq A\sqrt{N\log N}  \big)   }{Z_{N,\gb_N}^{\go, h=0} \big( M_N \leq A\sqrt{N\log N} \big)  }- 1  \bigg| \leq  C |h|\gb_N N  \, .
\end{equation}
Notice that $v_N /( \beta_N N^{d/2\alpha} ) \times  \gb_N N \leq N^{\frac{d}{2}(1-\frac{1}{\ga})}  \log N$  and goes to $0$ as $N\to\infty$
(the factor $\log N$ factor is present only in dimension $d=2$).
Hence, in view of \eqref{eq:comparethetwo}, we therefore only have to prove the result for $h=0$, that is
\begin{equation}
\label{eq:casealpha<1}
\frac{v_N}{\beta_N N^{d/2\alpha}}  \Big(Z_{N,\gb_N}^{\go, h=0} \big( M_N \leq A\sqrt{N\log N}  \big) -1 \Big)\overset{(d)}{\longrightarrow}\cW_{0}  \, ,
\end{equation}
since it also implies that $Z_{N,\gb_N}^{\go, h=0} \big( M_N \leq A\sqrt{N\log N}  \big)$ goes to $1$ in probability.

Now, since the $\go$'s are non-negative, on the event $M_N \leq A \sqrt{N\log N}$
we have
\begin{align*}
\bbP \Big(  \sum_{x\in \cR_N}  \go_x  > N^{\frac{d}{2\ga}} (\log N)^{\frac{d}{\ga}} \Big)
& \leq \bbP \Big(  \sum_{\|x\| \leq A\sqrt{N\log N}}  \go_x  >  N^{\frac{d}{2\ga}} (\log N)^{\frac{d}{\ga}} \Big) \\
& \leq  (A\sqrt{N\log N})^{d}  \bbP\big( \go_0 > N^{\frac{d}{2\ga}} (\log N)^{\frac{d}{\ga}} \big) 
\leq  C A^d   (\log N)^{-\frac{d}{2}} \,,
\end{align*}
where we used a union bound for the second inequality.
Hence, with $\bbP$ probability going to $1$,
we have that $\gb_N\sum_{x\in \cR_N} \go_x  \leq \gb_N N^{\frac{d}{2\ga}} (\log N)^{\frac{d}{\ga}} \to 0$
on the event $\cA_N:=\{M_N \leq A\sqrt{N\log N}\}$.  By a Taylor expansion, 
we therefore have 
\begin{equation*}
\Big|  Z_{N,\gb_N}^{\go, h=0} \big( M_N \leq A\sqrt{N\log N})  \big) -  \bP(\cA_N) - \gb_N \bE\Big[ \sum_{x\in \cR_N}  \go_x \ind_{\{M_N \leq A\sqrt{N\log N}\}}\Big]\Big|
\leq  \big( \gb_N N^{\frac{d}{2\ga}} (\log N)^{\frac{d}{\ga}} \big)^2 \,,
\end{equation*}
with $\bbP$-probability going to $1$ as $N\to\infty$.
Notice that since $\gb_N \leq c N^{1- d/\ga}$ and $v_N \leq N^{d/2-1} \log N$, we get that
\[
\frac{v_N}{\beta_N N^{d/2\alpha}}  \times  \big( \gb_N N^{\frac{d}{2\ga}} (\log N)^{\frac{d}{\ga}} \big)^2
\leq N^{ \frac{d}{2} (1-\frac{1}{\ga})}  (\log N)^{\frac{2d}{\ga}}\,,
\]
which goes to $0$ as $N\to\infty$ since $\ga <1$.

 Note that $1-\bP(\cA_N)\leq N^{-cA^2}$ and $A$ is large enough. All together, it boils down to proving that 
\begin{equation}
\label{eq:lastconvergence}
\frac{v_N}{N^{d/2\ga}} \bE\Big[ \sum_{x\in \cR_N}  \go_x \ind_{\{M_N \leq A\sqrt{N\log N}\}}\Big] 
 \xrightarrow{(d)} \cW_0 \, .
\end{equation}
Recalling the notation $f_N(x) = v_N \bP(x\in \cR_N)$,
we have the following bound
\begin{multline*}
\bigg| \frac{v_N}{N^{d/2\ga}} \sum_{x\in \cR_N}  \go_x \bP(x\in \cR_N ,  M_N \leq A\sqrt{N\log N} \Big]  -  \sum_{\|x\| \leq A\sqrt{N\log N}} \frac{\go_x}{N^{d/2\ga}} f_N (x) \bigg|\\
 \leq 
 \frac{v_N}{N^{d/2\ga}} \max_{\|x\| \leq A\sqrt{N\log N}} \{\go_x\} \bE\big[ |\cR_N| \ind_{ \{M_N > A\sqrt{N\log N} \} } \big] \\
 \leq N^{\frac{d}{2}} \max_{\|x\| \leq A\sqrt{N\log N}} \Big\{ \frac{\go_x}{N^{d/2\ga}} \Big\}  \left(\bP( M_N > A\sqrt{N\log N})+\log N\bP(|\cR_N|>AN/\log N)\mathbbm{1}_{d=2}\right)\, ,
\end{multline*}
which goes to $0$ in probability provided that $A$ has been fixed large enough.
Therefore, the convergence~\eqref{eq:lastconvergence} is equivalent to showing
\begin{equation}
\label{lastconvergence<1}
\sum_{\|x\| \leq A\sqrt{N\log N}} \frac{\go_x}{N^{d/2\ga}} f_N (x) 
 \xrightarrow{(d)} \cW_0 = \int_{\bbR^d \times \bbR^+} w f(x) \cP(\dd x , \dd w) \, ,
\end{equation}
with $\cP$ a Poisson point process on $\bbR^d \times \bbR^+$ with intensity $\ga w^{-\ga-1} \dd x \dd w$. 
The proof is very similar to the proof of Lemma~\ref{lem:termIII} --- see in particular the method used in~\eqref{eq:1sttermPoisson}-\eqref{eq:2ndtermPoisson}.

Since $w^{-\ga}$ is integrable around $0$ when $\ga\in (0,1)$, by the same argument as in \eqref{eq:1sttermPoisson}--\eqref{eq:2ndtermPoisson}, we have analogously to~\eqref{term1}
\begin{equation}
\label{8.13}
\sum\limits_{\|x\|\leq K\sqrt{N}} \frac{\go_x}{N^{d/2\ga}} f_N(x)\overset{(d)}{\longrightarrow}\int_{B_K \times \bbR_+} w f(x)\cP(\dd x,\dd w).
\end{equation}
For the remaining term in~\eqref{lastconvergence<1}, we get that it verifies
\begin{equation}
\label{8.14}
\bbP\bigg(   \sum\limits_{ K\sqrt{N} < \|x\| \leq A\sqrt{N\log N}}  \frac{\go_x}{N^{d/2\ga}}  f_N(x) > K^{-1} \bigg)
 \leq\sum\limits_{k=1}^{\infty}\bbP\bigg(\sum\limits_{\|x\|\in(2^{k-1},2^k]K\sqrt{N}}  \!\! \omega_x f_N(x)>2^{-k}K^{-1}N^{\frac{d}{2\ga}}\bigg) \, .
\end{equation}
Since $f_N(x) \leq C e^{ - c \|x\|^2 /N}$, we get that this probability
is bounded by a constant times
\begin{equation}
\label{8.15}
\sum_{k=1}^{\infty} 2^{kd}K^d N^{\frac{d}{2}} \bbP\left(\omega_0>c'2^{-k}K^{-1}e^{c2^{2(k-1)}K^2}N^{\frac{d}{2\ga}}\right)\leq c \sum_{k=1}^{\infty} (2^{k}K)^{d+\alpha}e^{-c_\alpha 2^k K^2}  \leq c e^{- c' K^2}\, .
\end{equation}
Combining this with~\eqref{8.13} we obtain~\eqref{lastconvergence<1} by letting $N\to\infty$ and then $K\to\infty$.
\end{proof}

\subsection{The case $\ga \in (1 ,\frac{d}{2})$}
We now prove~\eqref{convC2-1} and~\eqref{convC2-2} in the case $\ga \in (1,\frac d2)$; in particular $d\geq 3$.
The proof is similar to what is done in Section~\ref{sec:regC}: first we truncate the environment
and then we prove the convergence of the partition function with truncated environment.

\smallskip
{\bf Step 1. Truncation of the environment. }
Let us define the truncation level 
\begin{equation}\label{k_N}
k_N := (\log N)^{\frac{d}{\ga}} 
N^{\frac{d}{2\alpha}}\, ,
\end{equation}
and the truncated environment $\tilde \go  = (\go-\mu) \ind_{\{\go \leq k_N\}}$.
Let us stress that $\gb_N k_N \leq c (\log N)^{d/\ga} N^{1-d/2\ga} $ goes to $0$ as $N\to\infty$, since $\ga<d/2$.
We again define $\lambda_N := \log \bbE[e^{\gb_N \tilde \go}]$ and we notice that thanks to Lemma~\ref{lem:moments} (in the case $\gb_N k_N <1$)
we have 
\begin{equation}
\label{lambdaN2}
\lambda_N \leq C \gb_N k_N^{1-\ga} \leq  C  N^{1-\frac{d}{2} - \frac{d}{2\ga} } (\log N)^{C}\, .
\end{equation}
In particular, $\lambda_N N$ goes to $0$ as $N\to\infty$ since $\ga<d/2$.

\begin{lemma}
\label{lem:restrict<d/2}
Let $\ga \in (1,\frac d2)$.
Assume that $\lim_{N\to\infty} \gb_N N^{\frac{d-\ga}{\ga}} = \gb \in [0,+\infty )$.
Then for any $a>0$, conditionally on $\hat \cT_{\gb} =0$,  
$N^{a} \big( \bar Z_{N,\gb_N} - Z_{N,\gb_N}^{\tilde \go}  \big)$ goes to $0$ in probability.
\end{lemma}

\begin{proof}
Using~\eqref{notscaleN} and Lemma~\ref{lem:intermscale}, we get that 
for any $a>0$ there is some $A>0$ such that
\[
\lim_{N\to\infty} \bbP\Big( N^{a} \big( \bar Z_{N,\gb_N} - \bar Z_{N,\gb_N} (M_N \leq A\sqrt{N\log N})  \big)  >\gep  \, \Big| \, \hat \cT_{\gb} =0 \Big) =0 \, .
\]
Now, using the definition~\eqref{k_N}, we have that 
\begin{equation*}
\bbP\Big( \exists \, x , \|x\| \leq A\sqrt{N\log N} \,, 
\tilde \go_x \neq \go_x-\mu \Big)
\leq A^d (N \log N)^{d/2} \bbP(\go >k_N) \leq c'(\log N)^{- \frac{d}{2}}
\end{equation*}
Therefore, we get that 
\[
\lim_{N\to\infty} \bbP\Big( \bar Z_{N,\gb_N} (M_N \leq A\sqrt{N\log N})  \big)  \neq Z_{N,\gb_N}^{\tilde \go} (M_N \leq A\sqrt{N\log N})  \big)  \Big) =0 \, .
\]
It therefore only remains to observe that, by Markov's inequality,
\begin{align*}
\bbP\Big( N^a Z_{N,\gb_N}^{\tilde \go} (M_N > A\sqrt{N\log N})  >\gep  \Big) 
& \leq \gep^{-1} N^a \bbE\big[ Z_{N,\gb_N}^{\tilde \go} (M_N > A\sqrt{N\log N}) \big] \\
&\leq \gep^{-1} e^{\lambda_N N}  N^a \bP\big(M_N > A\sqrt{N\log N} \big)\,  .
\end{align*}
Hence, we get that 
\[
\bbP\Big( N^a Z_{N,\gb_N}^{\tilde \go} (M_N > A\sqrt{N\log N})  >\gep  \Big)  \leq C \gep^{-1} N^a \bP(M_N > A\sqrt{N\log N}),
\]
which goes to zero provided $A$ is fixed large enough. This concludes the proof.
\end{proof}

\medskip
\noindent
{\bf Step 2. Convergence of the truncated partition function.}
Thanks to Lemma~\ref{lem:restrict<d/2}, it therefore only remains to prove the convergence of the truncated partition function.

Let us stress that when $d\geq 5$ and $\ga \in (\frac{d}{d-2} ,\frac{d}{2})$, then we may apply Lemma~\ref{lem:reduction0}: the only requirement  is to have $\gb_N \leq c N^{-d/2\ga}$  and  $\gb_N k_N \leq c(\log N)^{\eta}$ for some $\eta<1$ (to apply the bounds~\eqref{lambdaN}-\eqref{epsilonN}), which is clearly the case here.
This proves the convergence~\eqref{convC2-1} in Theorem~\ref{thm:C2}.

\smallskip
We therefore focus on the case $\ga\in (1,\frac{d}{2}\wedge \frac{d}{d-2})$: 
to conclude the proof of the convergence~\eqref{convC2-2} in Theorem~\ref{thm:C2}, we need to prove the following:
\begin{equation}
\label{conv:<d/2}
\frac{v_N}{\beta_N N^{d/2\alpha}} \big(Z_{N,\gb_N}^{\tilde \go} -1 \big)\overset{(d)}{\longrightarrow}\cW_{0} \,,
\end{equation}
with $\cW_0$ defined in~\eqref{def:W0}.
This is analogous to Lemma~\ref{lem:reduction11}, and let us stress that this lemma uses the condition $\ga >d/2$ only to control~\eqref{term5}.
Let us rapidly go over the different steps of the proof, the main changes occuring in the proof of Lemma~\ref{lastconv} below, which is the analogous of Lemma~\ref{lem:termIII}.

Recall that we define $\lambda_N:=\log\bbE[e^{\beta_N\tilde{\omega}}]$ and $\epsilon_x:=e^{\beta_N\tilde{\omega}_x-\lambda_N}-1$;
let us stress that using that $\gb_N k_N$ goes to $0$, a simple Taylor expansion gives that  
$\bbE[\gep_0^2] \leq C \gb_N^2  \bbE[\tilde \go_x^2 ] +C\lambda_N^2 \leq C' \gb_N^2 k_N^{2-\ga}$. 

We perform the same chaos expansion as \eqref{polychaos} to get
\begin{equation}\label{chaos}
Z_{N,\beta_N}^{\tilde{\omega}}-1=\bE\big[e^{\lambda_N|\cR_N|}-1\big]+\sum\limits_{x\in\bbZ^d}\epsilon_x\bE\big[e^{\lambda_N|\cR_N|}\mathbbm{1}_{\{x\in\cR_N\}}\big]+\bY_N,
\end{equation}
where $\bY_N$ is defined as \eqref{def:RN}.

Note that $\bY_N$ is centered, and following the same computation as in \eqref{YNd=3}-\eqref{secondERN2}, we have that
$\bbE[\bY_N^2]\leq C(\bbE[(\epsilon_0)^2]J_N)^2\leq C\beta_N^4 k_N^{4-2\alpha}  J_N^2$
Hence, in view of the definition of $k_N = (\log N)^{d/\ga} N^{d/2\ga} $, of the definition of $v_N = N^{d/2-1}$ and $J_N \leq c N^{1/2}$ (recall $d\geq 3$), we get that
\begin{equation}\label{Y_N^2}
\Big(\frac{v_N}{\beta_N N^{d/2\alpha}}\Big)^2\bbE[\bY_N^2]\leq C     \gb_N^{2}  N^{\frac d\ga-\frac32}  (\log N)^{C}  \leq C' N^{\frac12-\frac{d}{\ga}}  (\log N)^{C}  \, ,
\end{equation}
which goes to $0$ as $N\to\infty$ since $\ga<d$. Therefore, $ v_N \bY_N / \gb_N N^{d/2\ga}$ goes to $0$ in probability.

We now handle the first two terms on the right-hand side of \eqref{chaos}: we use the same decomposition as in~\eqref{threetermsd=3} and the same notation $\mathbf{I}$, $\mathbf{II}$, $\mathbf{III}$ for the three terms. We now control the different terms.

For the term $\mathbf{I}$, using \eqref{lambdaN2} (in particular $\lambda_N N\to 0$) we have that 
$\mathbf{I} \leq c (\lambda_NN)^2 \leq N^{2} \gb_N^2 k_N^{2-2ga}$.
Using the definition of $v_N$ and  $k_N$ and since $\gb_N \leq c N^{1-d/\ga}$, we therefore get that 
\begin{equation}
\label{I}
\frac{v_N}{\beta_N N^{d/2\alpha}} \mathbf{I} 
 \leq C  N^{2- \frac{d}{2\ga}  -\frac{d}{2}} (\log N)^{C}  \leq 
 C' N^{\frac12- \frac{d}{\ga}} (\log N)^{C} \,,
\end{equation}
which goes to $0$ since $\ga<d$.
 
For the term $\mathbf{II}$,
it is centered and we have
$\bbE[(\mathbf{II})^2] \leq C\bbE[(\epsilon_0)^2] (\lambda_N N)^2 J_N\leq C' \gb_N^{4} k_N^{4-3\ga} N^2 J_N$.
Using the definition of $v_N$ and $k_N$, we therefore get that
\begin{equation}
\label{II}
\Big( \frac{v_N}{\beta_N N^{d/2\alpha}} \Big)^2 \bbE[(\mathbf{II})^2] 
 \leq C \gb_N^2 N^{\frac{d}{\ga} -\frac{d}{2}} J_N  (\log N)^C
 \leq N^{2 - \frac{d}{2}-\frac{d}{\ga}}   (\log N)^{C} \leq N^{2-\frac{d}{2}-(2\vee d-2)+\epsilon},
\end{equation}
where we used that $J_N \le c N^{1/2}$ and $\ga<\frac{d}{2}\wedge\frac{d}{d-2}$.
This goes to $0$ as $N\to\infty$ since $d\geq3$.

%
%

For the term $\mathbf{III}$, it is easy to see that since $\gb_N \tilde \go_0 \leq \gb_N k_N \to 0$, we have by a Taylor expansion
\begin{equation*}
\Big| e^{\lambda_N} \mathbf{III} -  \gb_N \sum_{x\in \bbZ^d}  \tilde \go_x \bP(x\in \cR_N) \Big| 
\leq  c \gb_N^2 \sum_{x\in \bbZ^d} \tilde \go_x^2 \bP(x\in \cR_N) \,. 
\end{equation*}
Hence, using that $\bbE[\go_x^2] \leq k^{2-\ga} $ we obtain that 
\begin{equation}
\label{III}
\frac{v_N}{\beta_N N^{d/2\ga}}
 \bbE\Big[ \Big| e^{\lambda_N} \mathbf{III} -  \gb_N \sum_{x\in \bbZ^d}  \tilde \go_x \bP(x\in \cR_N) \Big| \Big]
 \leq c\frac{v_N }{N^{d/2\alpha}}\beta_N k_N^{2-\alpha} N \leq cN^{1-\frac{d}{2\ga}}(\log N)^{C} \,,
\end{equation}
which goes to $0$ as $N\to\infty$ since $\alpha<d/2$.

\smallskip
All together, the proof of~\eqref{conv:<d/2} boils down to showing the following lemma
\begin{lemma}
\label{lastconv}
If $\ga\in (1,\frac{d}{2})$ and $d<\frac{2\ga}{\ga-1}$ (or equivalently $\ga\in (1,\frac{d}{2}\wedge \frac{d}{d-2}$), we have the convergence
\begin{equation}
\frac{v_N}{N^{d/2\ga}}\sum\limits_{x\in\bbZ^d}\tilde{\omega}_{x}\bP(x\in\cR_N)\overset{(d)}{\longrightarrow}\cW_0 \,.
\end{equation}
\end{lemma}

\begin{proof}
The following procedure is similar to the proof of Lemma~\ref{lem:termIII}. Note that Lemma~\ref{lem:termIII} uses the condition $\ga>d/2$, in particular to ensure that the limit $\cW_{\gb}$ is finite. Here, the condition $\ga>d/(d-2)$
also appears naturally, in particular to ensure that the limit $\cW_{0}$ is finite, see Proposition~\ref{prop:finiteW0}.

We write $f_N(x) =v_N\bP(x\in\cR_N)$ and
we use the following decomposition: for a fixed constant $K$ (large) 
\begin{equation*}
\label{decomp1}
\frac{v_N}{N^{d/2\ga}}\sum\limits_{x\in\bbZ^d}\tilde{\omega}_x\bP(x\in\cR_N)=\frac{1}{N^{d/2\ga}}\sum\limits_{\|x\|\leq K\sqrt{N}}\tilde{\omega}_x f_N(x)+\frac{1}{N^{d/2\ga}}\sum\limits_{\|x\|> K\sqrt{N}}\tilde{\omega}_x f_N(x) \, .
\end{equation*}

For the second term, we proceed as in~\eqref{8.14}-\eqref{8.15} to get that it is bounded by $K^{-1}$
with probability larger than $1- c e^{-c'K^2}$, uniformly in $N$.

For the first term,
then defining $\tilde{\omega}_x^{\epsilon}:=(\omega_x-\mu)\mathbbm{1}_{\{\omega_x\leq\epsilon N^{d/2\alpha}\}}$, we decompose it further as
\begin{equation}\label{decomp2}
\frac{1}{N^{d/2\alpha}}  \!\! \sum\limits_{\|x\|\leq K\sqrt{N}\atop(\omega_x-\mu)>\epsilon N^{d/2\ga}} 
\!\!\! \tilde{\omega}_{x}f_N(x)
+  \frac{1}{N^{d/2\alpha}}  \sum_{\|x\| \leq K \sqrt{N}} \bbE[\tilde \go_x^{\gep}]  f_N(x) + \frac{1}{N^{d/2\ga}}\sum\limits_{\|x\|\leq K\sqrt{N}}(\tilde{\omega}_x^{\epsilon}-\bbE[\tilde{\omega}_x^{\epsilon}])f_N(x).
\end{equation}

 By the same argument as for \eqref{term1} and \eqref{term2} in Lemma~\ref{lem:termIII}, we get that the sum of the first two terms converge in distribution to
\begin{equation}\label{conv:(1,2)}
\int_{ B_K \times (\epsilon,\infty)} w f(x)(\cP-\eta)(\dd x,\dd w).
\end{equation}
Note that it is then valid to send $\epsilon\downarrow 0$ and $K\uparrow\infty$
in~\eqref{conv:(1,2)}, since $\alpha < d/(d-2)$ (recall Proposition~\ref{prop:finiteW0}).
It therefore remains to prove that the last term in~\eqref{decomp2} is negligible.

To conclude the proof, we will show that  for any $\gep>0$, there exists $M_{\gep}$ which verifies $\lim_{\gep\downarrow 0} M_{\gep} = +\infty$, such that 
\begin{equation}
\label{whattoprove}
\bbP\bigg(  \sum_{ \|x\| \leq K\sqrt{N}}   N^{-\frac{d}{2\ga}} \widebar \go_x^{\gep}  f_N(x)  \geq  M_{\gep}^{-\frac12(\ga-1)} \bigg) \leq  C M_{\gep}^{-\frac12 (\ga-1)} \, ,
\end{equation} 
where we denoted $\widebar \go_x^{\gep} = \tilde \go_x^{\gep} - \bbE[\tilde \go_x^{\gep}] $ for simplicity. 
Then, sending $N\to\infty$ and then letting $\gep \downarrow 0$ and $ K \uparrow \infty$, this concludes the proof of Lemma~\ref{lastconv}.

To prove~\eqref{whattoprove}, we split the sum according to whether $\bar\go_{\gep} f_n(x)$ is larger or smaller than $M_{\gep} N^{d/2\ga}$.
For the first sum, we use Markov's inequality to get that 
\begin{equation}
\label{whattoprove1}
\begin{split}
\bbP\Big( \sum_{ \|x\| \leq K\sqrt{N}}   N^{-\frac{d}{2\ga}} \widebar \go_x^{\gep} f_N(x)  & \ind_{\{ \widebar \go_x^{\gep} f_N(x)  >   M_{\gep} N^{d/2\ga}  \}} \geq  \frac12 M_{\gep}^{-\frac12 (\ga-1)} \Big) \\
&\leq  M_{\gep}^{\frac{1}{2} (\ga-1)} \sum_{ \|x\| \leq K\sqrt{N}} f_N(x) N^{-\frac{d}{2\ga}} \bbE\big[ \widebar \go_x^{\gep} \ind_{\{ \widebar \go_x^{\gep} >M_{\gep} N^{d/2\ga} f_N(x)^{-1} \}} \big] \, .
\end{split}
\end{equation}
Notice that $f_N(x) \leq v_N  = N^{d/2-1}$, so $N^{d/2\ga}  f_N(x)^{-1} \geq N^{1- \frac{d}{2\ga}(\ga-1)}$ and goes to $+\infty$ as $N\to\infty$ since we have $d < 2\ga/(\ga-1)$: 
hence we can use the following estimate
\begin{equation}
\label{meangepx}
\bbE\big[ \widebar \go_x^{\gep} \ind_{\{ \widebar \go_x^{\gep} >  M_{\gep} N^{d/2\ga} f_N(x)^{-1} \}} \big] \leq 
C  M_{\gep}^{1-\ga} N^{\frac{d}{2\ga} (1-\ga )} f_N(x)^{\ga-1} \, .
\end{equation}
This shows that \eqref{whattoprove1} is bounded by 
\begin{equation*}
 C  M_{\gep}^{-\frac{1}{2} (\ga-1)} N^{-\frac{d}{2}} \sum_{ \|x\| \leq K\sqrt{N}} f_N(x)^{\ga}  \leq C  M_{\gep}^{-\frac{1}{2} (\ga-1)}  \, ,
\end{equation*}
where we have used an integral comparison, using that  $\int_{\bbR^d} f(x)^\ga \dd x <\infty$ since $\ga (d-2)<d$,
recall~\eqref{asympf0}\eqref{asympfinfty}.

It remains to control the sum with terms $\bar\go^{\gep} f_N(x)$ smaller than $M_{\gep} N^{d/2\ga}$.
With an exponential Chernov's bound (and using the independence of the different terms in the sum), 
we get
\begin{multline}
\label{Chernov}
\bbP\bigg( \sum_{ \|x\| \leq K\sqrt{N}}   N^{-\frac{d}{2\ga}} \widebar \go_x^{\gep} f_N(x)  \ind_{\{ \widebar \go_x^{\gep} f_N(x)  \leq M_{\gep} N^{d/2\ga}  \}}  \geq M_{\gep}^{-\frac{1}{2} (\ga-1) } \bigg) \\
 \leq \exp\Big( - M_{\gep}^{\frac12 (\ga-1) } \Big)  \prod_{\|x\| \leq K \sqrt{N}} \bbE\Big[ \exp\Big(  M_{\gep}^{ \ga-1}    N^{-\frac{d}{2\ga}} \widebar \go_x^{\gep} f_N(x) \ind_{\{ \widebar \go_x^{\gep} f_N(x)  \leq  M_{\gep} N^{d/2\ga}  \}} \Big) \Big] \, .
\end{multline}
We now estimate the exponential moment.
Let $ \zeta \in (1,2)$ be such that $\ga<\zeta < \frac{d}{d-2}$.
We write that 
\begin{multline*}
 \bbE\Big[ \exp\Big(  M_{\gep}^{ \ga-1}    N^{-\frac{d}{2\ga}} \widebar \go_x^{\gep} f_N(x) \ind_{\{ \widebar \go_x^{\gep} f_N(x)  \leq  M_{\gep} N^{d/2\ga}  \}} \Big) \Big]\\
 \leq 1 +  M_{\gep}^{\ga-1} N^{-\frac{d}{2\ga}} f_N(x) \bbE\big[  \widebar \go_x^{\gep} \ind_{\{ \widebar \go_x^{\gep} f_N(x)  \leq  M_{\gep} N^{d/2\ga}  \}}   \big] + C M_{\gep}^{\zeta(\ga-1)}  N^{-\frac{d}{2\ga} \zeta} \bbE \big[  |\widebar \go_x^{\gep}|^{\zeta}  \big] f_N(x)^{\zeta} \, ,\\
 \leq 1  + C' M_{\gep}^{2(\ga-1)}  \gep^{\zeta-\ga}  N^{-\frac{d}{2}} f_N(x)^{\zeta} \,,
\end{multline*}
 where the first term is negative for large enough $N$, since $\bbE[\widebar \go_x^{\gep}]=0$, $\bbE[\tilde \go_x^{\gep}]\to0$ as $N\to\infty$ by $\ga>1$, and $N^{d/2\ga}f_N(x)^{-1}\to\infty$ as $N\to\infty$, and
that $\bbE \big[  |\widebar \go_x^{\gep}|^{\zeta}  \big]  \leq C \gep^{\zeta-\ga} N^{\frac{d}{2\ga} (\zeta-\ga)}$ for the second term. 
Going back to~\eqref{Chernov}, and using that $1+z\leq e^{z}$ for any $z\in \bbR$, we get that the probability in the left-hand side of~\eqref{Chernov} is bounded by
\begin{multline*}
 \exp\bigg( - M_{\gep}^{\frac12 (\ga-1)}  + C' M_{\gep}^{2 (\ga-1)} \gep^{\zeta-\ga}  N^{-\frac{d}{2}} \sum_{\|x\|\leq K\sqrt{N}}  f_N(x)^{\zeta}  \bigg) 
 \leq  c  \exp\Big(- M_{\gep}^{\frac12 (\ga-1)}  +C'' M_{\gep}^{2 (\ga-1)}  \gep^{\zeta-\ga} \Big) \, ,
\end{multline*}
where we used again a sum integral comparison together 
with the fact that $\int_{\bbR^d} f(x)^{a} \dd x <+\infty$ for any $a<d/(d-2)$.
Therefore, choosing $M_{\gep} = \gep^{-(\zeta-\ga)/(\ga-1)}$,
the second term is a constant and we get that~\eqref{Chernov}
is bounded by $\exp( - M_{\gep}^{(\ga-1)/2})$,  which concludes the proof of~\eqref{whattoprove}.
\end{proof}

\subsection{Transversal fluctuations} The proof for transversal fluctuations are of order $\sqrt{N}$ is a mimic of Subsection \ref{fluct4} and we omit it.

\appendix
\section{Entropy-controlled Last-Passage Percolation (E-LPP)}\label{sec:ELPP}

Recall that $\mathbf{\Lambda}_r$ is the ball of radius $r$
in $\R^d$, and that $\Lambda_r = \mathbf{\Lambda}_r \cap \bbZ^d$.
For any $m\in \bbN$, 
let $\mathbf{\gU}_m^{(r)}$
denote a set of~$m$ independent random variables $(Z_i)_{i=1}^m$
uniform in $\mathbf{\Lambda}_r$,
and let $\gU_m^{(r)}$
denote a set of~$m$ \emph{distinct} sites taken uniformly in $\Lambda_r$.
We define the Entropy-controlled Last-Passage  Percolation
(E-LPP) as  the maximal cardinality
of a subset $\Delta$ of $\mathbf{\Upsilon}_{m}^{(r)}$ or $\Upsilon_{m}^{(r)}$ with entropy smaller than $B$ (recall the definition~\ref{eq:entropydef}).

More precisely, we define the continuous and discrete E-LPP as follows
\begin{equation}\label{eq:LPP}
\begin{split}
\sL_{m}^{(B)}(r)&:=\max\{|\Delta|: \Delta \subset \mathbf{\Upsilon}_m^{(r)}, \ent(\Delta)\leq B\},\\
L_{m}^{(B)}(r)&:=\max\{|\Delta|: \Delta \subset {\Upsilon}_m^{(r)}, \ent(\Delta)\leq B\}\, .
\end{split}
\end{equation}
This continuous non-directed E-LPP has been considered in~\cite[Sec.~3]{BT19c}, in dimension $d=2$.
The generalization of \cite[Thm.~3.1]{BT19c}
to the discrete E-LPP and to higher dimensions is straightforward
(we give a brief proof below for the sake of completeness).

\begin{theorem}\label{thm1}
There exists  a constant $c_d>0$ such that, for any $B>0$ we have that 
for any $k\geq 1$
\begin{align}
\label{eq:LPPcont}
\bbP\Big(\sL_m^{(B)}(r) > k\Big)  \le \Big(   \frac{c_d \,B^{1/2} m^{1/d}}{r k}\Big)^{d k}\,, \qquad 
\bbP\Big(L_m^{(B)}(r) > k\Big)  \le \Big(  \frac{ c_d \, B^{1/2 } m^{1/d}}{r k}\Big)^{d k}\,.
\end{align}
\end{theorem}

We have the following 
corollary, whose proof is also given below.
\begin{corollary}\label{cor:moment}
For any $b\le d$, there exists a constant $c_{b,d}$, only depending on $b$ and the dimension~$d$, such that for any $B>0$, any $m\geq1$ and any $r$,
\begin{equation}
\bbE\left[\left(\frac{\sL_{m}^{(B)}(r)}{(B^{1/2}m^{1/d}/r)\wedge m}\right)^{b}\right]\leq c_{b,d},\quad\bbE\left[\left(\frac{L_{m}^{(B)}(r)}{(B^{1/2}m^{1/d}/r)\wedge m}\right)^{b}\right]\leq c_{b,d}.
\end{equation}
\end{corollary}

\begin{proof}[Proof of Theorem~\ref{thm1}]
The proof relies on a first moment estimate.
We prove the bound only in the continuous setting, 
the computations being the same in the discrete setting,
using sum integral comparisons.

For any $B>0$ we let
$E_k(B)=\big\{\Delta=(x_i)_{i=1}^k \colon \ent(\Delta)\le B \big\}$ and we denote $\mathcal N_k$ the number of sets $\Delta \subset \mathbf{\gU}_m^{(r)}$ with $|\Delta|=k$, that have entropy at most~$B$. Markov's inequality gives that
\[\bbP\big( \sL_m^{(B)}(r)\geq k \big) = \bbP(\mathcal N_k \geq 1) \leq \bbE[\mathcal N_k] \, .\]
By exchangeability of the points in $\mathbf{\gU}_m^{(r)}$, we get  that 
\begin{equation}
\label{Nfirstmoment}
\bbE[\mathcal N_k] = \binom{m}{k} \bbP\Big( \exists \ \sigma\in \mathfrak{S}_k\ s.t. \  (Z_{\sigma(1)},\ldots,  Z_{\sigma(k)}) \in E_k(B) \Big) =  \frac{m!}{(m-k)!} \, \frac{\mathrm{Vol}(E_k(B))}{ (c_dr^d)^{k}}\, ,
\end{equation}
where we used that 
$ Z_1, \ldots , Z_k$ are independent uniform random variables in the domain $\boldsymbol \Lambda_{r}$, which has volume $c_d r^d$.
%
Then, we can compute the volume of $E_k(B)$  by iteration,
and we find that
\[
\text{Vol}(E_k(B))=
\Big( \frac{\pi^{d/2 }}{\Gamma(\frac{d}{2}+1)}\Big)^k \frac{\Gamma(d)^k}{\Gamma(dk+1)}B^{dk /2} \leq \Big(\frac{C_d B^{d/2}}{k^d}\Big)^k \,.
\]
Going back to~\eqref{Nfirstmoment}, and using that $\frac{m!}{(m-k)!}  \leq m^k$, this concludes the proof.
%
%
\end{proof}

\begin{proof}[Proof of Corollary \ref{cor:moment}]
We only prove the continuous case. The discrete case follows by the same argument. Note that if $m\leq B^{1/2}m^{1/d}/r$, the result is trivial. Hence, we assume that $m>B^{1/2}m^{1/d}/r$.
Choose $c_{b}$ such that $c_b \geq c_dk^{(b-1)/k}$ for all $k$,
with $c_d$ the constant in~\eqref{eq:LPPcont}.

$\ast$ If $c_{b}B^{d/2}m/r^{d}\leq1/2$, then
\begin{equation}\label{eq:Lb1}
\begin{split}
\bbE\left[\left(\sL_{m}^{(B)}(r)\right)^{b}\right]&\leq\sum\limits_{k=1}^{\infty}Ck^{b-1}\bbP\left(\sL_{m}^{(B)}(r)\geq k\right)\leq\sum\limits_{k=1}^{\infty}Ck^{b-1}\Big( \frac{c_d B^{d/2} m}{r^{d}k^{d}}\Big)^{k}\\
& \leq C\sum\limits_{k=1}^{\infty}\Big(c_{b}\frac{B^{d/2}m}{r^{d}k^{d}}\Big)^{k}\leq2Cc_{b}\frac{B^{d/2}m}{r^{d}}\leq2C\Big(c_{b}\frac{B^{d/2}m}{r^{d}}\Big)^{b/d}.
\end{split}
\end{equation}

$\ast$ If $c_{b}B^{d/2}m/r^{d}\geq1/2$, note that $c_{b}u^{d/b}>c\vee1$ for large $u$, then
\begin{equation}\label{eq:Lb2}
\bbE\bigg[\bigg(\frac{\sL_{m}^{(B)}(r)}{B^{1/2}m^{1/d}/r}\bigg)^{b}\bigg]  =\int_{0}^{\infty}\bbP\bigg(\frac{\sL_{m}^{(B)}(r)}{B^{1/2}m^{1/d}/r}\geq u^{1/b}\bigg)du \leq C+\int_{C}^{\infty}\left(c_d u^{-1/b}\right)^{d u^{1/b} B^{1/2} m^{1/d} /r} \dd u \, .
\end{equation}
Since $d u^{1/b} B^{1/2} m^{1/d} /r \geq c'_d u^{1/b}$, we get that the last integral is bounded by a constant.
\end{proof}

\subsection{Applications of E-LPP}

We now apply Theorem~\ref{thm1}
to prove Lemma~\ref{lemma1} and Proposition~\ref{tail:disT}.

\subsubsection*{Proof of Lemma~\ref{lemma1}.}

Let us start by recalling Remark \ref{rem1}: if $s\in \cD$ is such that $\ent(s)\le 2$, then we have $s\in \cD_{2}$.
For any $m\in \N$, we denote $\mathbf{\gU}_m^{(2)} := \{ \bY_i^{(2)} , 1\leq i \leq m\}$,
where we recall that $(\bY_i^{(2)})_{i\geq 1}$ 
are the positions of the order statistics of $(w,x) \in \cP_2 = \{ (w,x) \in \cP , \|x\|\leq 2\}$ and are i.i.d.\ uniform 
random variables in $\mathbf{\Lambda}_2$. 
Recalling that the ordered statistics of weights $(\bM_i^{(2)})_{i\geq 1}$ is a.s.\ decreasing, using \eqref{eq:energy1} we have that
\begin{equation}
\label{eq:decomppi}
\pi_{2}(s)=\sum\limits_{(w,x)\in\cP_2} w\mathbbm{1}_{\{x\in s([0,1])\}}
=\sum_{j=0}^\infty \sum_{i=2^j}^{2^{j+1}} \bM_{i}^{(2)} \mathbbm{1}_{\{\bY_i^{(2)}\in s([0,1])\}} \le 
\sum_{j=0}^\infty \bM_{2^j}^{(2)}\sL_{2^{j+1}}^{(2)}(2)\, ,
\end{equation}
where $\sL_m^{(2)}(2)$ is the continuous E-LPP defined
above in~\eqref{eq:LPP}. 
Let $\delta>0$ be such that $\frac1\ga-\frac1d>\delta$. A union bound gives that for any $t>0$,
\begin{equation}\label{pr:unionbound}
\bbP \Big(\sup\limits_{s: \ent(s)<2 }\{\pi_{2}(s)\} > t\Big) \le 
\sum_{j=0}^\infty \bbP \Big(\bM_{2^j}^{(2)}\sL_{2^{j+1}}^{(2)}(2) >  C  t 2^{j(\frac{1}{d}-\frac{1}{\ga}+\delta)}\Big).
\end{equation}
where $1/C=\sum_{j\in \N} 2^{j(\frac{1}{d}-\frac{1}{\ga}+\delta)}$.
Then, \eqref{pr:unionbound} is smaller than 
\begin{equation}\label{eq:sumlemma1}
 \sum_{j=0}^\infty\Big[   \bbP\Big(\sL_{2^{j+1}}^{(2)}(2) > C_0\log_2(2+t)2^{(j+1)/d}\Big) + \bbP\Big(\bM_{2^j}^{(2)} > C_1 \frac{t}{\log_2(2+t)} 2^{j(-\frac{1}{\ga}+\delta)}\Big)
\Big]\,,
\end{equation}
where $C_0$ is a constant chosen (large enough) below, and $C_1$ depends on $C$ and $C_0$. We use \eqref{eq:LPPcont} with $m=2^{j+1}$ and $k=C_0\log_2 t2^{(j+1)/d }$. 
Then if $C_0$ is taken large enough, we have 
\[
\bbP\Big(\sL_{2^{j+1}}^{(2)}(2) > C_0\log_2(2+t)2^{(j+1)/d}\Big)\le
\Big(\frac{1}{2}\Big)^{dC_0 \log_2(2+t) 2^{(j+1)/d } }\,.
\]
For the second term in \eqref{eq:sumlemma1}, we recall that $\bM_{i}^{(2)}\overset{(d)}{=}(2\sqrt 2)^{d/\ga}\textrm{Gamma}(i)^{-1/\ga}$, so that $\bbE\big[ (\bM_{i}^{(2)})^\ga \big]\leq c i^{-1}$, where $c$ is a constant independent of $i$. Therefore, using Markov's inequality we get that
\[
\bbP\Big(\bM_{2^j}^{(2)} > C_1 \frac{t}{\log_2(2+t)}2^{j(-\frac{1}{\ga}+\delta)}\Big)\le c \log_2(2+t)^{\ga} t^{-\ga} 2^{-j\delta \ga}\,.
\]
All together, we get that
\begin{align*}
\bbP \Big(\sup\limits_{s: \ent(s)<2 }\{\pi_{2}(s)\} > t\Big) 
&\leq \sum_{j=0}^{\infty} \bigg(  c \log_2(2+t)^\ga t^{-\ga} 2^{- j \gd a} +2^{ - dC_0(\log_2(2+t))2^{(j+1)/d } } \bigg)\\
& \leq c' \log_2(2+t)^\ga t^{-\ga}  + c' t^{- c'' C_0}
\leq  c'''\log_2(2+t)^\ga t^{- \ga} \, ,
\end{align*}
where the constant $c''$ depends only on $d$, and $C_0$
has been fixed large enough so that $c'' C_0 >\ga$.
This concludes the proof, by replacing  the exponent $\ga$
by some $a<\ga$ and getting rid of the term $\log_2(2+t)^{\ga}$ by adjusting the constant.
\qed

\subsubsection*{Proof of Proposition \ref{tail:disT}.}\label{sec:varprobDisc}
To prove Proposition \ref{tail:disT}, we start by bounding the tail of $\Omega_{r}^{(\ell)}(\Delta)$, $\Omega_{r}^{(>\ell)}(\Delta)$ in Lemmas \ref{lem:tail<ell} and \ref{lem:tail>ell} respectively.

\begin{lemma}\label{lem:tail<ell}
Let $\ga \in (0,d)$. There exists a constant $c$, such that for any $B>0$, $r>0$ and $t>1$,
\begin{equation*}
\bbP\bigg(\sup\limits_{\ent(\Delta)\leq B}\Omega_{r}^{(\ell)}(\Delta)\geq t\times r^{\frac{d}{\ga} -1} \, B^{1/2}\bigg)\leq ct^{- \alpha d / (\alpha+d)} \, .
\end{equation*}
\end{lemma}
\begin{proof}
The proof is the discrete version of Lemma~\ref{lemma1} proven above: using the discrete E-LPP defined in \eqref{eq:LPP},
we can write
\begin{equation}\label{Decomp<ell}
\sup\limits_{\ent(\Delta)\leq B}\Omega_{r}^{(\ell)}(\Delta)\leq\sum\limits_{j=0}^{\log_{2}\ell}\sum\limits_{i=2^{j}}^{2^{j+1}-1}M_{i}^{(r)}\mathbbm{1}_{\{Y_{i}^{(r)}\in\Delta\}}\leq\sum\limits_{j=0}^{\log_{2}\ell}M_{2^{j}}^{(r)}L_{2^{j+1}}^{(B)}(r)\, .
\end{equation}

The proof then follows the same lines as above, but we keep some details to help keep track of the different parameters.
Let $\delta>0$ be such that $1/\alpha-1/d>2\delta$ and let $1/C=\sum_{j=0}^{\infty}2^{j(1/d-1/\alpha+2\delta)}$.
For $\gep \in [0,1]$, by two consecutive union bounds we get 
\begin{equation}
\label{equationforOmega}
\begin{split}
\bbP\bigg(\sup\limits_{\ent(\Delta)\leq B}\Omega_{r}^{(\ell)}& (\Delta)  \geq t \times r^{\frac{d}{\ga} -1}  B^{1/2}\bigg)\leq\sum\limits_{j=0}^{\log_{2}\ell}\bbP\left(\frac{M_{2^{j}}^{(r)}}{r^{d/\ga}}\frac{L_{2^{j+1}}^{(B)}(r)}{B^{1/2}/r}\geq 
 C t \, 2^{j(\frac{1}{d}-\frac{1}{\alpha}+2\delta)}\right)  \\
 & \leq \sum\limits_{j=0}^{\log_{2}\ell} \bbP\bigg(\frac{M_{2^{j}}^{(r)}}{r^{d/\ga}} >C_{1} t^{1-\epsilon} 2^{-\frac{j}{\alpha}+j\delta}\bigg)+\sum\limits_{j=0}^{\log_{2}\ell}  \bbP\bigg(\frac{L_{2^{j+1}}^{(B)}(r)}{B^{\frac{1}{2}}/r}>C_{2}t^{\epsilon}(2^{j+1})^{\frac{1}{d}+\delta}\bigg). 
 \end{split}
\end{equation}
For the first term above, we can use Lemma~5.1 in \cite{BT19b}
to get that tor any $u>0$ and $\ell\leq|\Lambda_{r}|$, we have
\begin{equation}
\label{tail:orderstat}
\bbP\left(  M_{\ell}^{(r)} >u\,  r^{d/\ga} \ell^{-1/\ga} \right) < (cu)^{-\alpha \ell}.
\end{equation}
For the second term, we apply Makov's inequality
together with Corollary~\ref{cor:moment} with $b=d$.
All together, we get that the left-hand side of~\eqref{equationforOmega} is bounded by
\begin{equation}
\label{MnL:O<ell}
\sum_{j=0}^{\infty}  ( c t^{1-\gep} )^{-\ga 2^j}  2^{-j 2^{j} \gd \ga } +\sum_{j=0}^{\infty} c'  ( C_{2}2^{(j+1)\delta}t^{\epsilon})^{-d}  \leq C' t^{(1-\gep) \ga} + C' t^{- \gep d} \, .
\end{equation}
Finally, we choose $\epsilon=\alpha/(\alpha+d)$ and the proof is completed.
\end{proof}

\begin{lemma}\label{lem:tail>ell}
For any $\ell \le |\Lambda_r|$, there exists a constant $c$, such that for any $B>0$, $r>0$ and $t>1$,
\begin{equation*}
\bbP\left(\sup\limits_{\ent(\Delta)\leq B}\Omega_{r}^{(>\ell)}\geq t \times  r^{\frac{d}{\ga}-1} \, \ell^{ \frac{1}{d} -\frac{1}{\ga} } \, B^{1/2}\right)\leq ct^{-\frac{\alpha\ell d}{\alpha\ell+d} }.
\end{equation*}
\end{lemma}

\begin{proof}
The proof is similar to that of the previous lemma.
Similarly to \eqref{Decomp<ell}, we can write
\begin{equation*}
\sup\limits_{\ent(\Delta)\leq B}\Omega_{r}^{(>\ell)}(\Delta)\leq\sum\limits_{j=0}^{\log_{2}(|\Lambda_{r}|/\ell)}M_{2^{j}\ell}^{(r)}L_{2^{j+1}\ell}^{(B)}(r).
\end{equation*}
Using~\eqref{tail:orderstat}, we have
\begin{equation}
\bbP\left(\frac{M_{2^{j}\ell}^{r}}{r^{d/\ga} \ell^{1/\ga}}>C_{1} t^{1-\epsilon}(2^{j})^{-\frac{1}{\alpha}+\delta}\right)\leq (c t^{1-\gep})^{-2^j \ga \ell}2^{-j2^{j}\alpha\delta\ell}
\end{equation}
and by Markov's inequality and Corollary~\ref{cor:moment}
\begin{equation}
\bbP\left(\frac{L_{2^{j+1}\ell}^{(B)}(r)}{\ell^{\frac{1}{d}}B^{\frac{1}{2}}/r}>C_{2}t^{\epsilon}(2^{j+1})^{\frac{1}{d}+\delta}\right)\leq (C_{2}2^{(j+1)\delta})^{-d}t^{-\epsilon d}.
\end{equation}
The result follows by a similar bound to \eqref{MnL:O<ell}.
\end{proof}

\begin{proof}[Proof of Proposition \ref{tail:disT}]
Let us introduce
\begin{equation}
\label{decom:disT}
\cT_{N,r}^{\beta,(\ell)}([a,b)):=\max\limits_{\Delta\subset\Lambda_{r},\ent(\Delta)\in[a,b)}\left\{\beta\Omega_{r}^{(\ell)}(\Delta)-\ent_N(\Delta)\right\}.
\end{equation}

For any $b>0$, we can then decompose the variational problem as
\begin{align*}
\cT_{N,r}^{\beta,(\ell)}&= \max\Big\{ \cT_{N,r}^{\beta,(\ell)}([0,b)) ,  \,  \sup_{k\geq 0} \big\{ \cT_{N,r}^{\beta,(\ell)}([2^{k}b,2^{k+1}b)) \big\} \Big\}\\
&\leq \max\limits_{\Delta\subset\Lambda_{r},\ent(\Delta)\leq b}\beta\Omega_{r}^{(\ell)}(\Delta)  \vee \sup\limits_{k\geq0}\left\{\beta\max\limits_{\Delta\subset\Lambda_{r},\ent(\Delta)
\leq2^{k+1} b} \gO_r^{(\ell)} - \frac{2^k b}{N} \right\}
\end{align*}

Choosing $b = t (N\beta r^{\frac{d}{\ga}-1})^{2}$ 
and applying a union bound, we get
\begin{align*}
\bbP\left(\cT_{N,r}^{\beta,(\ell)}\geq tN \times(\beta r^{\frac{d}{\ga}-1})^{2}\right) 
\leq  \sum_{k=0}^{\infty} \bbP\Big(  \max\limits_{\Delta\subset\Lambda_{r},\ent(\Delta)
\leq 2^{k} t (N\beta r^{\frac{d}{\ga}-1})^{2} } \gO_r^{(\ell)}  \geq 2^{k-1}   tN \gb (r^{\frac{d}{\ga}-1})^{2}  \Big) \, 
\end{align*}
%
Therefore, applying Lemma~\ref{lem:tail<ell} (with $B=  2^{k} t (N\beta r^{\frac{d}{\ga}-1})^{2}$ and $t'= 2^{k/2-1} t^{1/2}$),
we obtain that
\begin{equation}\label{finish:T<ell}
\bbP\left(\cT_{N,r}^{\beta,(\ell)}\geq tN\times(\beta r^{\frac{d}{\ga}-1})^{2} \right)\leq c \sum_{k=0}^{+\infty} t^{-\frac{\alpha d}{2(\alpha+d)}} 2^{-\frac{\alpha d}{2(\alpha+d)}k}\leq c'\, t^{-\frac{\alpha d}{2(\alpha+d)}} \, .
\end{equation}

For \eqref{tail:T>ell}  the proof follows the same lines as \eqref{decom:disT}-\eqref{finish:T<ell}, by applying Lemma \ref{lem:tail>ell} instead of Lemma~\ref{lem:tail<ell}.
\end{proof}

 Let us now prove another technical result: we show that the paths that maximize $\hat\cT_\gb$ are concentrated around the argmax of the variational problem, which is needed to get~\eqref{notscaleN} ---~this is the analogous of Lemma 4.1 in \cite{AL11}. 
\begin{lemma}\label{lemma:techHatTbeta}
On the event $\{\hat\cT_\gb=0\}$,  for any $\eta \in(0,1)$ we have that 
\begin{equation}
\label{defhatTeta}
\hat \cT_{\gb,\geq \eta} :=    \sup_{ s\in \cD_1 , \sup \|s\| \geq \eta , \hatent(s) <+\infty} 
\big\{ \pi(s) -  \tfrac{1}{\gb} \hatent(s) \big\} <0,\qquad a.s.
\end{equation}
\end{lemma}
\begin{proof}
Recall the definition~\eqref{def:Dq} of $\cD_1$ and~\eqref{eq:energy2} of $\pi_1^{(\ell)}$ (recall that $\hatent(s)=+\infty$ if $\sup_{[0,1]}\|s\|>1$).
For $\ell \geq 1$, we define 
\[
\rho^{(\ell)}:=\sup_{s \in \cD_1 }\big| \pi(s)-\pi^{(\ell)}(s) \big| \,,
\qquad 
\hat \cT_{\gb,\geq \eta}^{(\ell)} :=    \sup_{ s\in \cD_1 , \sup \|s\| \geq \eta , \hatent(s) <+\infty} 
\big\{ \pi_1^{(\ell)}(s) -  \tfrac{1}{\gb} \hatent(s) \big\} \, .
\]
For all $\ell$ we have that
$\hat\cT_{\gb, \geq \eta}\le \hat\cT_{\gb, \geq \eta}^{(\ell)}+\rho^{(\ell)}$, so
the proof is a consequence of the following: 
\begin{equation}
\label{sumrest}
\lim_{\ell\to+\infty}\rho^{(\ell)}=0, \quad a.s., \qquad 
\limsup_{\ell \to +\infty }  \hat\cT_{\gb, \geq \eta}^{(\ell)}<0\qquad a.s.
\end{equation}

\smallskip
To prove the first part of \eqref{sumrest},
reasoning as in the proof of (4.13) in \cite{BT19b}, an integration by part gives that
 \begin{align}
  \pi(s)-\pi^{(\ell)}(s) 
  &\le  \sum_{i=\ell+1}^{\infty} \sL_i   (\bM_{i} -\bM_{i+1} )+\limsup_{n\to\infty} \bM_{n}  \sL_n .
 \end{align}
where $\bM_i:=\bM_i^{(1)}$ and $\sL_n:=\sL_n^{(d/2)}(1)$ are defined in \eqref{eq:energy3} and \eqref{eq:LPP} respectively; note that the restriction $\hatent(s) <+\infty$ implies that $s$ has length at most $1$, so $\ent(s)\leq d/2$. The law of large numbers gives  $ \lim_{n\to\infty} n^{1/\ga} \bM_{n}  = c_1$ a.s., and Theorem~\ref{thm1} gives $\limsup_{n\to\infty} n^{-1/d} \sL_n  <+\infty$ a.s. Since $\ga<d$, we therefore conclude that $ \limsup_{n\to\infty} \bM_{n}  \sL_n = 0$ a.s.

To complete the proof of the first part of~\eqref{sumrest}, we let $U_\ell := \sum_{i >\ell} \sL_i   (\bM_{i} -\bM_{i+1} )$ and we show that
$U_{\ell} <+\infty$ a.s., by showing that
$\bbE [U_{\ell}^2]$ is finite. For any $\gep>0$, by Cauchy-Schwarz inequality we have that
\begin{align*}
U_{\ell} \le \Big(\sum_{i> \ell} \big(i^{-\frac12 -\gep} \big)^2 \Big)^{1/2} \Big( \sum_{i>\ell} \big( i^{\frac12 +\gep}  \sL_i   (\bM_{i} -\bM_{i+1} )^2\Big)^{1/2}\, .
\end{align*}
Then, we get that for $\ell$ large enough
\begin{align*}
\bbE [U_{\ell}^2] &\le C \sum_{i > \ell} i^{1+2\gep} \bbE\big[ (\sL_i )^2\big] \bbE \big[  (\bM_{i} -\bM_{i+1} )^2 \big]\le C' \sum_{i >\ell} i^{2\gep+2/d-2/\ga-1} <+\infty \, ,
\end{align*}
where we used Corollary \ref{cor:moment} and the fact that $\bbE \big[ (\bM_{i} -\bM_{i+1} )^2 \big] \le c i^{-2-2/\ga}$ (see for instance Equation  (7.2)  in \cite{HM07}). Provided $\gep$ is small enough so that $\gep +1/d- 1/\ga <0$, we obtain that $\bbE [ U_{\ell}^2]<\infty$ so $U_{\ell} <\infty$ a.s. We therefore get that $\lim_{\ell\to\infty}U_{\ell}=0$, which proves the first part of \eqref{sumrest}.

\smallskip
We prove the second part of \eqref{sumrest} by contradiction.  Note that by \eqref{def:entropyCont2}-\eqref{phi}, if $\hatent(s)<+\infty$, there exists some parametrization $\varphi\in\Phi$ such that $(s\circ\varphi)(t)$ is 1-Lipschitz on $[0,1]$ and $\pi(s)=\pi(s\circ\varphi)$ for all $\varphi\in\Phi$. Hence, it is enough to consider only 1-Lipschitz paths $\{s\in\cD_1: \|s'\|\leq1\}$.
Let us suppose that there exists a sequence $\tilde s^{(\ell)}$ such that $\sup\|\tilde s^{(\ell)}\|\geq \eta$ and $\limsup \cZ_\gb^{(\ell)}(\tilde s^{(\ell)}) \ge 0$, where we set $\cZ_\gb^{(\ell)}(s) :=\pi^{(\ell)}(s)-\frac1\gb \hatent(s)$ (and $\cZ_{\gb}(s)$ when $\ell=+\infty$). 

We observe that $s\mapsto \cZ_\gb^{(\ell)}(s)$ is upper semi-continuous because $\hatent(s)$ is lower semi-continuous (it is a rate function of a large deviation principle) and $s\mapsto \pi^{(\ell)}(s)$ is upper semi-continuous by construction (we refer to Lemma 4.4 in \cite{BT19b} for more details).
Since the set of $1$-Lipshitz functions on~$[0,1]$ is compact for the uniform norm by Ascoli-Arzel\`{a} theorem, we can suppose that $\tilde s^{(\ell)}$ converges uniformly to $\tilde s$ as $\ell \to +\infty$. Then, using the upper semi-continuity we get
\begin{align*}
\hat \cT_{\gb} \geq  \cZ_\gb(\tilde s)\ge \limsup_{\ell\to +\infty} \cZ_\gb(\tilde s^{(\ell)})\ge \limsup_{\ell\to +\infty} \cZ_\gb^{(\ell)}(\tilde s^{(\ell)})\ge 0 \, .
\end{align*}
Using that when $\hat{\cT}_\gb=0$ the maximiser is unique and is given by $\hat s_\gb \equiv 0$, we conclude that $\tilde s=0$,
which contradicts the fact that $\sup\|\tilde s\|\geq \eta$.
\end{proof}

\begin{remark}\rm
In general,
we can define
\[
\hat \cT_{\gb, \|s-\hat s_{\gb}\|\geq \eta} :=    \sup_{ s\in \cD,\, \sup\|s-\hat s_\gb\|\geq \eta , \, \hatent(s) <+\infty} 
\big\{ \pi(s) -  \tfrac{1}{\gb} \hatent(s) \big\} \,,
\]
where $\hat s_\gb$ is the maximiser of $\hat\cT_\gb$; and similarly for $\cT_{\gb,\|s-s_{\gb}\|\geq \eta}$ with $s_{\gb}$ the maximizer of~$\cT_{\gb}$.
This requires the existence and uniqueness of the maximisers $\hat s_{\gb}, s_{\gb}$, which follows similarly to Section~4.6 in \cite{BT19b}.
Then, one has the analogous of Lemma~\ref{lemma:techHatTbeta}, that is $\hat \cT_{\gb, \|s-\hat s_{\gb}\|\geq \eta} <\hat{\cT}_\beta$ and $\cT_{\gb,\|s-s_{\gb}\|\geq \eta} <\cT_{ \beta}$ a.s.
\end{remark}

\section{Simple random walk estimates}
\label{Appendix}

\noindent
Let us collect
some technical results on $d$-dimensional
simple random walks that we use in the article.

\subsection{Large and moderate deviations for the simple random walk}

\begin{lemma}
\label{lem:rate}
Let $(S_n)_{n\geq 0}$ be a simple symmetric random walk on $\bbZ^d$. 
Let $\xi\in (\frac12,1]$ be fixed.
Then, for any $x\in \R^d$, denote $x_{N}^{(\xi)}$ the element of $\bbZ^d$ with the same parity as $N$ closest to $x N^{\xi}$. Then we have the local large deviation result.
\[
\lim_{N\to\infty} -\frac{1}{N^{2\xi-1}} \log\bP \big( S_{N} = x_N^{(\xi)} \big)
 = \begin{cases}
 \mathrm{J}_d(  x ) & \quad  \text{ if } \xi=1 \, ,\\
 \frac{ d}{ 2} \|x\|^2 &\quad  \text{ if } \xi \in (\frac12,1) \,,
 \end{cases}
\]
where
$\mathrm{J}_d(\cdot)$ is a given rate function, which verifies
$\mathrm{J} (\|x\|_1)   \leq \mathrm{J}_d(x) \leq \mathrm{J} (\|x\|_1) +\log d$, 
where $\mathrm{J}(t) = \frac12 (1+t)\log(1+t) + \frac12 (1-t) \log (1-t)$ is the large deviation rate function for the random walk in dimension $d=1$;
in particular, we have $\mathrm{J}_d(x) <+\infty$ if $\|x\|_1\leq 1$ and $\mathrm{J}_d(x) =+\infty$ if $\|x\|_1 > 1$.
We also have $\mathrm{J}_d(x) \geq \frac{1}{2} \|x\|^2$ for all $x\in \bbR^d$, and $\mathrm{J}_d(x) \sim \frac{d}{2} \|x\|^2$ as $\|x\| \downarrow 0$.
\end{lemma}

\begin{proof}
This is a standard result, that directly derives from local large deviation results for the simple random walk in dimension $d=1$,
by decomposing over the number of steps in each direction.
Let us just treat the case of $\xi=1$, the case $\xi\in (\frac12,1)$ being analogous.
The rate function $\mathrm{J}_d$ can be related to the rate function $\mathrm{J}$ in the following way, decomposing over the 
proportion of the time spent in each direction
\begin{equation}
\label{eq:ratefunctiond}
\mathrm{J}_d(x) = \inftwo{u_1, \ldots, u_d \in [0,1]}{u_1+\cdots+u_d =1}  \sum_{i=1}^d u_i \Big( \mathrm{J}\big( \tfrac{x_i}{u_i} \big)  + \log (d u_i) \Big) \, ,
\end{equation}
where the term $\sum_{i=1}^d u_i \log (du_i)$ comes from the entropic cost of spending a proportion $u_i$ of the time in direction $i$; more precisely, we have $\lim_{N\to\infty} - \frac{1}{N} \log ( \frac{N!}{ (u_1 N)! \cdots (u_d N)!} d^{-N})  = \sum_{i=1}^d u_i \log (du_i)$.

The rate function $\mathrm{J}_d$ does not appear to have a nice 
expression, but some properties can be derived. Note that $\mathrm{J}(\cdot)$ being an even function we can replace
 $\mathrm{J}(x_i/u_i)$ by $\mathrm{J}(|x_i|/u_i)$ in the above expression.
Since $\mathrm{J}(\cdot)$
is a convex function we get that 
\[
\sum_{i=1}^d u_i  \mathrm{J}\big( \tfrac{|x_i|}{u_i} \big) 
\geq \mathrm{J}\Big( \sum_{i=1}^d |x_i|\Big) = \mathrm{J}(\|x\|_1).
\]
Hence, using also that $\sum_{i=1}^{d} u_i \log(du_i) \geq 0$
for all $u_1,\ldots, u_d \in [0,1]$ with $u_1+\cdots+u_d =1$, we get that
$\mathrm{J}_d(x)  \geq \mathrm{J}(\|x\|_1)$.
On the other hand, 
since $\sum_{i=1}^{d} u_i \log(du_i) \leq \log d$,
taking $u_i = |x_i|/\|x\|_1$ we obtain that
$\mathrm{J}_d(x)  \leq \mathrm{J}(\|x\|_1) +\log d$.

To prove the last inequality, let us simply notice that $\mathrm{J}(t) \geq \frac12 t^2$ for all $t\in \bbR$,
so we get that 
$\mathrm{J}_d(x)  \geq  \frac{1}{2} \|x\|_1^2 \geq \frac12 \|x\|^2$.
Also, notice that when $\|x\| \downarrow 0$,
we have that $\mathrm{J}_d(x) \to 0$: for instance, taking $u_i=\frac1d$ in \eqref{eq:ratefunctiond},
we get that $\mathrm{J}_d(x) \leq \frac1d\sum_{i=1}^d \mathrm{J} (d x_i)$.
Since  the infimum in \eqref{eq:ratefunctiond}
goes to $0$, it means that it is attained for some $u_i$
approaching $\frac1d$ (so that $\log (d u_i)$ goes to $0$),
and we therefore get that 
$\mathrm{J}_d(x) \sim \frac1d\sum_{i=1}^d \mathrm{J} (d x_i) $ as $\|x\| \downarrow 0$.
Using that $\mathrm{J}(t) \sim \frac12 t^2$ as $t\downarrow0$,
we get that $\mathrm{J}_d(x) \sim \frac{d}{2} \|x\|^2$.
\end{proof}

\subsection{Probabilities that a given set of points is visited}

Here, for clarity, we separate the cases $\xi \in(\frac12,1)$
and $\xi =1$.
For a set of (ordered) points $\Delta = (y_1,\ldots, y_k) \in \bbZ^d$, with some abuse of notation we write $\Delta \subset \cR_N$
to mean that the points of $\Delta$ are visited in their order by the random walk before time $N$:
put otherwise
\begin{equation}
\label{def:Deltasubset}
\{\Delta \subset \cR_N\} = \big\{\exists~0 < t_{1}<\cdots<t_{k}\leq N, \mbox{s.t.}~S_{t_{1}}=y_1 ,\cdots,S_{t_{k}}=y_k \big\}\, .
\end{equation}
We now state large deviation principles for 
events $\{\Delta \subset \cR_N\} $.

\begin{lemma}\label{AppendixL1}
Let $\xi \in(\frac12 ,1)$.
For any ordered set
$\Delta=(x_{1},\cdots,x_{k})\subset\bbR^{d}$ of distinct points,
we have (omitting integer parts for notational simplicity)
\[
\lim\limits_{N\to\infty} - \frac{1}{N^{2\xi-1}} \log\bP\left( (x_{1}N^{\xi} ,\ldots, x_{k}N^{\xi} ) \subset \cR_N\right)= \ent(\Delta) \, ,
\]
where the entropy $\ent(\Delta)$
is defined in~\eqref{eq:entropydef}.
\end{lemma}

\begin{proof}
We only treat the case $k=2$  to lighten notation; the general case is analogous.

\smallskip
\noindent
\textit{Lower bound}: 
By a local large deviation, see \textit{e.g.} \cite[Theorem 3]{S67},
we get 
\begin{align*}
\bP\Big(\exists~1\leq s<t\leq N,\ & \mbox{s.t.}~S_{s}=x N^{\xi}, S_{t}=y N^{\xi}\Big)
 \geq\max\limits_{1\leq s< N}\bP(S_{s}=x N^{\xi},S_{N}=y N^{\xi} )\\
& =\max\limits_{0<u\leq1}\frac{c_d}{(u(1-u))^{\frac{d}{2}}N^{d}} \, \exp\bigg( -  N^{2\xi-1} \frac{d}{2}\Big(\frac{\|x\|^{2}}{u}+\frac{\|y-x\|^{2}}{1-u} +o(1) \Big) \bigg)
\end{align*}
Now the term $\frac{\|x\|^{2}}{u}+\frac{\|y-x\|^{2}}{1-u}$
is minimal for $u= \frac{\|x\|}{ \|x\|+\|y-x\|}$ and
choosing that specific $u$ gives that 
\begin{equation*}
\bP\Big(\exists~1\leq s<t\leq N,\  \mbox{s.t.}~S_{s}=x N^{\xi}, S_{t}=y N^{\xi}\Big) \geq 
 \frac{ c_{d,x,y}}{ N^d}\,  e^{-N^{2\xi-1}(1+o(1))\ent(\Delta)} \, ,
\end{equation*}
which proves the lower bound. Note that if $\|x\|$ and $\|y-x\|$ are bounded away from $0$ and $\infty$, then so is $c_{d,x,y}$.

\smallskip
\noindent
\textit{Upper bound}: Using again \cite[Theorem 3]{S67}, we have that
\begin{align}
\bP&\left(\exists~1\leq s<t\leq N, \mbox{s.t.}~S_{s}=xN^{\xi}, S_{t}=yN^{\xi}\right)
\leq \sum\limits_{s=1}^{N}\sum\limits_{t=s+1}^{N}\bP(S_{s}=xN^{\xi},S_{t}=yN^{\xi})\notag\\
\label{eq:lemmaA1}
&\qquad \leq \sum\limits_{s=1}^{N}\sum\limits_{t=s+1}^{N}\frac{c_d}{(s(t-s) )^{\frac{d}{2}}}
 \, \exp\bigg( - N^{2\xi-1}\frac{d}{2}\Big(\frac{\|x\|^{2}}{s/N}+\frac{\|y-x\|^{2}}{(t-s)/N}   +o(1)\Big)\bigg)\, ,
\end{align}
where the $o(1)$ is uniform in $N$.
To get an upper bound for \eqref{eq:lemmaA1}, we minimize
the quantity in the exponent, which is
$N^{2\xi-1}\ent(\Delta)$ (as above).
Then for any $d\geq1$, \eqref{eq:lemmaA1} is bounded above by
a constant times
\begin{align*}
&e^{- (1+o(1)) N^{2\xi-1} \ent(\Delta ) }\sum\limits_{s=1}^{N}\sum\limits_{t=s+1}^{N} \frac{1}{(s(t-s))^{d/2}}
\leq  c N\,  e^{- (1+o(1)) N^{2\xi-1} \ent(\Delta ) } \,,
\end{align*}
which gives the upper bound.
\end{proof}

Let us state the analogous 
result in the case $\xi=1$.
We do not prove it, since it is analogous to~Lemma~\ref{AppendixL1}.

\begin{lemma}\label{AppendixL1bis}
For any ordered set
$\Delta=(x_{1},\cdots,x_{k})\subset\bbR^{d}$ of distinct points,
we have (omitting integer parts for notational simplicity)
\[
\lim\limits_{N\to\infty}  -\frac{1}{N} \log\bP\big( (x_{1}N ,\ldots, x_{k}N ) \subset \cR_N \big)=  \hatent(\Delta) \, ,
\]
where the entropy $\hatent(\Delta)$
is defined in~\eqref{eq:hatentropydef}.
\end{lemma}
Notice here that the definition~\eqref{eq:hatentropydef} of the entropy $\hatent(\Delta)$ includes the infimum over the choices of times $0< t_1<\cdots <t_k \leq 1$ at which the random walk visits the points.

To conclude,
we  give a uniform upper bound on the probability
that a given set is visited.

\begin{lemma}\label{AppendixL2}
There are positive constants $C_{1}:=C_{1}(d)$ and $C_{2}:=C_{2}(d)$  such that for any
set  $\Delta=(x_{1},\cdots,x_{\ell})\subset\bbZ^{d}$ of distinct points, we have in dimension $d\geq 2$
\begin{equation*}
\bP(\Delta\subset \cR_N) \leq (C_{1})^{\ell} e^{-\frac{C_{2}}{N}\ent (\Delta)} ,
\end{equation*}
where $\ent_N(\Delta)$ is defined in \eqref{def:entropyN}.
\end{lemma}

\begin{proof}
We write $x_{0}=0$ and $k_{0}=0$ by convention. 
Let us start with the case of dimension $d\geq 3$.
By a union bound (recalling~\eqref{def:Deltasubset}), we have
\[
\bP\big( (x_1,\ldots, x_\ell)  \subset \cR_N \big) \leq \sum\limits_{k_{1}=1}^{N}\cdots\sum\limits_{k_{\ell}= k_{\ell-1}+1}^{N}\prod_{i=1}^{\ell}\bP(S_{k_i}-S_{k_{i-1}}=x_{k}-x_{k-1}).
\]
Now, by standard local large deviation results, there exist
constants $c_1, c_2$ such that  for any $t\geq 1$ and $x\in \mathbb{Z}^d$, 
\begin{equation*}
\bP(S_{t}=x)\leq \frac{c_{1}}{t^{d/2}}e^{-c_2 \|x\|^{2} /t } \, .
\end{equation*}
Thus, since for all $0=u_0<u_1<\cdots <u_\ell \leq 1$ we have
\begin{equation}
\label{eq:boundentrop}
\sum_{i=1}^{\ell} \frac{\|x_i-x_{i-1}\|^2}{u_i-u_{i-1}} \geq  \ent(\Delta) \, , 
\end{equation}
we get that
\begin{align*}
\bP\big( (x_1,\ldots, x_\ell)  \subset \cR_N \big) 
&\leq c_{1}^{\ell} \, e^{-\frac{C_{2}}{N}\ent(\Delta) } \sum_{k_{1}=1}^{N}\cdots\sum\limits_{k_{\ell}=k_{\ell-1}+1}^{N}\prod_{i=1}^{\ell}\frac{1}{(k_{i}-k_{i-1})^{d/2}}
\le (C_{1} )^{\ell} e^{-\frac{C_{2}}{N}\ent (\Delta)},
\end{align*}
which concludes the proof in dimension $d\geq 3$.

For the case of the dimension $d=2$, instead of the union bound, we use that, by Markov'a property
\[
\bP\big( (x_1,\ldots, x_\ell)  \subset \cR_N \big) \leq \sum\limits_{k_{1}=1}^{N}\cdots\sum\limits_{k_{\ell}= k_{\ell-1}+1}^{N}\prod_{i=1}^{\ell}\bP( H_{x_{i}-x_{i-1}} = k_i-k_{i-1}), 
\]
where $H_x = \min\{t \geq 1, S_t =x \}$ is the hitting time of the site $x$. 
Now, from Uchiyama's \cite[Thm.~1.4]{Uch11}, we get that uniformly for $x\in \bbZ^d\setminus \{0\}$ 
\[
\bP(H_x = k) \leq C \frac{\log(1+ \|x\|)}{ k (\log(1+ k))^2} e^{- c \|x\|^2/k}   \ind_{\{ k \geq  c\|x\|\}} \,  .
\]
Using again~\eqref{eq:boundentrop}, we get 
\begin{align*}
\bP\big( (x_1,\ldots, x_\ell)  \subset \cR_N \big) 
 \leq c_{1}^{\ell} e^{-\frac{C_{2}}{N}\ent(\Delta) }  \prod_{i=1}^{\ell} 
\bigg( \sum_{k_i=k_{i-1} + c\|x_i-x_{i-1}\|}^N \frac{ \log(1+ \|x_i-x_{i-1}\|) }{(k_{i}-k_{i-1}) \log (1+ k_i-k_{i-1})^2} \bigg) \, .
\end{align*}
All the sums are bounded, 
so this concludes the proof in the case of the dimension $d=2$.
\end{proof}

\subsection*{Intersection of ranges of independent random walks}

Recall that  $J_N := \sum_{x\in \bbZ^d} \bP(x\in \cR_N)^2 $.

\begin{lemma}
\label{lem:JN}
We have the following asymptotics, as $N\to+\infty$,
\[
J_N  = 
(1+o(1))
\begin{cases}
c_2 \frac{N}{(\log N)^2}  & \text{ if } d=2 \,,\\
 c_3 \sqrt{N} & \text{ if } d=3 \,,\\
\end{cases} 
\qquad 
J_N  = 
(1+o(1))
\begin{cases}
 c_4 \log N  & \text{ if } d=4 \,,\\
 c_d     & \text{ if } d=5 \,.\\
\end{cases} 
\]
\end{lemma}

\begin{proof}
First of all,  we can rewrite $ J_N = \bE^{\otimes 2} [ \cJ_N]$,
with  $\cJ_N := |\cR_N^{(1)} \cap \cR_N^{(2)}|$ the intersection of ranges $\cR_N^{(1)}$ and $\cR_N^{(2)}$ of two independent random walks $S^{(1)}$ and $S^{(2)}$.
In Chen's book~\cite{X15}, the weak convergence of $\cJ_N$
is considered:
in dimension $d=2$, $\frac{(\log N)^2}{N} \cJ_N$ is shown to converge in distribution, see \cite[Thm.~5.3.4]{X15};
in dimension $d=3$, $\frac{1}{\sqrt{N}} \cJ_N$ is shown to converge in distribution, see \cite[Thm.~5.3.4]{X15};
in dimension $d=4$, $\frac{1}{\log N} \cJ_N$ is shown to converge in distribution, see \cite[Thm.~5.5.1]{X15}.
The convergence of the expectation $J_N =\bbE[\cJ_N]$
then comes as a consequence of the uniformity in the integrability of $\cJ_N$, see~\cite[\S~6.2]{X15}.

In dimension $d=5$, $\cJ_N$ converges to $|\cR_{\infty}^{(1)} \cap \cR_{\infty}^{(2)}|$, which is a.s.\ finite.
The convergence of its expectation $J_N$ follows by monotone convergence, the fact that $\bbE[|\cR_{\infty}^{(1)} \cap \cR_{\infty}^{(2)}|]<+\infty$ being given in~\eqref{eq:series1} (with $\ga=2$).
\end{proof}

\section{Technical estimate on the environment}
\label{app:envir}

We collect here some estimates that are needed along the paper.

\begin{lemma}
\label{lem:moments}
For any non-negative integer $p<\ga$ and any sequence $k_N \geq 1$, we have:
\begin{itemize}
\item If $\gb_N k_N \geq 1$,
\begin{equation}
\label{moments-1}
\Big| \bbE \big[\exp(\beta_N (\go-\mu) \ind_{\{\go \leq k_N\}} ) \big] - 1 - \sum_{i=1}^{p} \frac{\gb_N^i}{i!} \bbE[ (\go-\mu)^i] \Big|
\leq C e^{\gb_N k_N} 
 \begin{cases}
 \gb_N^{ \ga} & \text{ if } p+1 >\ga \, ,\\
\gb_N^{\ga} \log k_N  & \text{ if } p+1 =\ga \, ,\\
 \gb_N^{p+1}  & \text{ if } p +1 <\ga \, ;
 \end{cases}
\end{equation}
\item If $\gb_N k_N <1$, in the case $\ga < p+1$, we get 
\begin{equation}
\label{moments-2}
\Big| \bbE \big[\exp(\beta_N (\go-\mu) \ind_{\{\go \leq k_N\}} ) \big] - 1 - \sum_{i=1}^{p} \frac{\gb_N^i}{i!} \bbE[ (\go-\mu)^i] \Big|
\leq 
 C \gb_N k_N^{1-\ga} 
\end{equation}
with an extra factor $\log k_N$ in the upper bound if $\ga=p+1$.
\end{itemize}
\end{lemma}

\begin{proof}
We can use that for any $x \in \bbR$, we have
$|e^x - \sum_{i=0}^{p} \frac{x^i}{i!} | \leq |x|^{p+1} e^{|x|}$.
Hence, we get that
\begin{align*}
\Big| \bbE \big[\exp(\beta_N (\go-\mu) \ind_{\{\go \leq k_N\}} ) \big] - &  1 - \sum_{i=1}^{p} \frac{\gb_N^i}{i!} \bbE\big[ (\go-\mu)^i\big] \Big| \\
&\leq \sum_{i=1}^{p} \frac{\gb_N^i}{i!} \bbE\big[ |\go-\mu|^i \ind_{\{\go >k_N\}} \big]  +
e^{\gb_N k_N}  \gb_N^{p+1} \bbE\big[ |\go-\mu|^{p+1} \ind_{\{\go \leq k_N\}}\big]  \, .
\end{align*}
And since $p<\ga$, we have for any $i \leq p$
$\bbE [ |\go-\mu|^i \ind_{\{\go >k_N\}} ] \leq C k_N^{i-\ga}$.
On the other hand, 
if $\ga >  p+1$, then 
$\bbE[ |\go-\mu|^{p+1} \ind_{\{\go \leq k_N\}}] \leq C$;
if $\ga<p+1$, then
$\bbE[ |\go-\mu|^{p+1} \ind_{\{\go \leq k_N\}}] \leq C k_N^{p+1-\ga}$;
if $\ga =p+1$, then
$\bbE[ |\go-\mu|^{p+1} \ind_{\{\go \leq k_N\}}] \leq C \log k_N$.
We therefore get that
\begin{align*}
&\Big| \bbE \big[\exp(\beta_N (\go-\mu) \ind_{\{\go \leq k_N\}} ) \big] -   1 - \sum_{i=1}^{p} \frac{\gb_N^i}{i!} \bbE\big[ (\go-\mu)^i\big] \Big| \\
&\qquad \leq C \big(e^{\gb_n k_N} -1 \big)k_N^{-\ga} + C e^{\gb_N k_N} (\gb_N k_N)^{p+1}
 \begin{cases}
 k_N^{- \ga} & \text{ if } p+1 >\ga \, ,\\
 k_N^{-\ga} \log k_N  & \text{ if } p+1 =\ga \, ,\\
  k_N^{-(p+1)}  & \text{ if } p +1 <\ga \, .
 \end{cases}
\end{align*}
Then, in the case where $\gb_N k_N \geq  1$,
bounding $k_N^{p+1-\ga}$ by $\gb_N^{\ga-p-1}$ in the case $p+1>\ga$, we get \eqref{moments-1}.
On the other hand, if $\gb_N k_N \leq 1$, considering only the case 
$\ga<p+1$, we get ~\eqref{moments-2}.
\end{proof}

\subsection*{Acknowledgments} The authors would like to thank the referees for their comments and suggestions, which helped improve the quality of the article.

\bibliographystyle{abbrv}
\bibliography{references}

\end{document}